\documentclass[a4paper,twoside,11pt]{article}
\usepackage{color}
\definecolor{grey}{rgb}{0.65,0.65,0.65}
\newcommand{\gr}[1]{\textcolor{grey}{#1}}
\newcommand{\equj}[1]{\gr{#1}}
\newcommand{\difj}[1]{#1}

\usepackage{amsmath}
\usepackage{amsthm}
\usepackage{amssymb}
\usepackage[english,vlined,linesnumbered]{algorithm2e} %%,ruled,nofillcomment
\usepackage{hyperref}
\usepackage{array}
\usepackage{enumerate,enumitem}
\usepackage{multirow}	% pour les tableaux
\usepackage[utf8]{inputenc} 
\usepackage[T1]{fontenc}

\usepackage[title]{appendix}

\newtheorem{theorem}{Theorem}
\newtheorem{lemma}{Lemma}
\newtheorem{proposition}{Proposition}
\newtheorem{corollary}{Corollary}
\newtheorem{definition}{Definition}
\newtheorem{fact}{Fact}

\usepackage[capitalise,nameinlink]{cleveref}[0.19]

\crefname{appsec}{Appendix}{Appendices}
\crefname{fact}{Fact}{Facts}
\crefformat{equation}{\textup{#2(#1)#3}}
\Crefformat{equation}{\textup{#2(#1)#3}}
\crefrangeformat{equation}{\textup{#3(#1)#4--#5(#2)#6}}
\Crefrangeformat{equation}{\textup{#3(#1)#4--#5(#2)#6}}
\crefmultiformat{equation}{\textup{#2(#1)#3}}{ and \textup{#2(#1)#3}}{, \textup{#2(#1)#3}}{, and \textup{#2(#1)#3}}
\Crefmultiformat{equation}{\textup{#2(#1)#3}}{ and \textup{#2(#1)#3}}{, \textup{#2(#1)#3}}{, and \textup{#2(#1)#3}}
\crefrangemultiformat{equation}{\textup{#3(#1)#4--#5(#2)#6}}{ and \textup{#3(#1)#4--#5(#2)#6}}{, \textup{#3(#1)#4--#5(#2)#6}}{, and \textup{#3(#1)#4--#5(#2)#6}}
\Crefrangemultiformat{equation}{\textup{#3(#1)#4--#5(#2)#6}}{ and \textup{#3(#1)#4--#5(#2)#6}}{, \textup{#3(#1)#4--#5(#2)#6}}{, and \textup{#3(#1)#4--#5(#2)#6}}

\newcommand{\card}[1]{\left|#1\right|}
\newcommand{\abs}[1]{\card{#1}}
\newcommand{\set}[1]{\left\{#1\right\}}

\title{Optimizing alphabet reduction pairs of arrays\footnote{Most of the material in this paper appeared in a journal version published in \emph{Mathematics in Computer Science} \cite{CT26-reg}. The remaining material has been accepted for publication in the proceedings of the LATIN 2026 conference \cite{T26}.}}
\author{%
	Jean-François Culus\thanks{jean-francois.culus@st-cyr.terre-net.defense.gouv.fr MEMIAD, UA, Crec Saint-Cyr, France} 
 and Sophie Toulouse\thanks{sophie.toulouse@lipn.univ-paris13.fr LIPN (UMR CNRS 7030), Institut Galil\'ee, Universit\'e Paris 13, France}
}%fin_auteurs
\begin{document}
\maketitle

\begin{abstract}
In \cite{CT18}, we introduced a family of combinatorial designs, which we call {\em alphabet reduction pairs of arrays}, or ARPAs for short. 
These designs are parameterized by three integers $q,p,k$ with $p\leq q$ and $k\leq p$:
$q$ is the size of the alphabet $\set{\sigma_1, \sigma_2 ,\ldots, \sigma_q}$ from which the arrays draw their entries;
$p$ is the maximum number of distinct symbols allowed in a row of the second array;
$k$ is the largest integer for which the two arrays coincide---up to the order of their rows---on any $k$-element subset of their columns. 
The first array must contain at least one occurrence of the word $\sigma_1\ \sigma_2\ \cdots\ \sigma_q$ as a row. 
Intuitively, the goal is to cover the largest possible number of occurrences of this word of $q$ distinct symbols with the smallest possible number of words of at most $p$ different symbols.

ARPAs are related to the approximability of {\em Constraint Satisfaction Problems with bounded constraint arity} ($k$-CSPs). In this context, we are particularly interested in ARPAs in which the frequency of the word $\sigma_1\ \sigma_2\ \cdots\ \sigma_q$ is maximal. Our goal is precisely to study such ARPAs, which we call {\em optimal}.
To this end, we introduce a seemingly simpler family of combinatorial designs called {\em Cover pairs of arrays} (CPAs). The arrays of a CPA take Boolean entries, and must also coincide (up to the order of their rows) on any $k$-element subset of their columns. Intuitively, the idea behind these array pairs is to cover the largest possible number of occurrences of the word consisting of $q$ ones using the smallest possible number of Boolean words of length $q$ and weight at most $p$. 

We show that, when it comes to maximizing the frequency of their target word, either $\sigma_1\ \sigma_2\ \cdots\ \sigma_q$ or $1\ 1\ \cdots\ 1$, ARPAs and CPAs are equivalent. 
As a corollary of our proof, computing the frequency of $\sigma_1\ \sigma_2\ \cdots\ \sigma_q$ in optimal ARPAs reduces to solving a linear program in $q + p + 1$ continuous variables and $k + 1$ constraints. 
In addition, we prove the optimality of the ARPAs given in \cite{CT18} for the case $p=k$ and provide optimal ARPAs for the cases $k=1$ and $k=2$. 

~\\{\noindent}{\bf Keywords:} 
	Alphabet reduction pair of arrays;
	Cover pair of arrays;
	Approximability of $k$-CSPs;
	Combinatorial designs; 
	Differential approximation.
\end{abstract}

% ____________________________ Introduction
% _______ définition
% _______ motivation (CSPs & résultats liés C18)
% _______ outline (& notations)
\section{Introduction}
\begin{figure}[t]{\footnotesize
\caption{$(q, p)$-ARPAs of strength 2 and 3. 
We highlight in gray the rows of the form $0\ 1\ \cdots\ q-1$ in the array $Q$.}
\label{fig-Gamma-ex}
%% /!\/!\ l'ordre des lignes des tableaux ici doit être le même que dans Delta-ex /!\/!\
\begin{center}
\begin{tabular}{c|c}
$\begin{array}{c}
	(Q, P)\in\Gamma(4, 3, 2)\\[4pt]	%% (q, p, k) =(4, 3, 2)
	\begin{array}{ccc}\setlength\arraycolsep{1.5pt}
		\begin{array}{cccc}
			Q^0		&Q^1	&Q^2	&Q^3\\\hline
			\gr{0}	&\gr{1}	&\gr{2}	&\gr{3}\\
			\gr{0}	&\gr{1}	&\gr{2}	&\gr{3}\\
			0		&1		&0		&2\\
			0		&0		&2		&2\\
			3		&1		&2		&2\\
			3		&0		&0		&3\\
		\end{array}&&\setlength\arraycolsep{1.5pt}
		\begin{array}{cccc}
			P^0		&P^1	&P^2	&P^3\\\hline
			0		&1		&2		&2\\
			0		&1		&2		&2\\
			0		&1		&0		&3\\
			0		&0		&2		&3\\
			3		&1		&2		&3\\
			3		&0		&0		&2\\
		\end{array}
	\end{array}\\[-6pt]\\
	R^*(Q, P)/R(Q, P) =2/6 =1/3\\[-4pt]\\\hline\\[-6pt]
	(Q, P)\in\Gamma(5, 3, 2)\\[4pt]	%% (q, p, k) =(5, 3, 2)
	\begin{array}{ccc}\setlength\arraycolsep{1.5pt}	
		\begin{array}{ccccc}
			Q^0		&Q^1	&Q^2	&Q^3	&Q^4	\\\hline
			\gr{0}	&\gr{1}	&\gr{2}	&\gr{3}	&\gr{4}\\
			1		&2		&2		&1		&4	\\
			3		&3		&2		&2		&4	\\
			3		&3		&3		&3		&3	\\
			1		&1		&3		&1		&3	\\
			0		&2		&3		&2		&3	\\
		\end{array}&&\setlength\arraycolsep{1.5pt} 
		\begin{array}{ccccc}
			P^0		&P^1	&P^2	&P^3	&P^4\\\hline
			3		&3		&2		&3		&4	\\
			1		&1		&2		&1		&4	\\
			0		&2		&2		&2		&4	\\
			0		&1		&3		&3		&3	\\
			1		&2		&3		&1		&3	\\
			3		&3		&3		&2		&3	\\
		\end{array}
	\end{array}\\[-6pt]\\
	R^*(Q, P)/R(Q, P) =1/6
\end{array}$
&$\begin{array}{c}
	(Q, P)\in\Gamma(5, 4, 3)\\[4pt]	%% (q, p, k) =(5, 4, 3)
	\begin{array}{ccc}\setlength\arraycolsep{1.5pt}	
		\begin{array}{ccccc}
			Q^0	&Q^1&Q^2&Q^3&Q^4\\\hline
			0	&0	&1	&3	&4	\\
			0	&0	&2	&0	&4	\\
			0	&0	&2	&3	&3	\\
			0	&1	&1	&0	&4	\\
			0	&1	&1	&3	&3	\\
			0	&1	&2	&0	&3	\\
			\gr{0}&\gr{1}&\gr{2}&\gr{3}&\gr{4}\\
			\gr{0}&\gr{1}&\gr{2}&\gr{3}&\gr{4}\\
			\gr{0}&\gr{1}&\gr{2}&\gr{3}&\gr{4}\\
			4	&0	&1	&0	&3	\\
			4	&0	&1	&0	&3	\\
			4	&0	&2	&3	&4	\\
			4	&1	&1	&3	&4	\\
			4	&1	&2	&0	&4	\\
			4	&1	&2	&3	&3	\\
		\end{array}&&\setlength\arraycolsep{1.5pt} 
		\begin{array}{ccccc}
			P^0	&P^1&P^2&P^3&P^4\\\hline
			0	&0	&1	&0	&3\\
			0	&0	&2	&3	&4\\
			0	&0	&2	&3	&4\\
			0	&1	&1	&3	&4\\
			0	&1	&1	&3	&4\\
			0	&1	&2	&0	&4\\
			0	&1	&2	&0	&4\\
			0	&1	&2	&3	&3\\
			0	&1	&2	&3	&3\\
			4	&0	&1	&0	&4\\
			4	&0	&1	&3	&3\\
			4	&0	&2	&0	&3\\
			4	&1	&1	&0	&3\\
			4	&1	&2	&3	&4\\
			4	&1	&2	&3	&4\\
		\end{array}
	\end{array}\\[-6pt]\\
	R^*(Q, P)/R(Q, P) =3/15 =1/5
\end{array}$
\end{tabular}
\end{center}}
\end{figure}

\begin{figure}[t]{\footnotesize
\caption{Verifying that the upper-left array pair in \cref{fig-Gamma-ex} satisfies condition $(k_=)$.}
\label{fig-Gamma-k=}
\begin{center}
\begin{tabular}{c|c|c}
$\begin{array}{ccc}				%% ______________ {0, 1}
	\setlength\arraycolsep{1.5pt}	
	\begin{array}{cc}
		Q^0		&Q^1	\\\hline
		\gr{0}	&\gr{1}	\\
		\gr{0}	&\gr{1}	\\
		0		&1		\\
		0		&0		\\
		3		&1		\\
		3		&0		\\
	\end{array}&&\setlength\arraycolsep{1.5pt} 
	\begin{array}{cc}
		P^0		&P^1	\\\hline
		0		&1		\\
		0		&1		\\
		0		&1		\\
		0		&0		\\
		3		&1		\\
		3		&0		\\
	\end{array}
\end{array}$&$\begin{array}{ccc}	%% ______________ {0, 2}
	\setlength\arraycolsep{1.5pt}	
	\begin{array}{cc}
		Q^0		&Q^2	\\\hline
		\gr{0}	&\gr{2}	\\
		\gr{0}	&\gr{2}	\\
		0		&0		\\
		0		&2		\\
		3		&2		\\
		3		&0		\\
	\end{array}&&\setlength\arraycolsep{1.5pt} 
	\begin{array}{cc}
		P^0		&P^2	\\\hline
		0		&2		\\
		0		&2		\\
		0		&0		\\
		0		&2		\\
		3		&2		\\
		3		&0		\\
	\end{array}
\end{array}$&$\begin{array}{ccc}	%% ______________ {0, 3}
	\setlength\arraycolsep{1.5pt}	
	\begin{array}{cc}
		Q^0		&Q^3	\\\hline
		\gr{0}	&\gr{3}	\\
		\gr{0}	&\gr{3}	\\
		0		&2		\\
		0		&2		\\
		3		&2		\\
		3		&3		\\
	\end{array}&&\setlength\arraycolsep{1.5pt} 
	\begin{array}{cc}
		P^0		&P^3	\\\hline
		0		&2		\\
		0		&2		\\
		0		&3		\\
		0		&3		\\
		3		&3		\\
		3		&2		\\
	\end{array}
\end{array}$\\[-4pt]&&\\\hline&&\\[-4pt]
$\begin{array}{ccc}				%% ______________ {1, 2}
	\setlength\arraycolsep{1.5pt}	
	\begin{array}{cc}
		Q^1		&Q^2	\\\hline
		\gr{1}	&\gr{2}	\\
		\gr{1}	&\gr{2}	\\
		1		&0		\\
		0		&2		\\
		1		&2		\\
		0		&0		\\
	\end{array}&&\setlength\arraycolsep{1.5pt} 
	\begin{array}{cc}
		P^1		&P^2	\\\hline
		1		&2		\\
		1		&2		\\
		1		&0		\\
		0		&2		\\
		1		&2		\\
		0		&0		\\
	\end{array}
\end{array}$&$\begin{array}{ccc}	%% ______________ {1, 3}
	\setlength\arraycolsep{1.5pt}	
	\begin{array}{cc}
		Q^1		&Q^3	\\\hline
		\gr{1}	&\gr{3}	\\
		\gr{1}	&\gr{3}	\\
		1		&2		\\
		0		&2		\\
		1		&2		\\
		0		&3		\\
	\end{array}&&\setlength\arraycolsep{1.5pt} 
	\begin{array}{cc}
		P^1		&P^3	\\\hline
		1		&2		\\
		1		&2		\\
		1		&3		\\
		0		&3		\\
		1		&3		\\
		0		&2		\\
	\end{array}
\end{array}$&$\begin{array}{ccc}	%% ______________ {2, 3}
	\setlength\arraycolsep{1.5pt}	
	\begin{array}{cc}
		Q^2		&Q^3	\\\hline
		\gr{2}	&\gr{3}	\\
		\gr{2}	&\gr{3}	\\
		0		&2		\\
		2		&2		\\
		2		&2		\\
		0		&3		\\
	\end{array}&&\setlength\arraycolsep{1.5pt} 
	\begin{array}{cc}
		P^2		&P^3	\\\hline
		2		&2		\\
		2		&2		\\
		0		&3		\\
		2		&3		\\
		2		&3		\\
		0		&2		\\
	\end{array}
\end{array}$
\end{tabular}
\end{center}}
\end{figure}

% _____________________________________________  Définition ARPAs
For a positive integer $q$, we consider the set $\Sigma_q = \set{0, 1 ,\ldots, q-1}$ of $q$ symbols. In \cite{CT18}, we introduced a family of combinatorial designs, which we here call {\em alphabet reduction pairs of arrays} (ARPAs for short), along with the associated quantities of interest.

\begin{definition}[Alphabet reduction pairs of arrays]\label{def-Gamma}
Let $k >0$, $p\geq k$, and $q\geq p$ be three integers. 
Two arrays $Q$ and $P$ with $q$ columns on the symbol set $\Sigma_q$ form a {\em $(q, p)$-alphabet reduction pair of arrays (a $(q, p)$-ARPA for short) of strength $k$} if they satisfy:
\begin{itemize}
	\item[] $(\Gamma_Q)$ $Q$ contains at least one occurrence of the row $0\ 1\ \cdots\ q-1$;
	\item[] $(\Gamma_P)$ each row of $P$ involves at most $p$ distinct symbols;
	\item[] $(k_=)$ if we extract $k$ columns from $Q$ and the same $k$ columns from $P$, then we obtain the same array up to row permutation.
\end{itemize} 

$\Gamma(q, p, k)$ denotes the set of such pairs of arrays. Furthermore, let $R^*(Q, P)$ and $R(Q, P)$ refer to, respectively, the number of occurrences of the row $0\ 1\ \cdots\ q-1$ in $Q$, and the number of rows in $P$. Then we define $\gamma(q, p, k)$ as the highest ratio $R^*(Q, P)/R(Q, P	)$ attained over $\Gamma(q, p, k)$.
\end{definition}
Note that the condition $(k_=)$ ensures that the number of rows in $Q$ and $P$ is the same. 
\Cref{fig-Gamma-ex} shows $(q, p)$-ARPAs of strength $k$ for $(q, p, k)\in\set{(4, 3, 2), (5, 3, 2), (5, 4, 3)}$. 
For example, consider the upper-left pair of arrays in this figure. For this ARPA, we have $(q, p, k) =(4, 3, 2)$, $R(Q, P) =6$, and $R^*(Q, P) =2$. Furthermore, we can check in \cref{fig-Gamma-k=} that this pair of arrays satisfies condition $(k_=)$.

% _____________________________________________  Motivation & travaux antérieurs ARPAs
\subsection{Motivation and previous results} 

In an optimization {\em Constraint Satisfaction Problem} (CSP for short) over the alphabet $\Sigma_q$, the objective is to assign values from $\Sigma_q$ to variables so as to maximize a weighted sum of constraint evaluations. We denote by $\mathsf{k\,CSP\!-\!q}$ the CSP over $\Sigma_q$ where each constraint involves at most $k$ variables. 
For a positive integer $q$ and a CSP instance $I$ over $\Sigma_q$, we denote by $\mathrm{opt}(I)$ its optimal value, by $\mathrm{wor}(I)$ its worst solution value, and by $\mathrm{opt}_p(I)$ the optimal value when restricting to solutions whose coordinates involve at most $p <q$ distinct values. 
A value $\mathrm{apx}(I)$ \emph{approximates $\mathrm{opt}(I)$ within a differential factor $\rho$}, where $\rho\in(0, 1]$, if it satisfies $(\mathrm{apx}(I) -\mathrm{wor}(I))/(\mathrm{opt}(I) -\mathrm{wor}(I))\geq\rho$. 
By extension, a solution $x$ of $I$ is \emph{$\rho$-differential approximate} if its objective value approximates $\mathrm{opt}(I)$ within a differential factor of $\rho$. 
For some insight into the differential approximation measure, the approximability of $k$-CSPs, or more specifically their differential approximability see, {\em e.g.}, \cite{ABMV77,AAP80,BR95,DP96,MM17,N98,EP05,CT18,CT26-E}.  

ARPAs are related to CSPs in that the higher the ratio $R^*(Q, P)/R(Q, P)$ is in an ARPA of $\Gamma(q, p, k)$, the stronger the guarantee that $\mathrm{opt}_p(I)$ approximates $\mathrm{opt}(I)$, and the more we benefit, for $\mathsf{k\,CSP\!-\!q}$, from using an approximation algorithm for $\mathsf{k\,CSP\!-\!p}$.

\begin{theorem}[\cite{CT18}]\label{thm-CT18}
For all constant integers $k\geq 2$, $p\geq k$, and $q\geq p$, on any instance $I$ of $\mathsf{k\,CSP\!-\!q}$, the best solutions among those whose components take at most $p$ distinct values are $\gamma(q, p, k)$-differential approximate. Formally, $\mathrm{opt}_p(I)$ satisfies:
$$\begin{array}{rl}
	(\mathrm{opt}_p(I) - \mathrm{wor}(I))/(\mathrm{opt}(I) - \mathrm{wor}(I))	&\geq \gamma(q, p, k).
\end{array}$$

Moreover, if there exists a polynomial-time algorithm for $\mathsf{k\,CSP\!-\!p}$ that computes $\rho$-differential approximate solutions (where $\rho\in(0, 1]$), it can be used for $\mathsf{k\,CSP\!-\!q}$ to compute solutions that are $\rho\times\gamma(q, p, k)$-differential approximate within polynomial time.
\end{theorem} 
Hence, as far as CSPs are concerned, we are interested in the frequency of the row $0\ 1\ \cdots\ q-1$ in the array $Q$. 
The following bounds are known for $\gamma(q, p, k)$:

\begin{theorem}[\cite{CT18}]\label{thm-gamma_qpk}
Let $k > 0$, $p\geq k$, and $q\geq p$ be three integers. 
If $p = q$, then $\gamma(q, q, k) = 1$. 
If $q > p > k$, then $\gamma(q, p, k)\geq\gamma(q-p +k, k, k)$.
If $p = k < q$, then 
\begin{align}\label{eq-qkk-LB}
\gamma(q, k, k) &\textstyle
	\geq 2/(\sum_{r =0}^k \binom{q}{r}\binom{q-1 -r}{k-r} +1).
\end{align}
\end{theorem}

From \cref{thm-CT18,thm-gamma_qpk}, and the fact that $\mathsf{2\,CSP\!-\!2}$ is approximable within a differential ratio of $2-\pi/2$ \cite{N98}, we deduced in \cite{CT18} that, for all integers $q > 2$, $\mathsf{2\,CSP\!-\!q}$ is approximable within a differential ratio of $(2-\pi/2)/(q-1)^2$.

% _____________________________________________  Outline
\subsection{Outline and notations}

In \cite{CT18}, we identified two issues concerning the lower bounds given in \cref{thm-gamma_qpk}: proving that, as we conjectured, the stated bound is tight for the case $p =k$; providing finer estimates of $\gamma(q, p, k)$ when $p >k$.
 %
% _____________________________________________  Intro CPAs
Essentially, when computing $\gamma(q, p, k)$, our goal is to cover as many occurrences of the row $0\ 1\ \cdots\ q-1$ as possible, using the smallest collection of rows, each with coordinates taking at most $p$ distinct values. Given an index $j\in\Sigma_q$, the most critical aspect of an entry $a$ in a column $j$ of $Q$ or $P$ is whether it matches $j$ or not. When $a \neq j$, the specific value of $a$ becomes less significant, except for ensuring that the subarrays of $k$ columns of $Q$ and $P$ coincide. These considerations lead us to introduce the following family of combinatorial designs:

% ________________ Tables de couverture :: définition
\begin{definition}[Cover pair of arrays]\label{def-Delta}
Let $k >0$, $d\geq k$, and $\nu\geq d$ be three integers.
Two arrays $N$ and $D$ with $\nu$ columns on the symbol set $\set{0, 1}$ form a {\em $(\nu, d)$-cover pair of arrays (a $(\nu, d)$-CPA for short) of strength $k$} if they satisfy condition $(k_=)$ and:
\begin{itemize}\itemsep0pt
	\item[] $(\Delta_N)$ $N$ contains at least one occurrence of the row $1\ 1\ \cdots\ 1$;
	\item[] $(\Delta_D)$ each row of $D$ contains at most $d$ 1s.
\end{itemize} 

$\Delta(\nu, d, k)$ denotes the set of such pairs of arrays. Furthermore, let $R^*(N, D)$ and $R(N, D)$ refer to, respectively, the number of occurrences of the row $1\ 1\ \cdots\ 1$ in $N$, and the number of rows in $D$.
Then we define $\delta(\nu, d, k)$ as the maximum ratio $R^*(N, D)/R(N, D)$ over $\Delta(\nu, d, k)$.
\end{definition}

\Cref{fig-Delta-ex} shows CPAs of strength 2 and 3.
% _____________________________________________  Outline
The manuscript is organized as follows:
In \cref{sec-cpa}, we show that the lower bound for $\gamma(q, k, k)$ given in \cite{CT18} is an upper bound for $\delta(q, k, k)$ (\cref{thm-delta_d=k-UB}), and thus the exact value of $\gamma(q, k, k)$. 
We also introduce a subfamily of CPAs, called {\em regular}, which will be used in the rest of the paper. 
In \cref{sec-D2G}, we show how to derive from regular CPAs ARPAs of the same strength (\cref{thm-reg_joker}). 
In \cref{sec-LP}, we provide a characterization of regular CPAs that maximize the frequency of the all-1 row using linear programming (\cref{thm-delta_opt}). 
In \cref{sec-opt}, we conclude that computing $\gamma(q, p, k)$ reduces to computing $\delta(q, p, k)$ (\cref{cor-gamma=delta}). 
In addition, we derive the exact value of $\gamma(q, k, k)$, $\gamma(q, p, 2)$, and $\gamma(q, p, 1)$ from the results of the previous sections (\cref{cor-delta-p=k+k<=2}). 
Finally, in \cref{sec-conc}, we briefly discuss the results obtained and directions for further research.

%% thm-delta_opt
But first, in \cref{sec-lem_tech}, we present a technical lemma that will be used in subsequent arguments, namely in the proof of \cref{thm-delta_d=k-UB,prop-Delta_PL-SB}.

% _____________________________________________  Notations
\medskip
{\bf{\em Notations.}}
For a positive integer $\nu$, $[\nu]$ denotes the integer interval $\set{1, 2 ,\ldots, \nu}$.
Unless otherwise specified, we index the coordinates of rows occurring in $(q, p)$-ARPAs by $\Sigma_q$, and the coordinates of rows occurring in $(\nu, d)$-CPAs by $[\nu]$. 
The rows of an array with $R$ rows are indexed by $[R]$. 
For an array $M$, $M_r$ and $M^j$ denote its row of index $r$ and its column of index $j$, respectively. 

\begin{figure}[t]{\footnotesize
\caption{$(\nu, d)$-CPAs of strength 2 and 3. 
We highlight in gray the rows of the form $1\ 1\ \cdots\ 1$ in the array $N$.}
\label{fig-Delta-ex}
\begin{center}
%% /!\/!\ l'ordre des lignes des tableaux ici doit être le même que dans Gamma-ex /!\/!\
\begin{tabular}{c|c}
$\begin{array}{c}
	(N, D)\in\Delta(4, 3, 2)\\[4pt]	%% (\nu, d, k) =(4, 3, 2)
	\begin{array}{ccc}\setlength\arraycolsep{1.5pt}
		\begin{array}{cccc}
			N^1		&N^2	&N^3	&N^4\\\hline
			\gr{1}&\gr{1}&\gr{1}&\gr{1}\\
			\gr{1}&\gr{1}&\gr{1}&\gr{1}\\
			1		&1		&0		&0	\\
			1		&0		&1		&0	\\
			0		&1		&1		&0	\\
			0		&0		&0		&1	\\
		\end{array}&&\setlength\arraycolsep{1.5pt} 
		\begin{array}{cccc}
			D^1		&D^2	&D^3	&D^4\\\hline
			1		&1		&1		&0	\\
			1		&1		&1		&0	\\
			1		&1		&0		&1	\\
			1		&0		&1		&1	\\
			0		&1		&1		&1	\\
			0		&0		&0		&0	\\
		\end{array} 
	\end{array}\\[-6pt]\\
	R^*(N, D)/R(N, D) =2/6 =1/3\\[-4pt]\\\hline\\[-6pt]
	(N, D)\in\Delta(5, 3, 2)\\[4pt] %% (\nu, d, k) =(5, 3, 2)
	\begin{array}{ccc}\setlength\arraycolsep{1.5pt}	
		\begin{array}{ccccc}
			N^1		&N^2	&N^3	&N^4	&N^5\\\hline
			\gr{1}	&\gr{1}	&\gr{1}	&\gr{1}	&\gr{1}\\
			0		&0		&1		&0		&1	\\
			0		&0		&1		&0		&1	\\
			0		&0		&0		&1		&0	\\
			0		&1		&0		&0		&0	\\
			1		&0		&0		&0		&0	\\
		\end{array}&&\setlength\arraycolsep{1.5pt} 
		\begin{array}{ccccc}
			D^1		&D^2	&D^3	&D^4	&D^5\\\hline
			0		&0		&1		&1		&1	\\
			0		&1		&1		&0		&1	\\
			1		&0		&1		&0		&1	\\
			1		&1		&0		&1		&0	\\
			0		&0		&0		&0		&0	\\
			0		&0		&0		&0		&0	\\
		\end{array} 
	\end{array}\\[-6pt]\\
	R^*(N, D)/R(N, D) =1/6
\end{array}$
&$\begin{array}{c}
	(N, D)\in\Delta(5, 4, 3)\\[4pt]	%% (\nu, d, k) =(5, 4, 3)
	\begin{array}{ccc}\setlength\arraycolsep{1.5pt}		
		\begin{array}{ccccc}
			N^1	&N^2&N^3&N^4&N^5\\\hline
			1	&0	&0	&1	&1	\\ % 0 3 4
			1	&0	&1	&0	&1	\\ % 0 2 4
			1	&0	&1	&1	&0	\\ % 0 2 3
			1	&1	&0	&0	&1	\\ % 0 1 4
			1	&1	&0	&1	&0	\\ % 0 1 3
			1	&1	&1	&0	&0	\\ % 0 1 2 
			\gr{1}&\gr{1}&\gr{1}&\gr{1}&\gr{1}	\\
			\gr{1}&\gr{1}&\gr{1}&\gr{1}&\gr{1}	\\
			\gr{1}&\gr{1}&\gr{1}&\gr{1}&\gr{1}	\\
			0	&0	&0	&0	&0	\\
			0	&0	&0	&0	&0	\\
			0	&0	&1	&1	&1	\\ % 2 3 4
			0	&1	&0	&1	&1	\\ % 1 3 4
			0	&1	&1	&0	&1	\\ % 1 2 4
			0	&1	&1	&1	&0	\\ % 1 2 3
		\end{array}&&\setlength\arraycolsep{1.5pt} 
		\begin{array}{ccccc}
			D^1	&D^2&D^3&D^4&D^5\\\hline
			1	&0	&0	&0	&0	\\ % 0
			1	&0	&1	&1	&1	\\ % 0 2 3 4
			1	&0	&1	&1	&1	\\ % 0 2 3 4
			1	&1	&0	&1	&1	\\ % 0 1 3 4
			1	&1	&0	&1	&1	\\ % 0 1 3 4
			1	&1	&1	&0	&1	\\ % 0 1 2 4
			1	&1	&1	&0	&1	\\ % 0 1 2 4
			1	&1	&1	&1	&0	\\ % 0 1 2 3
			1	&1	&1	&1	&0	\\ % 0 1 2 3
			0	&0	&0	&0	&1	\\ % 4
			0	&0	&0	&1	&0	\\ % 3
			0	&0	&1	&0	&0	\\ % 2
			0	&1	&0	&0	&0	\\ % 1
			0	&1	&1	&1	&1	\\ % 1 2 3 4
			0	&1	&1	&1	&1	\\ % 1 2 3 4
		\end{array} 
	\end{array}\\[-6pt]\\
	R^*(N, D)/R(N, D) =3/15 =1/5
\end{array}$
\end{tabular}
\end{center}}
\end{figure}

% ____________________________ Lemme technique
\section{Technical lemma}\label{sec-lem_tech}
\begin{lemma}\label{lem-fAB}
For all sets $A$ and $B$ of real numbers such that $|B| > |A|$, we have:
\begin{align}\label{eq-fAB}
f(A, B)	\textstyle
	:=\sum_{i\in B}\frac{\prod_{a\in A}(i-a)}{\prod_{b\in B\setminus\set{i}} (i-b)}
		&=\begin{cases}
			1	&\text{if $|B| = |A| + 1$}\\
			0	&\text{if $|B| > |A| + 1$}
		\end{cases}.
\end{align}
\end{lemma}

\begin{proof}
We show that for all $A, B\subseteq\mathbb{R}$ with $|B| > |A|$, the quantities $f(A, B)$ satisfy the following recurrence relation and initial conditions:
\begin{align}\label{eq-f_A_B-rec}
f(A, B)	
		&=\begin{cases}
				1	&\text{if $A =\emptyset\text{ and }|B| =1$},\\
				0	&\text{if $A =\emptyset\text{ and }|B| >1$},\\
				f(A\setminus\set{a}, B\setminus\set{b})
					+(b-a)f(A\setminus\set{a}, B),\ a\in A,\ b\in B	
				&\text{if $A\neq\emptyset$}.
		\end{cases}
\end{align}

When $A = \emptyset$, we consider the polynomial:
$$\begin{array}{rl}
	P(X)	&:=\sum_{i \in B}\prod_{b\in B\setminus\set{i}} \frac{X-b}{i-b}.
\end{array}$$
$P$ is the Lagrange interpolating polynomial at the nodes $(b, 1)$ for $b\in B$. 
%% (tous les produits sont nuls, sauf celui pour lequel $b_i = b$ qui vaut 1).  
By construction, this polynomial has degree at most $|B|-1$, and takes the value $1$ at each $b\in B$. 
We deduce that $P$ is the constant polynomial equal to 1. 
Note that $f(\emptyset, B)$ is the coefficient of $X^{|B|-1}$ in $P(X)$. 
Since $P(X) = 1$, this coefficient is $1$ when $|B| = 1$, and $0$ otherwise. 
 
When $A\neq\emptyset$, for $(a, b)\in A\times B$, we write successively:
$$\begin{array}{rl}
f(A, B)
	&=\sum_{i\in B\setminus\set{b}}	 \frac{\prod_{a'\in A}(i-a')}{\prod_{b'\in B\setminus\set{i}} (i-b')}
									+\frac{\prod_{a'\in A}(b-a')}{\prod_{b'\in B\setminus\set{b}} (b-b')}\\[6pt]
	&=\sum_{i\in B\setminus\set{b}} 
		\frac{(i-b+b-a)\times\prod_{a'\in A\setminus\set{a}}(i-a')}{\prod_{b'\in B\setminus\set{i}} (i-b')}
			+\frac{(b-a)\times\prod_{a'\in A\setminus\set{a}}(b-a')}{\prod_{b'\in B\setminus\set{b}} (b-b')}\\[6pt]
	&=\sum_{i\in B\setminus\set{b}}	
							\frac{\prod_{a'\in A\setminus\set{a}}(i-a')}{\prod_{b'\in B\setminus\set{b, i}} (i-b')}
		+(b-a)\times\sum_{i\in B}\frac{\prod_{a'\in A\setminus\set{a}}(i-a')}{\prod_{b'\in B\setminus\set{i}} (i-b')}.
\end{array}$$

Thus, relation \cref{eq-f_A_B-rec} holds.
It follows from \cref{eq-f_A_B-rec} that $f(A, B)$ can be expressed as a weighted sum of terms of the form $f(\emptyset, B')$ with $B'\subseteq B$ and $|B'|\in\set{|B|-|A|, |B|-|A|+1 ,\ldots, |B|}$, including a single term with $|B'| = |B|-|A|$ whose coefficient is $1$. We conclude that $f(A, B) = 1$ if $|B|-|A|=1$ and $0$ otherwise.
\end{proof}

% ____________________________ Delta
% ______________ Delta
% _______ Définition
% _______ De Gamma vers Delta rég (trivial)
% _______ Borne sup gamma(q, k, k) via delta(q, k, k)
% _______ Delta rég
\section{Cover pairs of arrays}\label{sec-cpa}
% ___________________________________________________ ARPA->CPA (dont UB)
\subsection{From alphabet reduction to cover pairs of arrays}
 %
% ________________ De ARPA vers CPA
An ARPA can naturally be interpreted as a CPA. 
For instance, the ARPAs of \cref{fig-Gamma-ex} can be interpreted as the CPAs of \cref{fig-Delta-ex}.

\begin{definition}\label{def-G_D}
For a positive integer $q$, we define the surjective map $\pi_q:\Sigma_q^q\rightarrow\set{0, 1}^q$ that maps a word $u = (u_0, u_1 ,\ldots, u_{q-1})$ of $\Sigma_q^q$ to the word 
$\pi_q(u) =(\pi_q(u)_1, \pi_q(u)_2 ,\ldots, \pi_q(u)_q)$ of $\set{0, 1}^q$ defined by
$$\begin{array}{rll}
	\pi_q(u)_j &=\begin{cases}
		1		&\text{if $u_{j-1}=j-1$},	\\
		0		&\text{otherwise},
	\end{cases}	&j\in[q].
\end{array}$$

By extension, given an array $M$ with rows in $\Sigma_q^q$, we define $\pi_q(M)$ as the array obtained from $M$ by applying $\pi_q$ to each of its rows. 

For two pairs of arrays, $(Q, P)$ with rows in $\Sigma_q^q$ and $(N, D)$ with rows in $\set{0, 1}^q$, we say that $(Q, P)$ \emph{can be interpreted as} $(N, D)$ if we can order their rows in such a way that $(\pi_q(Q), \pi_q(P)) = (N, D)$.
\end{definition}

\begin{proposition}\label{prop-Gamma_2_Delta}
Let $k > 0$, $p \geq k$, and $q \geq p$ be three integers. Then for all $(Q, P)\in\Gamma(q, p, k)$, there exists $(N, D)\in\Delta(q, p, k)$ such that $R^*(N, D) = R^*(Q, P)$ and $R(N, D) = R(Q, P)$.
\end{proposition}

\begin{proof}
Define $(N, D)$ as $\left(\pi_q(Q), \pi_q(P)\right)$. 
By definition of $\pi_q$, the number of rows in $N$ and $D$ is the same as in $Q$ and $P$, and the row of 1s occurs in $N$ as often as the row $0\ 1\ \cdots\ q-1$ occurs in $Q$. 
Furthermore, $(N, D)$ clearly satisfies $(k_=)$. For $(\Delta_D)$, we observe that the number of distinct symbols in a word $u = (u_0, u_1 ,\ldots, u_{q-1})\in\Sigma_q^q$ is an upper bound on the number of coordinates $u_j$ of $u$ satisfying $u_j = j$, and thus, on the number of 1s in the word $\pi_q(u)$. 
\end{proof}

% ________________ UB
We now prove that when $p = k$, the lower bound for $\gamma(q, k, k)$ given in \cref{thm-gamma_qpk} is an upper bound for $\delta(q, k, k)$. 
\begin{theorem}\label{thm-delta_d=k-UB}
Let $k > 0$ and $\nu > k$ be integers. Then every $(\nu, k)$-CPA $(N, D)$ of strength $k$ satisfies
	$R^*(N, D)/R(N, D) \leq 2/(\sum_{h = 0}^k \binom{\nu}{h}\binom{\nu-1-h}{k-h}+1)$.
\end{theorem}

\begin{proof}
First consider three integers $k > 0$, $d\geq k$, and $\nu\geq d$, and a CPA $(N, D)\in\Delta(\nu, d, k)$.
Given $i\in\set{0, 1 ,\ldots, \nu}$, we denote by $b_i$ the number of rows with $i$ 1s in $N$ and, if $i\leq d$, by $a_i$ the corresponding number in $D$. 
In particular, $b_\nu = R^*(N, D)$ and $\sum_{i = 0}^\nu b_i+\sum_{i = 0}^d a_i = 2R(N, D)$. 

The array pair $(N, D)$ satisfies $(k_=)$ if and only if, for any $k$ integers $j_1, j_2 ,\ldots, j_k\in[\nu]$ and any $w\in\set{0, 1}^k$, there are as many rows $u = (u_1, u_2 ,\ldots, u_\nu)$ satisfying $(u_{j_1}, u_{j_2} ,\ldots, u_{j_k}) = w$ in $D$ as in $N$. 
For $H := \set{j_i : i\in[k] \wedge w_i = 1}$ and $L := \set{j_i : i\in[k] \wedge w_i = 0}$, we have $(u_{j_1}, u_{j_2} ,\ldots, u_{j_k}) = w$ if and only if the coordinates of $u$ indexed by $H$ are all equal to 1, and the coordinates of $u$ indexed by $L$ are all equal to 0.
Thus, $(N, D)$ satisfies $(k_=)$ if and only if, for any two subsets $H, L$ of $[\nu]$ such that $H\cap L = \emptyset$ and $\card{H \cup L} = k$, the number of rows $u = (u_1, u_2 ,\ldots, u_\nu)$ satisfying $u_H = (1, 1 ,\ldots, 1)$ and $u_L = (0, 0 ,\ldots, 0)$ is the same in both arrays. 
For a fixed cardinality $h\in\set{0, 1 ,\ldots, k}$ of $H$, summing these equalities over all possible pairs $(H, L)$ yields the equality:
\begin{align}\label{eq-sum_h}
\begin{array}{l}
	\sum_{H, L\subseteq [\nu]: H \cap L = \emptyset, |H| = h, |L| = k-h}
		\card{\set{r\in[R(N, D)]\,:\,(N_r^j = 1, j\in H)\ \wedge\ (N_r^j = 0, j\in L)}}	\\	=\ 
	\sum_{H, L\subseteq [\nu]: H \cap L = \emptyset, |H| = h, |L| = k-h}
		\card{\set{r\in[R(N, D)]\,:\,(D_r^j = 1, j\in H)\ \wedge\ (D_r^j = 0, j\in L)}}.
\end{array}
\end{align}

Let $u\in\set{0, 1}^\nu$. If $u$ has $i$ non-zero coordinates, then $u_H = (1, 1 ,\ldots, 1)$ for exactly $\binom{i}{h}$ subsets $H\subseteq[\nu]$ of size $h$, and $u_L = (0, 0 ,\ldots, 0)$ for exactly $\binom{\nu-i}{k-h}$ subsets $L\subseteq[\nu]$ of size $k-h$. 
Thus, a row $N_r$ of $N$ or $D_r$ of $D$ with $i$ 1s and $\nu-i$ 0s is involved $\binom{i}{h}\times\binom{\nu-i}{k-h}$ times in the left and right sides, respectively, of equality \cref{eq-sum_h}.
From these observations, we deduce that the numbers $b_i$, $i\in\set{0, 1 ,\ldots \nu}$ and $a_i$, $i\in\set{0, 1 ,\ldots d}$ satisfy the equalities:
\begin{align}\label{eq-qkk-==!=}	%%{\nu-k+h}
\textstyle	\sum_{i = h}^\nu \binom{i}{h}\binom{\nu-i}{k-h} b_i
&\textstyle	=\sum_{i = h}^d \binom{i}{h}\binom{\nu-i}{k-h} a_i,	
				&h\in\set{0, 1 ,\ldots, k}.
\end{align}

%% using \cref{lem-fAB}
Now suppose that $d=k$. 
We prove by induction on $h$ that the following equality holds:
\begin{align}\label{eq-qkk-MN-rec}
a_h-b_h &\textstyle
	=(-1)^{k-h} \sum_{i=k+1}^\nu\binom{i}{h}\binom{i-1-h}{k-h} b_i,	&h\in\set{0, 1 ,\ldots, k}.
\end{align}
Assuming this identity, we observe successively:
$$\begin{array}{rll}
2R(N, D)
&=		\sum_{h=0}^k (b_h+a_h)+\sum_{h=k+1}^\nu b_h\\
&\geq 	\sum_{h=0}^k \abs{a_h-b_h}+\sum_{h=k+1}^\nu b_h
\\ %%	&\text{as $a_h, b_h \geq 0$, $h\in\set{0, 1 ,\ldots, k}$}\\
&\geq 	\sum_{h=0}^k \sum_{i=k+1}^\nu \binom{i}{h}\binom{i-1-h}{k-h} b_i+\sum_{h=k+1}^\nu b_h
&\text{by \cref{eq-qkk-MN-rec}}\\
&\geq \sum_{h=0}^k \binom{\nu-1-h}{k-h}\binom{\nu}{h} b_\nu+b_\nu.\\ 
%%	&\text{as $\nu > k$ and $b_{k+1}, b_{k+2},\ldots, b_{\nu-1} \geq 0$}
\end{array}$$
The conclusion then follows from considering $b_\nu = R^*(N, D)$. 

We now establish \cref{eq-qkk-MN-rec}.
For $h=k$, we have $a_k=\sum_{i=k}^\nu\binom{i}{k} b_i$ by \cref{eq-qkk-==!=}.
Thus $a_k-b_k=\sum_{i=k+1}^\nu \binom{i}{k} b_i$, and relation \cref{eq-qkk-MN-rec} is satisfied for $h=k$. 
Now suppose that \cref{eq-qkk-MN-rec} holds for all $\ell\in\set{h+1 ,\ldots, k}$, for some $h\in\set{0, 1 ,\ldots, k-1}$, and consider the case $h$. 
We observe successively:
\begin{align}
\nonumber\textstyle\binom{\nu-h}{k-h}(b_h-a_h) 
	&\textstyle=\sum_{\ell=h+1}^k \binom{\ell}{h}\binom{\nu-\ell}{k-h}(a_\ell-b_\ell)
					-\sum_{i=k+1}^\nu\binom{i}{h} \binom{\nu-i}{k-h} b_i
		&\text{by \cref{eq-qkk-==!=}}		\\
\nonumber
	&\textstyle=\sum_{\ell=h+1}^k \binom{\ell}{h}\binom{\nu-\ell}{k-h} \times
						(-1)^{k-\ell} \sum_{i=k+1}^\nu\binom{i}{\ell}\binom{i-1-\ell}{k-\ell} b_i
					-\sum_{i=k+1}^\nu\binom{i}{h} \binom{\nu-i}{k-h} b_i
		&\text{by \cref{eq-qkk-MN-rec}}
\end{align}
Using the identity $\binom{\ell}{h}\binom{i}{\ell}=\binom{i}{h}\binom{i-h}{\ell-h}$, we obtain the following equality:
\begin{align}\label{eq-Delta_h}
\textstyle\binom{\nu-h}{k-h}(b_h-a_h) 
	&\textstyle=\sum_{i=k+1}^\nu\binom{i}{h} b_i\times\left(
			\sum_{\ell=h+1}^k
				(-1)^{k-\ell}\binom{\nu-\ell}{k-h}\binom{i-h}{\ell-h}\binom{i-1-\ell}{k-\ell}
			-\binom{\nu-i}{k-h}
	\right).
\end{align}
Hence, a sufficient condition for \cref{eq-qkk-MN-rec} to be satisfied at $h$ is that for all $i\in\set{k +1 ,\ldots, \nu}$, the coefficient of $\binom{i}{h}b_i$ in the right-hand side of \cref{eq-Delta_h} is equal to $-\binom{\nu-h}{k-h}$ times the coefficient of $\binom{i}{h}b_i$ in the right-hand side of \cref{eq-qkk-MN-rec}, namely:
$$\begin{array}{rl}
\sum_{\ell=h+1}^k (-1)^{k-\ell}\binom{\nu-\ell}{k-h}\binom{i-h}{\ell-h}\binom{i-1-\ell}{k-\ell}
	-\binom{\nu-i}{k-h}	&=-\binom{\nu-h}{k-h}\times(-1)^{k-h}\binom{i-1-h}{k-h}.	
\end{array}$$
Equivalently, it suffices to establish the following equality:
\begin{align}\label{eq-Delta_h_i}
\textstyle\sum_{\ell=h}^k(-1)^{k-\ell}\binom{\nu-\ell}{k-h}\binom{i-h}{\ell-h}\binom{i-1-\ell}{k-\ell}
	-\binom{\nu-i}{k-h} &=0,	&i\in\set{k+1 ,\ldots, \nu}
\end{align}

Fix $i\in\set{k +1 ,\ldots, \nu}$. 
Define the sets 
$$A := \set{\nu-k+h+1 ,\ldots, \nu}\text{ and }B := \set{h ,\ldots, k}\cup\set{i}.$$
Then we observe that the left-hand side of \cref{eq-Delta_h_i} is a constant multiple 
of the quantity $f(A,B)$ defined by \cref{eq-fAB} (see \cref{appendix:sec-UB} for a detailed proof). 
Since $|B|=k-h+2=|A|+2$, \cref{lem-fAB} implies $f(A,B)=0$, thus concluding the proof. 
\end{proof}

\Cref{thm-delta_d=k-UB} together with \cref{prop-Gamma_2_Delta,thm-gamma_qpk} implies for all integers $k > 0$ and $q > k$ the equality $\gamma(q, k, k) = 2/(\sum_{r = 0}^k \binom{q}{r}\binom{q-1-r}{k-r}+1) = \delta(q, k, k)$.
The rest of the paper focuses on a specific family of cover pairs of arrays, called \emph{regular}. 

% ___________________________________________________ Tables régulières
% ________________ Des CPAs vers les ARPAs (jokers) :: CPAs régulières
\subsection{Regular pairs of arrays}

% ________________ CPAs régulières :: définition
\begin{definition}\label{def-delta-reg}
The {\em weight} of a Boolean vector is the number of its non-zero coordinates.
For a positive integer $\nu$, we define an array $M$ whose rows belong to $\set{0, 1}^\nu$ to be {\em regular} if, for all $i\in\set{0, 1 ,\ldots, \nu}$, Boolean words of length $\nu$ and weight $i$ all occur with the same frequency in $M$. 
By extension, a CPA $(N, D)$ where $N$ and $D$ are both regular is said to be {\em regular}.
\end{definition}
For instance, the CPA on the right of \cref{fig-Delta-ex} is regular. In this CPA, 
the words with weight 0, 3, and 5 occur twice, once, and three times in $N$, respectively, and the words with weight 1 and 4 occur once and twice in $D$, respectively.

% ________________ CPAs régulières :: restriction w.l.o.g.
\begin{proposition}\label{prop-reg_delta}
For all positive integers $k$, $d \geq k$, and $\nu \geq d$, if there exists $(N, D)\in\Delta(\nu, d, k)$ that realizes $\delta(\nu, d, k)$, then there exists $(N', D')\in\Delta(\nu, d, k)$ that realizes $\delta(\nu, d, k)$, and is regular.
\end{proposition}

\begin{proof}
If $(N, D)$ is not regular, then for $M\in\set{N, D}$, we construct $M'$ by inserting all possible permutations of each row of $M$. 
The conditions $(\Delta_N)$, $(\Delta_D)$, and $(k_=)$ are invariant under any permutation of the columns of the two arrays $N$ and $D$. 
%% NB ex. le (4,3)-CPA de force é de la Fig. 3 est invariant par échange des 2 1ères colonnes
Therefore, $(N', D')$ is the concatenation of $\nu!$ (not necessarily distinct) elements $(N'',D'')$ of $\Delta(\nu, d, k)$ corresponding to the $\nu!$ permutations of the columns of $N$ and $D$, each of which satisfies $R^*(N'', D'') = R^*(N, D)$ and $R(N'', D'') = R(N, D)$. 
Hence, $(N', D')$ is an element of $\Delta(\nu, d, k)$ with a ratio $R^*(N', D')/R(N', D')$ equal to $R^*(N, D)/R(N, D)$, which by assumption is equal to $\delta(\nu, d, k)$.

It remains to show that $(N', D')$ is regular. 
Let $M\in\set{N, D}$ and $u\in\set{0, 1}^\nu$. 
By construction, a row of $M'$ coincides with $u$ if and only if there exists a permutation $\sigma$ on $[\nu]$ and a row $M_s$ of $M$ such that
	$\sigma$ maps the 1s and the 0s of $M_s$ to those of $u$, 
	implying that $M_s$ has the same weight as $u$. 
If $i$ denotes the weight of $u$, we deduce that the number of rows of the form $u$ in $M'$ is $i!\times(\nu-i)!$ times the number of rows of weight $i$ in $M$. 
Since this number depends only on $i$, we conclude that $(N', D')$ is indeed regular. 
\end{proof}

Since the numbers $\delta(\nu, d, k)$ we are interested in can be achieved by regular arrays, we will henceforth restrict our focus to regular CPAs. 
We note that an array $M$ of such CPAs is fully characterized by the number of occurrences of the words of each admissible weight.

% ________________ CPAs régulières :: représentation
\begin{definition}\label{def-reg_xy}
Let $d > 0$ and $\nu\geq d$ be two integers.
With a $(\nu, d)$-CPA $(N, D)$, we associate the vector $(y, x)\in\mathbb{N}^{\nu+1}\times\mathbb{N}^{d+1}$ where
	for $i \in\set{0, 1 ,\ldots, \nu}$, $y_i$ indicates the number of times the $\nu$-dimensional Boolean words of weight $i$ occur in $N$ and, if $i \leq d$, $x_i$ indicates the number of times the $\nu$-dimensional Boolean words of weight $i$ occur in $D$.
We refer to $(y, x)$ as {\em the representative} vector of $(N, D)$.
\end{definition}
For example, the representative vector of the $(5, 4)$-CPA of \cref{fig-Delta-ex} is given by the vector:
$$\begin{array}{rl}
(y_0, y_1, y_2, y_3, y_4, y_5, x_0, x_1, x_2, x_3, x_4)
	&=(2, 0, 0, 1, 0, 3, 0, 1, 0, 0, 2).
\end{array}$$

\begin{proposition}\label{prop-reg_xy} 
Let $k > 0$, $d\geq k$, $\nu\geq d$ be three integers, and $(y, x)\in\mathbb{N}^{\nu+1}\times\mathbb{N}^{d+1}$. 
Then $(y, x)$ is the representative vector of a (regular) $(\nu, d)$-CPA $(N, D)$ of strength $k$ if and only if $y_\nu > 0$, and:
\begin{align}\label{delta_xy-equ}
	\textstyle		\sum_{i = h}^\nu \binom{\nu-k}{i-h}y_i
	&=\textstyle	\sum_{i = h}^d \binom{\nu-k}{i-h}x_i,		&h\in\set{0, 1 ,\ldots, k}.
\end{align} 

When this occurs, we have 
$$\begin{array}{rl}
R^*(N, D) &=y_\nu,\\
R(N, D)
	&=\sum_{i = 0}^\nu\binom{\nu}{i}y_i	=\sum_{i = 0}^d\binom{\nu}{i}x_i
	 = \left(\sum_{i = 0}^\nu\binom{\nu}{i}y_i+\sum_{i = 0}^d\binom{\nu}{i}x_i\right)/2.
\end{array}$$
\end{proposition}

\begin{proof}
We only argue that $(N, D)$ satisfies $(k_=)$ if and only if $(y, x)$ satisfies \cref{delta_xy-equ} (the rest is trivial). First, we observe that $(N, D)$ satisfies $(k_=)$ if and only if, for all $k$-cardinality subsets $K$ of $[\nu]$ and all $w\in\set{0, 1}^k$, there are as many rows in $N$ and $D$ that coincide with $w$ on their coordinates with index in $K$. Formally, $N$ and $D$ must satisfy:
\begin{align}\label{eq_k=_Delta}
\card{\set{r\in[R(N, D)]:N_r^K = w}}
	&=\card{\set{r\in[R(N, D)]:D_r^K = w}},
		&K\subseteq[\nu],\,|K| = k,\ w\in\set{0, 1}^k.
\end{align}

Given a pair $(K, w)$, where $K\subseteq[\nu]$ and $w\in\set{0, 1}^k$, let $h$ denote the number of $1$'s in $w$. For $i\in\set{0, 1 ,\ldots, \nu}$, the number of words $u\in\set{0, 1}^\nu$ of weight $i$ that satisfy $u_K = w$ is the number of ways of choosing $i-h$ coordinates to set to 1 among those with index in $[\nu]\setminus K$. Since such words occur $y_i$ times in $N$ and, if $i\leq d$, $x_i$ times in $D$, we deduce that there are $\sum_{i = 0}^\nu\binom{\nu-k}{i-h}y_i$ and $\sum_{i = 0}^d\binom{\nu-k}{i-h}x_i$ rows in $N$ and $D$, respectively, that coincide with $w$ at their coordinates with index in $K$. 
\end{proof}

% ____________________________ De Delta rég vers Gamma (jokers)
\section{From cover to alphabet reduction pairs of arrays}
\label{sec-D2G}
% ___________________________________________________ Des CPAs vers les ARPAs (jokers)
\Cref{prop-Gamma_2_Delta} shows how to interpret a $(q, p)$-ARPA of strength $k$ as a $(q, p)$-CPA of the same strength. 
In this section, we address the converse question: can we interpret a $(\nu, d)$-CPA $(N, D)$ of strength $k$ as a $(\nu, d)$-ARPA of strength $k$, provided that it is regular?
Here, we think of $(N, D)$ as a partially defined ARPA: 
	the 1s indicate the entries matching their column index; 
	the remaining entries must be assigned values different from their column index to satisfy $(\Gamma_P)$ and $(\Gamma_=)$. 

We provide a partial answer to this question by showing how a regular $(\nu, d)$-CPA of strength $k$ can be transformed into a $(\nu, d')$-ARPA of strength $k$, where $d'$ is either $d$, $d+1$, or $d+2$, depending on $(N, D)$. 
This is a first step towards generalizing the equality of the numbers $\gamma(q, k, k)$ and $\delta(q, k, k)$ to the case where $p>k$.

% ________________ Des CPAs vers les ARPAs (jokers) :: construction
\subsection{Construction}

%% ____ entrée de l'algorithme
Let $k > 0$, $d\geq k$, and $\nu > d$ be three integers. We consider a regular CPA $(N, D)\in\Delta(\nu, d, k)\setminus\Delta(\nu, d-1, k)$ with no row in common between $N$ and $D$.
By the assumption that $(N, D)\notin\Delta(\nu, d-1, k)$, the words of weight $d$ appear in $D$. Thus, there exists a largest integer $r\in\set{d, d+1 ,\ldots, \nu-1}$ such that the words of weight $r$ appear at least once in $N$ or $D$.

Since $N$ and $D$ have no common rows, the pair $(N, D)$ can be uniquely represented by a vector $z = (z_0, z_1 ,\ldots, z_\nu)\in\mathbb{Z}^{\nu+1}$. 
Specifically, for each $i\in\set{0, 1 ,\ldots, \nu}$, $|z_i|$ indicates the number of times each word of weight $i$ occurs in $(N, D)$; these words appear in $N$ if $z_i < 0$ and in $D$ if $z_i > 0$.
If $(y, x)$ is the representative vector of $(N, D)$, then $z$ is simply defined for $i\in\set{0, 1 ,\ldots, \nu}$ by: 
\begin{align}\label{eq-xy2z}
	z_i	&=\begin{cases}
					-y_i	&\text{if $i >d$}\\
				x_i-y_i	&\text{otherwise}	%if $i\leq d$}\\
			\end{cases}.
\end{align} 

Note that, according to \cref{prop-reg_xy}, $(N, D)$ satisfies $(k_=)$ if and only if $z$ satisfies: 
\begin{align}\label{delta_z-equ}\textstyle
\sum_{i = h}^{\nu-k+h} \binom{\nu-k}{i-h} z_i 	&=0,		&h\in\set{0, 1 ,\ldots, k}.
\end{align}

%% ____ sortie de l'algorithme
%% ?? nllabel vs. label pour les lignes
\begin{algorithm}
\caption{Construction underlying Theorem \ref{thm-reg_joker}\label{alg-D2G}}
\KwIn{A regular CPA in $\Delta(\nu, d, k)\setminus\Delta(\nu, d-1, k)$ 
			encoded by $z = (z_0, z_1 ,\ldots, z_\nu)\in\mathbb{Z}^{\nu +1}$}
\KwOut{An ARPA in $\Gamma(\nu, d', k)$ for some $d'\in\set{d, d+1, d+2}$, 
			encoded by $\tilde{z}:\Sigma_\nu^\nu\rightarrow\mathbb{N}$}
\tcc{Initialization}
\ForEach{$u\in\Sigma_\nu^\nu$}{
	$\tilde{z}(u) \longleftarrow 0$\;
}
$r\longleftarrow \max\set{i\in\set{d, d+1 ,\ldots, \nu-1}\,:\,z_i \neq 0}$\;
\tcc{Rows derived from weight-$\nu$ rows of the original CPA}
$\tilde{z}((0, 1,\ldots, \nu-1)) \longleftarrow z_\nu$\;
\tcc{Rows derived from weight-$r$ rows of the original CPA}
\ForEach{$J\subseteq\Sigma_\nu$ with $|J|=r$}{
	$\tilde{z}(g(J)) \longleftarrow z_r$\;
	\label{alg-g-r}
}
\tcc{Rows derived from original CPA rows of weight $i<r$}
\ForEach{$i\in\set{0, 1 ,\ldots, r-1}$ for which $z_i\neq 0$}{
	\ForEach{$J\subseteq\Sigma_\nu$ with $|J|=i$}{
		\For{$c=0$ \KwTo $c_*(J)$}{
			\If{$c < c_*(J)$}{
				$\tilde{z}(g^c(J)) \longleftarrow \binom{\nu-c-2-i}{r-1-i} / \binom{\nu-1-i}{r-i} \times z_i$;
				\label{alg-g^c}	
			}
			\ElseIf{$c_*(J) < \nu -r$}{
				$\tilde{z}(g(J)) \longleftarrow \binom{\nu-c-1-i}{r-i} / \binom{\nu-1-i}{r-i} \times z_i$;
				\label{alg-g-i}
			}
		}
	}
}
\end{algorithm}

We derive from $(N, D)$ an ARPA $(Q, P)$ of strength $k$ that can be interpreted as $(N, D)$. Similarly to $(N, D)$, we represent $(Q, P)$ by means of a function $\tilde{z}: \Sigma_\nu^\nu\rightarrow\mathbb{N}$,  where, for $u\in\Sigma_\nu^\nu$,  
	$\tilde{z}(u)$ indicates the number of times $u$ occurs as a row in $(Q,P)$, 
	occurring in $Q$ if $\tilde{z}(u)<0$, in $P$ if $\tilde{z}(u)>0$. 
 %
%% ____ algorithme
The construction is described in \cref{alg-D2G}. 
Examples are given in \cref{fig-D2G}.
In this algorithm, for a subset $J$ of $\Sigma_\nu$, $c_*(J)$ refers to the smallest integer in $J\cup\set{\nu-r}$. 
Furthermore, for any $J\subseteq \Sigma_\nu$ and any $c\in\set{0, 1 ,\ldots, c_*(J)-1}$, $g(J)$ and $g^c(J)$ are the words of $\Sigma_\nu^\nu$ defined coordinate-wise for each $j\in\Sigma_\nu$ by:
$$\begin{array}{lcl}
\begin{array}{rl}
	g(J)_j 	&=\begin{cases}
				j		&\text{if $j\in J$}\\
				c_*(J)	&\text{otherwise}\\%% (thus $g(J)_j\neq j$)
			\end{cases},
\end{array}&&\begin{array}{rl}
	g^c(J)_j &=\begin{cases}
				j		&\text{if $j\in J$}\\
				c+1	&\text{if $j\notin J$ and $j\leq c$}\\%% (thus $g^c(J)_j\neq j$)}\\
				c		&\text{if $j\notin J$ and $j >c$}\\%% (thus $g^c(J)_j\neq j$)}
			\end{cases}.
\end{array}
\end{array}$$
We denote by $u(J)$ the indicator vector of $J$, defined for $j\in\Sigma_\nu$ by $u(J)_j = 1$ if $j\in J$ and 0 otherwise. By construction, for every $j\in\Sigma_\nu$, the coordinate of index $j$ in both $g(J)$ and $g^c(J)$ is equal to $j$ if and only if $j\in J$. The Boolean words $\pi_\nu(g(J))$ and $\pi_\nu(g^c(J))$ therefore both coincide with $u(J)$. 
\begin{figure}[t]{\footnotesize
\caption{\label{fig-D2G}
$(\nu, d)$-ARPAs of strength 2 (left) and 3 (right), derived using Algorithm \ref{alg-D2G} from regular $(\nu, d)$-CPAs of the same strength. 
	The gray and black colors indicate entries with values 1 and 0, respectively, in the original CPAs.}
\begin{center}
\begin{tabular}{cc}
$\begin{array}{c}
	(Q, P)\in\Gamma(5, 2, 2)\\[4pt] %% (solution régulière sous-optimale)
	\begin{array}{ccc}\setlength\arraycolsep{1.5pt}
		\begin{array}{ccccc}
			Q^0			&Q^1		&Q^2		&Q^3		&Q^4\\
		\hline
			\equj{0}	&\equj{1}	&\equj{2}	&\equj{3}	&\equj{4}\\
		\hline
			\equj{0}	&\difj{0}	&\difj{0}	&\difj{0}	&\difj{0}\\
			\equj{0}	&\difj{0}	&\difj{0}	&\difj{0}	&\difj{0}\\
			\equj{0}	&\difj{0}	&\difj{0}	&\difj{0}	&\difj{0}\\
			\difj{1}	&\equj{1}	&\difj{0}	&\difj{0}	&\difj{0}\\
			\difj{1}	&\equj{1}	&\difj{1}	&\difj{1}	&\difj{1}\\
			\difj{1}	&\equj{1}	&\difj{1}	&\difj{1}	&\difj{1}\\
			\difj{1}	&\difj{0}	&\equj{2}	&\difj{0}	&\difj{0}\\
			\difj{2}	&\difj{2}	&\equj{2}	&\difj{1}	&\difj{1}\\
			\difj{2}	&\difj{2}	&\equj{2}	&\difj{2}	&\difj{2}\\
			\difj{1}	&\difj{0}	&\difj{0}	&\equj{3}	&\difj{0}\\
			\difj{2}	&\difj{2}	&\difj{1}	&\equj{3}	&\difj{1}\\
			\difj{3}	&\difj{3}	&\difj{3}	&\equj{3}	&\difj{2}\\
			\difj{1}	&\difj{0}	&\difj{0}	&\difj{0}	&\equj{4}\\
			\difj{2}	&\difj{2}	&\difj{1}	&\difj{1}	&\equj{4}\\
			\difj{3}	&\difj{3}	&\difj{3}	&\difj{2}	&\equj{4}\\
		\end{array}&&\setlength\arraycolsep{1.5pt}
		\begin{array}{ccccc}
			P^0			&P^1		&P^2		&P^3		&P^4\\
		\hline
			\equj{0}	&\equj{1}	&\difj{0}	&\difj{0}	&\difj{0}\\
			\equj{0}	&\difj{0}	&\equj{2}	&\difj{0}	&\difj{0}\\
			\equj{0}	&\difj{0}	&\difj{0}	&\equj{3}	&\difj{0}\\
			\equj{0}	&\difj{0}	&\difj{0}	&\difj{0}	&\equj{4}\\
			\difj{1}	&\equj{1}	&\equj{2}	&\difj{1}	&\difj{1}\\
			\difj{1}	&\equj{1}	&\difj{1}	&\equj{3}	&\difj{1}\\
			\difj{1}	&\equj{1}	&\difj{1}	&\difj{1}	&\equj{4}\\
			\difj{2}	&\difj{2}	&\equj{2}	&\equj{3}	&\difj{2}\\
			\difj{2}	&\difj{2}	&\equj{2}	&\difj{2}	&\equj{4}\\
			\difj{3}	&\difj{3}	&\difj{3}	&\equj{3}	&\equj{4}\\
		\hline
			\difj{1}	&\difj{0}	&\difj{0}	&\difj{0}	&\difj{0}\\
			\difj{1}	&\difj{0}	&\difj{0}	&\difj{0}	&\difj{0}\\
			\difj{1}	&\difj{0}	&\difj{0}	&\difj{0}	&\difj{0}\\
			\difj{2}	&\difj{2}	&\difj{1}	&\difj{1}	&\difj{1}\\
			\difj{2}	&\difj{2}	&\difj{1}	&\difj{1}	&\difj{1}\\
			\difj{3}	&\difj{3}	&\difj{3}	&\difj{2}	&\difj{2}\\
		\end{array}
	\end{array}
\end{array}$&$\begin{array}{c}
	(Q, P)\in\Gamma(5, 4, 3)\\[4pt] %% (sol. rég. optimale et minimale)
	\begin{array}{ccc}\setlength\arraycolsep{1.5pt}
		\begin{array}{ccccc}
			Q^0			&Q^1		&Q^2		&Q^3		&Q^4\\
		\hline
			\equj{0}	&\equj{1}	&\equj{2}	&\equj{3}	&\equj{4}\\
			\equj{0}	&\equj{1}	&\equj{2}	&\equj{3}	&\equj{4}\\
			\equj{0}	&\equj{1}	&\equj{2}	&\equj{3}	&\equj{4}\\
		\hline
			\equj{0}	&\equj{1}	&\equj{2}	&\difj{0}	&\difj{0}\\
			\equj{0}	&\equj{1}	&\difj{0}	&\equj{3}	&\difj{0}\\
			\equj{0}	&\equj{1}	&\difj{0}	&\difj{0}	&\equj{4}\\
			\equj{0}	&\difj{0}	&\equj{2}	&\equj{3}	&\difj{0}\\
			\equj{0}	&\difj{0}	&\equj{2}	&\difj{0}	&\equj{4}\\
			\equj{0}	&\difj{0}	&\difj{0}	&\equj{3}	&\equj{4}\\
			\difj{1}	&\equj{1}	&\equj{2}	&\equj{3}	&\difj{0}\\
			\difj{1}	&\equj{1}	&\equj{2}	&\difj{0}	&\equj{4}\\
			\difj{1}	&\equj{1}	&\difj{0}	&\equj{3}	&\equj{4}\\
			\difj{1}	&\difj{0}	&\equj{2}	&\equj{3}	&\equj{4}\\
		\hline
			\difj{1}	&\difj{0}	&\difj{0}	&\difj{0}	&\difj{0}\\
			\difj{1}	&\difj{0}	&\difj{0}	&\difj{0}	&\difj{0}\\
		\end{array}&&\setlength\arraycolsep{1.5pt}
		\begin{array}{ccccc}
			P^0			&P^1		&P^2		&P^3		&P^4\\
		\hline
			\equj{0}	&\equj{1}	&\equj{2}	&\equj{3}	&\difj{0}\\
			\equj{0}	&\equj{1}	&\equj{2}	&\equj{3}	&\difj{0}\\
			\equj{0}	&\equj{1}	&\equj{2}	&\difj{0}	&\equj{4}\\
			\equj{0}	&\equj{1}	&\equj{2}	&\difj{0}	&\equj{4}\\
			\equj{0}	&\equj{1}	&\difj{0}	&\equj{3}	&\equj{4}\\
			\equj{0}	&\equj{1}	&\difj{0}	&\equj{3}	&\equj{4}\\
			\equj{0}	&\difj{0}	&\equj{2}	&\equj{3}	&\equj{4}\\
			\equj{0}	&\difj{0}	&\equj{2}	&\equj{3}	&\equj{4}\\
			\difj{1}	&\equj{1}	&\equj{2}	&\equj{3}	&\equj{4}\\
			\difj{1}	&\equj{1}	&\equj{2}	&\equj{3}	&\equj{4}\\
		\hline
			\equj{0}	&\difj{0}	&\difj{0}	&\difj{0}	&\difj{0}\\
			\difj{1}	&\equj{1}	&\difj{0}	&\difj{0}	&\difj{0}\\
			\difj{1}	&\difj{0}	&\equj{2}	&\difj{0}	&\difj{0}\\
			\difj{1}	&\difj{0}	&\difj{0}	&\equj{3}	&\difj{0}\\
			\difj{1}	&\difj{0}	&\difj{0}	&\difj{0}	&\equj{4}\\
		\end{array}
	\end{array}
\end{array}$
\end{tabular}
\end{center}}
\end{figure}
 %% illustration

In the next section, we show that the resulting array pair $(Q, P)$
can be interpreted as $(N, D)$ (\cref{fact-algo-inter}), 
satisfies $(\Gamma_P)$ with respect to some specific integer $d'$ depending on $(N, D)$ (\cref{fact-algo-G_D}), 
and satisfies $(k_=)$ (\cref{fact-algo-k=}). 
Precisely, we establish the following theorem:

\begin{theorem}\label{thm-reg_joker}
Let $k >0$, $d \geq k$, and $\nu >d$ be three integers, and let $(N, D)$ be a regular CPA of $\Delta(\nu, d, k)\setminus\Delta(\nu, d-1, k)$ such that $N$ and $D$ have no common row. 
We denote by $r$ the largest weight among the words other than $1\ 1\ \cdots\ 1$ that occur in $(N, D)$.  
Let $d'$ be defined by
$$\begin{array}{rl}
d' &=\begin{cases}
		d+2	&\text{if $r >d$,}\\
		d+1	&\text{if $r = d$, and the words of weight $d-1$ occur in $D$,}\\
		d		&\text{otherwise (thus $r = d$ and the rows of $D$ have weight $\neq d-1$).}
	\end{cases}
\end{array}$$
Then, from $(N, D)$, we can derive a $(\nu, d')$-ARPA $(Q, P)$ of strength $k$ that can be interpreted as $(N, D)$.
\end{theorem}

% ________________ Des CPAs vers les ARPAs (jokers) :: preuve
\subsection{Proof of \cref{thm-reg_joker}}

\begin{fact}\label{fact-algo-inter}
$(Q, P)$ can be interpreted as $(N, D)$. 
\end{fact}

\begin{proof}
We must show that, for any $J\subseteq\Sigma_\nu$, the word $u(J)$ occurs in $\pi_\nu(Q)$ and in $\pi_\nu(P)$ the same number of times as in $N$ and $D$. 
Let $i = |J|$. Equivalently, we must show that the number of rows of the form $g(J)$ or $g^c(J)$, for $c\in\set{0, 1 ,\ldots, c_*(J)-1}$, in $Q$ and $P$ equals the number of rows $u(J)$ in $N$ and $D$, respectively. 
If $i\in\set{\nu, r}$ or $z_i = 0$, the claim is trivial. Thus assume that $i < r$ and $z_i\neq 0$.
Our goal is to prove the equality $\sum_{c = 0}^{c_*(J)-1}\tilde{z}(g^c(J))+\tilde{z}(g(J)) = z_i$.
Now, we have:
$$\begin{array}{rl}
	\sum_{c = 0}^{c_*(J)-1}\tilde{z}(g^c(J))+\tilde{z}(g(J))
		=\frac{z_i}{\binom{\nu-1-i}{r-i}}\!\times\!\left(
			\sum_{c = 0}^{c_*(J)-1} \binom{\nu-c-2-i}{r-1-i}+\binom{\nu-c_*(J)-1-i}{r-i}
		\right)
\end{array}.$$

By using the identity $\binom{\nu-c-2-i}{r-1-i} = \binom{\nu-c-1-i}{r-i}-\binom{\nu-c-2-i}{r-i}$, we can simplify this expression to obtain the desired equality. 
\end{proof}

\begin{fact}\label{fact-algo-G_D}
$(Q, P)$ satisfies $(\Gamma_P)$ with respect to the integer $d'$ as defined in \cref{thm-reg_joker}.
\end{fact}

\begin{proof}
By construction, $P$ contains words of the form either $g(J)$ or $g^c(J)$, where $J$ is a subset of $\Sigma_\nu$.
 %
%% ____ g(J)
The coordinates of a word $g(J)$ are drawn from the set $J\cup\set{c_*(J)}$. 
On the one hand, if \cref{alg-D2G} considers such a word, then either $|J| = r$, or $|J| < r$ and $c_*(J) < \nu-r$. In both cases, we have $c_*(J)\in J$. 
On the other hand, a word $g(J)$ may not occur in $P$ unless $z_{|J|} >0$, while $z_{|J|}$ may not be positive unless $|J|\leq d$. The number of pairwise distinct values taken by the coordinates of the words $g(J)$ occurring in $P$ is thus bounded above by $d$. 

%% ____ g^c(J)
A word $g^c(J)$ may not occur in $P$ unless $z_{|J|} >0$, $|J|\leq d$ and $|J| <r$.
The maximum size of a subset $J$ for which a word $g^c(J)$ occurs in $P$ is therefore 
	$d$ if $r >d$, 
	$d-1$ if $r = d$ and $z_{d-1} > 0$, 
	and at most $d-2$ otherwise (thus $r = d$ and $z_{d-1} = 0$). 
The conclusion follows from the fact that the coordinates of $g^c(J)$ come from the set $J\cup\set{c, c+1}$.
\end{proof}

The third ingredient of the proof of \cref{thm-reg_joker} is based on the following technical lemma:
%
%% z_\nu <0 :: pour que r soit bien défini
\begin{lemma}\label{lem-f_z}	%% _______________ f_z = 0
Let $k > 0$ and $\nu > k$ be two integers, and $z = (z_0, z_1 ,\ldots, z_\nu)$ be a sequence of integers  satisfying \cref{delta_z-equ} with $z_\nu <0$. 
We denote by $r$ the largest integer in $\set{k, k+1 ,\ldots, \nu-1}$ for which $z_r\neq 0$. 
Then, for all natural numbers $h$ and $\lambda$ such that $h+\lambda <k$ and $c\leq\nu-r$, we have:
\begin{align}\label{eq-f=0}
f_z(h, \lambda, c)	&:=\textstyle
	\sum_{i = h}^r
		\binom{\nu-c-h-\lambda}{i-h}\binom{\nu-c-i}{r-i}/\binom{\nu-1-i}{r-i} 
		\times z_i
	&=0.
\end{align}
\end{lemma}
Essentially, we prove that for such parameter sets $(h, \lambda, c)$, the quantities $f_z(h, \lambda, c)$ satisfy the following recurrence relation and the initial condition
\begin{align}\label{eq-f-rec}
f_z(h, \lambda, c)
&=\begin{cases}
	f_z(h, \lambda-1, c)-f_z(h+1, \lambda-1, c)	&\text{if $\lambda >0$},	\\
	0												&\text{otherwise}.
\end{cases}
\end{align}
The detailed proof of this lemma is given in \cref{appendix:sec-fact_k=}. 
Now we state the third and final fact that proves \cref{thm-reg_joker}:

\begin{fact}\label{fact-algo-k=}
$(Q, P)$ satisfies $(k_=)$. 
\end{fact}

\begin{proof} 
$(Q, P)$ satisfies $(k_=)$ if and only if, given any $k$ symbols $j_1, j_2 ,\ldots, j_k\in\Sigma_\nu$, for all $w\in\Sigma_\nu^k$, the number of rows in $P$ that coincide with $w$ on their coordinates of index $(j_1, j_2 ,\ldots, j_k)$ equals the number of such rows in $Q$. Equivalently, $\tilde{z}$ must satisfy the following relation:
\begin{align}\label{gamma_z-equ}\textstyle
\sum_{u = (u_0, u_1 ,\ldots, u_{\nu-1})\in\Sigma_\nu^\nu: u_K = w} \tilde{z}(u)	&=0,
	&K\subseteq\Sigma_\nu,\,\card{K} = k,\ w\in\Sigma_\nu^k.
\end{align}

%% _______________ paramètres à considérer
Consider a set $K = \set{j_1, j_2 ,\ldots, j_k}$ of $k$ pairwise distinct symbols from $\Sigma_\nu$, and a word $w$ from $\Sigma_\nu^k$. Since $(Q, P)$ can be interpreted as $(N, D)$, which satisfies $(k_=)$, we already know that $\tilde{z}$ satisfies \cref{gamma_z-equ} for $(K, w)$ when $w$ matches the indices in $K$. So it remains to show that $\tilde{z}$ satisfies \cref{gamma_z-equ} at $(K, w)$ for all $w\in\Sigma_\nu^k\setminus\set{(j_1, j_2 ,\ldots, j_k)}$. 
Given $(K, w)$, we partition the set $K$ based on its relation to $w$. Specifically, we define $H$ as the set of indices $j_i\in K$ such that $w_i = j_i$, and $L$, its complement in $K$. 
Proving that \cref{gamma_z-equ} holds at $(K, w)$ for all $w\in\Sigma_\nu^k\setminus\set{(j_1, j_2 ,\ldots, j_k)}$ is equivalent to showing that, for any partition $(H, L)$ of $K$ where $|H| = h <k$ and $L = \set{\ell_1, \ell_2 ,\ldots, \ell_{k-h}}$, $\tilde{z}$ satisfies the following identity for all $v = (v_1, v_2, \ldots, v_{k-h})\in\Sigma_\nu^{k-h}$ such that $v_s\neq\ell_s$ for $s\in[k-h]$: 
\begin{align}%\label{e-def-R_HLv}	%
R(H, L, v)\textstyle:=\sum_{u = (u_0, u_1 ,\ldots, u_{\nu-1})\in\Sigma_\nu^\nu: 
	u_H = H \wedge u_L = v} \tilde{z}(u) &=0.
\end{align}

We order the elements of $L$ so that $\ell_1 < \ell_2 <\cdots < \ell_{k-h}$. By construction, the words $u\in\Sigma_\nu^\nu$ for which $\tilde{z}(u)$ is assigned a non-zero value are of the form $g(J)$ or $g^c(J)$ where $J\subseteq\Sigma_\nu$, $z_{|J|}\neq 0$, and $c$ is a natural number less than $c_*(J)$. 
In particular, no row $u$ in $(Q, P)$ can satisfy $(u_{\ell_1}, u_{\ell_2} ,\ldots, u_{\ell_{k-h}})=(v_1, v_2 ,\ldots, v_{k-h})$ unless there exist two natural numbers $s < k-h$ and $c\leq\nu-r$ such that $c\leq c_*(H)$ and
$$\begin{array}{rll}
v_1 = v_2 =\cdots= v_s
	&= c + 1 	&\notin\set{\ell_1, \ell_2 ,\ldots, \ell_s},						\\ 
v_{s+1} = v_{s+2} =\cdots= v_{k-h}
	&= c			&\notin\set{\ell_{s+1}, \ell_{s+2} ,\ldots, \ell_{k-h}}.
\end{array}$$ 
 %	
%% _______________ quantités intervenant dans R(H, L, v)
For such a triple $(H, L, v)$, the contribution of \cref{alg-D2G} to $R(H, L, v)$ is as follows:
\begin{itemize}
	\item\label{it-g-r} if $s = 0$, by line \ref{alg-g-r}, setting $\tilde{z}(g(J))$ to $z_r$ for each subset $J$ of size $r$ of $\Sigma_\nu$ such that $c_*(J) = c$, $H\subseteq J$, and $J\cap L = \emptyset$. Recall that $c_*(J)$ is defined to be the smallest integer in $J\cup\set{\nu-r}$, which ensures $J\subseteq\set{c_*(J), c_*(J)+1 ,\ldots, \nu-1}$. Moreover, since $|J| = r$, this integer belongs to $J$. Thus, $J$ should more precisely satisfy:
\begin{align}\label{eq-cond_J-g}
	H\cup\set{c}\subseteq J\subseteq \Sigma_\nu\setminus(L\cup\Sigma_c).
\end{align}
	\item\label{it-g-i} If $s = 0$ and $c < \nu-r$, by line \ref{alg-g-i}, setting $\tilde{z}(g(J))$ to $\binom{\nu-c-1-i}{r-i}/\binom{\nu-1-i}{r-i}\times z_i$ for each $i\in\set{0, 1 ,\ldots r-1}$ such that $z_i\neq 0$, and each $J\subseteq\Sigma_\nu$ such that $|J| = i$, $c_*(J) = c$, $H\subseteq J$, and $J\cap L = \emptyset$. Note that the integer $c_*(J)$, if it is less than $\nu-r$, necessarily belongs to $J$. Thus, again, the considered subsets $J$ of $\Sigma_\nu$ must satisfy \cref{eq-cond_J-g}.
	\item\label{it-g^c} If $c <c_*(H)$ (in which case $c\notin H$ and $c <\nu-r$), by line \ref{alg-g^c}, setting $\tilde{z}(g^c(J))$ to $\binom{\nu-c-2-i}{r-1-i}/\binom{\nu-1-i}{r-i}\times z_i$ for each $i\in\set{0, 1 ,\ldots r-1}$ such that $z_i\neq 0$, and each subset $J$ of size $i$ of $\Sigma_\nu$ such that $c_*(J) >c$ (implying $J\subseteq\set{c+1, c+2 ,\ldots, \nu-1}$), $H\subseteq J$, and $J\cap L = \emptyset$. 
In this case, $\ell_{s+1} > c$ and, if $s >0$, $\ell_s\leq c$. 
To summarize, the subsets $J$ of $\Sigma_\nu$ considered for this case satisfy $|J| <r$, and: 
\begin{align}\label{eq-cond_J-g^c}
	H\subseteq J\subseteq
		\Sigma_\nu\setminus(\set{\ell_{s+1}, \ell_{s+2} ,\ldots, \ell_{k-h}}\cup\Sigma_{c+1}). 
\end{align}
	\item\label{it-g^{c-1}} If $s = 0$ and $\ell_{k-h} <c$, by line \ref{alg-g^c}, setting $\tilde{z}(g^{c-1}(J))$ to $\binom{\nu-c-1-i}{r-1-i}/\binom{\nu-1-i}{r-i}\times z_i$ for each $i\in\set{0, 1 ,\ldots, r-1}$ such that $z_i\neq 0$, and each $J\subseteq\Sigma_\nu$ such that $|J| = i$, $c_*(J)\geq c$ (implying $J\subseteq\Sigma_\nu\setminus\Sigma_c$), $H\subseteq J$, and $J\cap L = \emptyset$. Since $L\subseteq\Sigma_c$, the subsets $J$ of $\Sigma_\nu$ to consider in this case are those satisfying $|J| <r$, and: 
\begin{align}\label{eq-cond_J-g^{c-1}}
	H\subseteq J\subseteq \Sigma_\nu\setminus\Sigma_c.
\end{align}
\end{itemize}

We denote by $\lambda$ the cardinality of the set $\set{c+1, c+2 ,\ldots, \nu-1}\cap L$. In particular, this cardinality equals the positive integer $k-h-s$ in the third case, and $0$ in the fourth. 
 %
%% ____ contribution des g(J)
For $i\in\set{0, 1 ,\ldots, r}$, the number of subsets $J$ of size $i$ of $\Sigma_\nu$ satisfying \cref{eq-cond_J-g} is $\binom{\nu-c-h-\lambda}{i-h}$ if $c\in H$ and $\binom{\nu-c-1-h-\lambda}{i-1-h}$ otherwise. 
If $s = 0$, it follows that the contribution of the lines \ref{alg-g-r} and \ref{alg-g-i} of \cref{alg-D2G} to $R(H, L, v)$ is given by the expression \cref{g(J)-c_in_H} if $c\in H$, \cref{g(J)-c_notin_H} if $c\notin H$, below:
\begin{align}
\label{g(J)-c_in_H}	
\textstyle
	\binom{\nu-c-h-\lambda}{r-h}		\times 	z_r		&\textstyle
	+\sum_{i = h}^{r-1}\binom{\nu-c-h-\lambda}{i-h}\times\binom{\nu-c-1-i}{r-i}/\binom{\nu-1-i}{r-i} \times z_i,
\\\label{g(J)-c_notin_H}	
\textstyle
	\binom{\nu-c-1-h-\lambda}{r-1-h}	\times	z_r		&\textstyle
	+\sum_{i = h+1}^{r-1}\binom{\nu-c-1-h-\lambda}{i-1-h}\times\binom{\nu-c-1-i}{r-i}/\binom{\nu-1-i}{r-i} \times z_i.
\end{align}

Otherwise, these lines do not contribute to $R(H, L, v)$.
For $i\in\set{0, 1 ,\ldots, r-1}$, the number of subsets $J$ of size $i$ of $\Sigma_\nu$ satisfying \cref{eq-cond_J-g^c} and \cref{eq-cond_J-g^{c-1}} is $\binom{\nu-c-1-h-\lambda}{i-h}$ and $\binom{\nu-c-h}{i-h}$, respectively.
We define:
\begin{align}
\label{g^c(J)}				
&\textstyle
	\ \ \sum_{i = h}^{r-1}\binom{\nu-c-1-h-\lambda}{i-h} \times\binom{\nu-c-2-i}{r-1-i}/\binom{\nu-1-i}{r-i} \times z_i,\\
\label{g^{c-1}(J)}		
&\textstyle
	\ \ \sum_{i = h}^{r-1}\binom{\nu-c-h}{i-h} \times\binom{\nu-c-1-i}{r-1-i}/\binom{\nu-1-i}{r-i} \times z_i.
\end{align}
Therefore, the contribution of line \ref{alg-g^c} in the algorithm to $R(H, L, v)$ is equal to:
%% NB c\notin H sachant c\leq c_*(H) = >c <c_*(H) (mais c'est qd même+explicite de mettre $c <\nu-r$)
\begin{itemize}
\item\cref{g^c(J)} if $c\notin H$, $c <\nu-r$, and 
		either $s >0$ and $\ell_s \leq c <\ell_{s+1}$,
		or $s = 0$ and $c <\ell_1$;
\item\cref{g^{c-1}(J)} if $s = 0$ and $\ell_{k-h} <c$.
\end{itemize}
Otherwise, the contribution is 0.

% ____ Synthèse
Hence, whenever $R(H, L, v)$ is not trivially zero (i.e., if $\tilde{z}(u)\neq 0$ for at least one $u\in\Sigma_\nu^\nu$ such that $u_H = H$ and $u_L = v$), it is the sum of one or more of the quantities \cref{g(J)-c_notin_H,g(J)-c_in_H,g^c(J),g^{c-1}(J)}. 
The possible cases for the evaluation of $R(H, L, v)$ are listed below. In each case, we express $R(H, L, v)$ as a sum of the quantities \cref{g(J)-c_notin_H,g(J)-c_in_H,g^c(J),g^{c-1}(J)}: 
\begin{enumerate}%%TODO%%[C{a}se 1.]
\item\label{g-c_in_H+g^{c-1}}		%% pour $f_z$ il faut $c <=\nu-r$, $\lambda = 0$ et $h <k$
$s = 0$, $c\in H$, and $\ell_{k-h} <c$ (thus $\lambda = 0$):
						$R(H, L, v) = \cref{g(J)-c_in_H}+\cref{g^{c-1}(J)}$.
\item\label{g-c_in_H}				%% pour $f_z$ il faut $c <\nu-r$, $\lambda <=k-h$ et $h <k$
%% NB trivialement nul en $c = \nu-r$ sachant $\lambda >0$; 
$s = 0$, $c\in H$, $c <\ell_{k-h}$ (thus $\lambda\neq 0$),
	 and $c <\nu-r$:	$R(H, L, v) = \cref{g(J)-c_in_H}$.
\item\label{g+g^c}					%% pour $f_z$ il faut $c <\nu-r$, $\lambda = k-h$ et $h <k$
$s = 0$, $c\notin H$, $c <\ell_1$ (thus $\lambda = k-h$), 
	and $c <\nu-r$:	$R(H, L, v) = \cref{g(J)-c_notin_H}+\cref{g^c(J)}$.
\item\label{g-c_notin_H+g^{c-1}} 	%% pour $f_z$ il faut $c <=\nu-r$, $\lambda = 0$ et $h <k$
$s = 0$, $c\notin H$, and $\ell_{k-h} <c$ (thus $\lambda = 0$):
						$R(H, L, v) = \cref{g(J)-c_notin_H}+\cref{g^{c-1}(J)}$.	
\item\label{g-c_notin_H} 			%% pour $f_z$ il faut $c <\nu-r$ et $\lambda <k-h$
%% NB trivialement nul en $c = \nu-r$ sachant $\lambda >0$
$s = 0$, $c\notin H$, $\ell_1 <c <\ell_{k-h}$ (thus $\lambda\neq 0, k-h$), 
	and $c <\nu-r$: 	$R(H, L, v) = \cref{g(J)-c_notin_H}$.
\item\label{g^c}					%% pour $f_z$ il faut $c <\nu-r$ et $\lambda <k-h$
$s >0$, $c\notin H$, $\ell_s \leq c <\ell_{s+1}$ (thus $\lambda = k-h-s\neq 0, k-h$), 
	and $c <\nu-r$: 	$R(H, L, v) = \cref{g^c(J)}$.
\end{enumerate}
Note that the expressions \cref{g(J)-c_notin_H,g(J)-c_in_H} are trivially zero if $c = \nu-r$ and $\lambda >0$. This explains why we do not consider the value $\nu-r$ for $c$ in cases \ref{g-c_in_H} and \ref{g-c_notin_H}.

%% _______________ R(H, L, v) as f_z
For each of the six cases identified, $R(H, L, v)$ can be expressed as a linear combination of numbers $f_z(h', \lambda', c')$ from \cref{lem-f_z}. We report these expressions in \cref{tab-fact6-valf}. Their proofs can be found in \cref{appendix:sec-fact_k=}. 
All the terms $f_z(h', \lambda', c')$ occurring in these expressions satisfy $h'+\lambda' <k$ and $c'\leq\nu-r$. This implies, by \cref{lem-f_z}, that for all of the 6 cases identified, $R(H, L, v)$ is equal to 0: the proof is complete.
\end{proof}
\begin{table}[t]{\footnotesize
\caption{The expression of $R(H, L, v)$ as a function of numbers $f_z(h', \lambda', c')$ of  \cref{lem-f_z}, all with parameters $(h', \lambda', c')$ such that $h'+\lambda' <k$ and $c'\leq \nu-r$, for the six cases to consider. Recall that the natural numbers $h$, $c$, and $\lambda$ always satisfy $h <k$, $\lambda\leq k-h$, and $c\leq\nu-r$.}
\label{tab-fact6-valf}
\begin{center}
\begin{tabular}{l|l}
Case	&Value of $R(H, L, v)$\\[4pt]\hline
\ref{g-c_in_H+g^{c-1}}.		%&$\cref{g^c(J)}+\cref{g^{c-1}(J)}$
	&$\!\!\!f_z(h, 0, c)$					\\
[4pt]
\ref{g-c_in_H}.				%&$\cref{g(J)-c_in_H}$
	&$\!\!\!f_z(h, \lambda-1, c+1)$,		where $\lambda >0$ and $c <\nu-r$\\
[5pt]
\ref{g+g^c}.				%&$\cref{g(J)-c_notin_H}+\cref{g^c(J)}$ 
	&$\!\!\!\left\{\begin{array}{ll}
		f_z(h, k-h-1, \nu-r)							&\text{if $c = \nu-r-1$}\\
		f_z(h, k-h-1, c+1)-f_z(h, k-h-1, c+2)	&\text{otherwise}
	\end{array}\right.$,					where $c <\nu-r$\\
[12pt]
\ref{g-c_notin_H+g^{c-1}}.	%&$\cref{g^c(J)}+\cref{g^{c-1}(J)}$
	&$\!\!\!\left\{\begin{array}{ll}
		f_z(h, 0, \nu-r)								&\text{if $c = \nu-r$}\\
		f_z(h, 0, c)-f_z(h, 0, c+1)					&\text{otherwise (thus $c <\nu-r$)}
	\end{array}\right.$						\\
[11pt]
\ref{g-c_notin_H}.			%&$\cref{g(J)-c_notin_H}$
	&$\!\!\!f_z(h+1, \lambda-1, c+1)$, 	where $0 <\lambda <k-h$ and $c <\nu-r$\\
[5pt]
\ref{g^c}.					%&$\cref{g^c(J)}$
	&$\!\!\!\left\{\begin{array}{ll}
		f_z(h, \lambda, \nu-r)							&\text{if $c = \nu-r-1$}\\
		f_z(h, \lambda, c+1)-f_z(h, \lambda-1, c+2)	&\text{otherwise}
	\end{array}\right.$, 					where $0 <\lambda <k-h$ and $c <\nu-r$\\
\end{tabular}
\end{center}}
\end{table}

Together, \cref{fact-algo-inter,fact-algo-G_D,fact-algo-k=} establish \cref{thm-reg_joker}. 
Note that the proposed construction easily extends to pairs $(N, D)$ of arrays that share a collection $M$ of rows in common. Basically, it suffices to replace each row $u = (u_0, u_1 ,\ldots, u_{\nu-1})$ of $M$ in both arrays by the row $g(J(u))$. The rest of the transformation remains unchanged. However, we are hardly interested in such pairs of arrays whose shared rows can be regarded as superfluous.

% ____________________________ Delta rég opt par PL
\section{Characterizing optimal regular CPAs through linear programming}
\label{sec-LP}
Given three integers $k > 0$, $d\geq k$, and $\nu > d$, we can deduce from \cref{thm-reg_joker} that $\gamma(\nu, d +2, k)\geq\delta(\nu, d, k)$. 
Our focus is on understanding how close the integer $d'$, as defined in \cref{thm-reg_joker}, can be to $d$. 
To investigate this proximity, we analyze the regular $(\nu, d)$-CPAs of a given strength $k$ that realize $\delta(\nu, d, k)$.

% ___________________________________________________ Delta opt (par PL)
According to \cref{prop-reg_delta,prop-reg_xy}, $\delta(\nu, d, k)$ is the value of the following mathematical program $P_{\nu, d, k}$ in integer variables:
$$P_{\nu, d, k}:=\ \left\{\begin{array}{rrlll}
\multicolumn{2}{l}{
	\max\ 2y_\nu/\left(\sum_{i =0}^\nu\binom{\nu}{i}y_i +\sum_{i =0}^d\binom{\nu}{i}x_i\right)
} &\\s.t.
	&-\sum_{i =h}^{\nu -k +h}\binom{\nu -k}{i -h}y_i +\sum_{i =h}^d \binom{\nu -k}{i -h}x_i
			&= 0,		&h\in\set{0, 1 ,\ldots, k}	&\cref{delta_xy-equ}\\
	&y_\nu	&> 0		&\\
	&y_0, y_1 ,\ldots, y_\nu, x_0, x_1 ,\ldots, x_d 	&\in\mathbb{N}
\end{array}\right..$$

We observe that $P_{\nu, d, k}$ admits feasible solutions for any choice of parameters $\nu, d, k$ such that $\nu\geq d\geq k >0$.
First, we know from \cref{thm-gamma_qpk} that $\Gamma(\nu, d, k)\neq\emptyset$.
According to \cref{prop-Gamma_2_Delta}, this implies $\Delta(\nu, d, k)\neq\emptyset$.
Finally, the proof of \cref{prop-reg_delta} shows how to derive a regular $(\nu, d)$-CPA of strength $k$ from an element of $\Delta(\nu, d, k)$ that may be non-regular.

We consider the continuous relaxation of $P_{\nu, d, k}$. 
For this new problem, we observe that, for any positive real $\lambda$, two solution vectors $(y, x)\in\mathbb{R}^{\nu +1}\times\mathbb{R}^{d +1}$ and $(\lambda y, \lambda x)$ are either both feasible or both infeasible. 
Moreover, two such vectors yield the same objective value. Accordingly, setting the value of $y_\nu$ to any positive value leaves the set of values of the feasible solutions unchanged. 
These considerations suggest introducing the following linear program $LP_{\nu, d, k}$ in continuous variables:
$$LP_{\nu, d, k}:=\ \ \left\{\begin{array}{rrlll}
\multicolumn{2}{l}{\min \sum_{i =0}^{\nu -1}\binom{\nu}{i}y_i +\sum_{i =0}^d\binom{\nu}{i}x_i}		&\\s.t.	&-\sum_{i =k}^{\nu -1} \binom{\nu -k}{i -k} y_i 
		 +\sum_{i =k}^d \binom{\nu -k}{i -k} x_i	&= 1	&(c_k)\\
		&-\sum_{i =h}^{\nu -k +h} \binom{\nu -k}{i -h} y_i 
		 +\sum_{i =h}^d \binom{\nu -k}{i -h} x_i	&= 0	&(c_h), &h\in\set{0, 1 ,\ldots, k-1}\\
	&y_0, y_1 ,\ldots, y_{\nu -1}, x_0, x_1 ,\ldots, x_d	&\geq 0		&
\end{array}\right..$$

% ________________ Delta opt :: PL
We observe that $LP_{\nu, d, k}$ admits optimal solutions for any parameter set $(\nu, d, k)$ such that $\nu\geq d\geq k >0$. On the one hand, we can derive from a feasible solution $(y, x)$ of $P_{\nu, d, k}$ (whose existence we know) the feasible solution 
\begin{align}\label{eq-sol_LP}
(y_0/y_\nu, y_1/y_\nu ,\ldots, y_{\nu -1}/y_\nu, x_0/y_\nu, x_1/y_\nu ,\ldots, x_d/y_\nu)
\end{align}
of $LP_{\nu, d, k}$: thus the domain of $LP_{\nu, d, k}$ is not empty. 
On the other hand, the objective function of $LP_{\nu, d, k}$, which we want to minimize, is bounded below by 0 over this domain.

This program can be considered equivalent to $P_{\nu, d, k}$ in the following sense: from an optimal solution of $LP_{\nu, d, k}$ with value $v$, we can deduce an optimal solution of $P_{\nu, d, k}$---and thus, a regular $(\nu, d)$-CPA of strength $k$---with value $2/(v +1)$, and vice versa.

\begin{proposition}\label{prop-Delta-PL}
For all integers $k > 0$, $d\geq k$, and $\nu > d$, if $\mathrm{opt}(LP_{\nu, d, k})$ denotes the optimum value of $LP_{\nu, d, k}$, then 
	$\delta(\nu, d, k) =2/\left(\mathrm{opt}(LP_{\nu, d, k}) +1\right)$.
\end{proposition}

\begin{proof}
We first show that $\mathrm{opt}(LP_{\nu, d, k})\leq 2/\delta(\nu, d, k) -1$.
Let $(N, D)\in\Delta(\nu, d, k)$ be a regular CPA with $(y, x)$ as the representative vector. 
It follows from \cref{prop-reg_xy} that $(y, x)$ is a feasible solution of $P_{\nu, d, k}$, with value $R^*(N, D)/R(N, D)$. 
We consider the scaled vector $(y', x')$ of $\mathbb{R}^\nu\times\mathbb{R}^{d +1}$ defined by \cref{eq-sol_LP}. By construction, $(y', x')$ is a feasible solution of $LP_{\nu, d, k}$, with value:
$$\begin{array}{rll}%
	1/y_\nu\times\left(
		\sum_{i =0}^{\nu -1}\binom{\nu}{i}y_i +\sum_{i =0}^d\binom{\nu}{i}x_i
	\right)
	&\displaystyle= (2R(N, D) -y_\nu)/y_\nu
	&= 2R(N, D)/R^*(N, D) -1.
\end{array}$$

We then deduce from \cref{prop-reg_delta} that $\mathrm{opt}(LP_{\nu, d, k})\leq 2/\delta(\nu, d, k) -1$.

Conversely, we show that $\delta(\nu, d, k)\geq 2/\left(\mathrm{opt}(LP_{\nu, d, k}) +1\right)$.
Let $(y^*, x^*)$ be an optimal solution of $LP_{\nu, d, k}$. 
We can assume that $(y^*, x^*)$ is an extreme point of $LP_{\nu, d, k}$ and thus, that the coordinates of $y^*$ and $x^*$ are rational numbers. 
So there exists a positive integer $R^*$ such that $(R^* y^*, R^* x^*)$ has integer coordinates. 
Since $(y^*, x^*)$ satisfies the constraints $(c_h)$, $h\in\set{0, 1 ,\ldots, k}$ of $LP_{\nu, d, k}$, 
	$(R^* y^*, R^* x^*)$ satisfies constraints \cref{delta_xy-equ} of $P_{\nu, d, k}$. 
We thus know from \cref{prop-reg_xy} that $(R^* y, R^*, R^* x)$ is the representative vector of a regular $(\nu, d)$-CPA $(N, D)$ of strength $k$ such that
$R^*(N, D) =R^*$ and 
$2R(N, D) =\sum_{i =0}^{\nu -1}\binom{\nu}{i} R^* y^*_i +R^* +\sum_{i =0}^d\binom{\nu}{i}R^* x^*_i$.
This CPA therefore satisfies:
$$\begin{array}{rll}
\displaystyle\frac{R^*(N, D)}{R(N, D)}
&\displaystyle= R^*\times\frac{2}{\sum_{i =0}^{\nu -1}\binom{\nu}{i} R^* y^*_i +R^* 
						+\sum_{i =0}^d\binom{\nu}{i}R^* x^*_i}
&\displaystyle= \frac{2}{\mathrm{opt}(LP_{\nu, d, k}) +1}.
\end{array}$$

We deduce that $\delta(\nu, d, k)\geq 2/\left(\mathrm{opt}(LP_{\nu, d, k}) +1\right)$.
\end{proof}

% ________________ Delta opt :: Base
Thereafter, we represent the bases of $LP_{\nu, d, k}$ by means of two sets $Y\subseteq\set{0, 1 ,\ldots, \nu -1}$ and $X\subseteq\set{0, 1 ,\ldots, d}$ of word weights that identify the set of basic variables $\set{y_i\,|\,i \in Y}\cup\set{x_i\,|\,i \in X}$. 

\begin{proposition}\label{prop-Delta_PL-Base}
Let $k >0$, $d \geq k$, $\nu >d$ be three integers, $Y$ be a subset of $\set{0, 1 ,\ldots, \nu -1}$, and $X$ be a subset of $\set{0, 1 ,\ldots, d}$. 
Then $(Y, X)$ is a basis of $LP_{\nu, d, k}$ if and only if $Y\cap X =\emptyset$ and $|Y \cup X| =k +1$.
\end{proposition}

\begin{proof}
For $i\in\set{0, 1 ,\ldots, d}$, it is clear that two variables $y_i$ and $x_i$ cannot both be basic variables. We represent the constraint matrix of $LP_{\nu, d, k}$ by the $(k +1)\times(\nu +d +1)$ matrix $M$ where:
\begin{itemize}
	\item for $h\in\set{0, 1 ,\ldots, k}$, $M_{h, 0}, M_{h, 1} ,\ldots, M_{h, \nu +d}$ are the coefficients associated with the constraint $(c_h)$;
	\item for $i\in\set{0, 1 ,\ldots, \nu +d}$, $M_{0, i}, M_{1, i} ,\ldots, M_{k, i}$ are the coefficients associated with the variable $y_i$ if $i <\nu$, $x_{i -\nu}$ otherwise.
\end{itemize}
Accordingly, $M$ is defined for $(h, i)\in\set{0, 1 ,\ldots, k}\times \set{0, 1 ,\ldots, \nu +d}$ by
\begin{align}\nonumber
\begin{array}{rll}
M_{h, i}	&=\begin{cases}
						-\binom{\nu -k}{i -h}		&\text{if $i <\nu$},		\\				%y_i
						\binom{\nu -k}{i -\nu -h}	&\text{otherwise}.	\\ %if $i\geq\nu$	%x_{i -\nu}
					\end{cases}
\end{array}
\end{align}

Let $\lambda_0, \lambda_1 ,\ldots, \lambda_k$ be $k+1$ reals such that $\sum_{h =0}^k \lambda_h M_{h, i} =0$, $i\in\set{0, 1 ,\ldots, \nu +d}$. We can successively prove by induction on the index $i$ that $\lambda_0, \lambda_1, \ldots, \lambda_k$ are all zero, implying that the matrix $M$ has full rank.
\end{proof}

% ________________ Delta opt :: Solution de base
Now that we have characterized the bases of $LP_{\nu, d, k}$, we derive the expression of its basic solutions. 

%% NB ici la généralisation à $d\leq\nu$ ne coûterait rien
\begin{proposition}\label{prop-Delta_PL-SB}
Let $k > 0$, $d\geq k$, and $\nu > d$ be three integers, and $(Y, X)$ be a basis of $LP_{\nu, d, k}$. Then the basic variables of the corresponding basic solution are defined by (the other variables are all zero):
\begin{align}\label{eq-SB_xy}
\left\{\begin{array}{rrll}
y_i	&=-\prod_{a\in(Y\cup X)\setminus\set{i}}\frac{\nu -a}{i -a}/\binom{\nu}{i},	&i\in Y\\
x_i	&= \prod_{a\in(Y\cup X)\setminus\set{i}}\frac{\nu -a}{i -a}/\binom{\nu}{i},	&i\in X\\
\end{array}\right..
\end{align}
\end{proposition}

\begin{proof}
By \cref{prop-Delta_PL-Base}, $Y\cap X =\emptyset$.
We define $Z :=Y\cup X\cup\set{\nu}$, and consider the $(\nu +1)$-dimensional vector $z$ (analogous to that in \cref{eq-xy2z}) whose components for $i\in\set{0, 1 ,\ldots, \nu}$ are defined by
\begin{align}\label{eq-SB_z}
z_i	&=\begin{cases}
		\prod_{a\in Z\setminus\set{i,\nu}}\frac{\nu -a}{i -a}/\binom{\nu}{i}	
			&\text{if $i\in Z\setminus\set{\nu}$},	\\
		-1 	&\text{if $i =\nu$},			\\
		0		&\text{otherwise}.
	\end{cases}
\end{align}

Our goal is to show that the proposed solution $(y, x)$ satisfies the constraints $(c_0) ,\ldots, (c_{k-1}), (c_k)$ of $LP_{\nu, d, k}$, which is equivalent to showing that $z$ satisfies the equalities \cref{delta_z-equ}.
Let $h\in\set{0, 1 ,\ldots, k}$. We define two sets $A_h$ and $B_h$ of indices that help to express the quantity $\sum_{i =h}^{\nu -k +h}\binom{\nu -k}{i -h}z_i$ that we want to prove to be zero. The set 
$B_h :=Z\cap\set{h, h +1 ,\ldots, \nu -k +h}$ contains the indices of the non-zero coordinates of $z$ in the range $\{h, h +1, \ldots, \nu -k +h\}$, while $A_h :=(\set{0, 1 ,\ldots, \nu}\setminus\set{h, h +1 ,\ldots, \nu -k +h})\setminus Z$ groups the indices of the zero coordinates of $z$ outside the range $\set{h, h +1, \ldots, \nu -k +h}$.
We obtain the following expression of $\sum_{i =h}^{\nu -k +h}\binom{\nu -k}{i -h}z_i$ as a function of these sets, where $f(A_h, B_h)$ is the number defined in \cref{lem-fAB}:
\begin{align}\label{eq-SB-tech}
\textstyle\sum_{i =h}^{\nu -k +h} \binom{\nu -k}{i -h} z_i	
&\textstyle	=(-1)^{k-h +1}
	\frac{\prod_{i\in Z\setminus\set{\nu}}(\nu -i)}{\prod_{i =0}^{k-1}(\nu -i)}\times
	f(A_h, B_h).
\end{align}

A detailed proof of identity \cref{eq-SB-tech} is given in \cref{appendix:sec-PLSB}. 
We know from \cref{prop-Delta_PL-Base} that $|Z| =|Y\cup X| +1 =k +2$. Thus, we deduce from their definition that the cardinalities of the sets $A_h$ and $B_h$ satisfy:
$$\begin{array}{rll}	
|A_h|	&=\card{\set{0, 1 ,\ldots, \nu}} 
			-\card{\set{h, h +1 ,\ldots, \nu -k +h}}
			-\card{Z}
			+\card{B_h}\\
		&=(\nu +1) -(\nu -k +1) -(k +2) +|B_h|\qquad=|B_h| -2.
\end{array}$$
According to \cref{lem-fAB}, the right-hand side of identity \cref{eq-SB-tech} is therefore zero. We conclude that $z$ does indeed satisfy the equality \cref{delta_z-equ} of rank $h$. 
\end{proof}

Note that \cref{prop-Delta_PL-SB} implies in particular that $LP_{\nu, d, k}$ has no degenerate basic solution provided that $d <\nu$.
The expression \cref{eq-SB_xy} of basic solutions allows us to characterize the {\em feasible} bases, i.e., those whose associated basic solutions are feasible.

% ________________ Delta opt :: Base réal
\begin{proposition}\label{prop-Delta_PL-SRB}
Let $k > 0$, $d\geq k$, and $\nu > d$ be three integers, and $(Y, X)$ be a basis of $LP_{\nu, d, k}$. 
Then $(Y, X)$ is {\em feasible} if and only if the elements of $Y\cup X$, when listed in decreasing order, alternate between $X$ and $Y$, starting with $X$. 
\end{proposition}

\begin{proof}	%% Immédiat depuis \cref{eq-SB_xy}. 
If $(y, x)$ is the basic solution determined by $(Y, X)$, then $(Y, X)$ is feasible if and only if the coordinates of $(y, x)$ are all non-negative. 
We know from \cref{prop-Delta_PL-Base} that $Y\cup X$ is a set of $k +1$ totally ordered elements. 
Thus, we can assign to each element $i\in Y\cup X$ an integer in $[k +1]$, denoted by $\mathop{rk}(i)$, which indicates the rank of $i$ in $Y\cup X$. 
Let $i\in X$. 
By \cref{eq-SB_xy}, we have:
$$\begin{array}{rll}
x_i	&=\prod_{a\in(Y\cup X)\setminus\set{i}}\frac{\nu -a}{i -a}/\binom{\nu}{i}
	&=\prod_{a\in(Y\cup X)\setminus\set{i}}\frac{\nu -a}{|i -a|}/\binom{\nu}{i}\times(-1)^{k +1 -\mathop{rk}(i)}.
\end{array}$$

Consequently, $x_i$ is non-negative if and only if $\mathop{rk}(i) \equiv k +1 \bmod{2}$. 
Symmetrically, we can deduce from \cref{eq-SB_xy} that for all $i\in Y$, $y_i$ is non-negative if and only if $\mathop{rk}(i) \not\equiv k +1 \bmod{2}$. This means that the elements of $Y\cup X$, sorted in descending order, must alternately belong to $X$ and $Y$, starting with $X$, to ensure feasibility.
\end{proof}

In particular, \cref{prop-Delta_PL-SRB} implies that the largest index in a feasible basis $(Y, X)$ belongs to $X$, and is therefore less than or equal to $d$. 
We now switch to {\em optimal} feasible bases, i.e., those whose associated basic solutions are optimal.

% ________________ Delta opt :: Base opt
\begin{proposition}\label{prop-Delta_PL-SB_opt}
Let $k >0$, $d \geq k$, and $\nu >d$ be three integers, and $(Y, X)$ be a feasible basis of $LP_{\nu, d, k}$.
If $(Y, X)$ is optimal, then $d \in X$ and $0\in Y\cup X$.
\end{proposition}

\begin{proof}
We denote by $v(Y, X)$ the objective value of the basic solution associated with $(Y, X)$. 
Since $(Y, X)$ is feasible, the coordinates of this solution are all non-negative.
So we deduce from \cref{eq-SB_xy} that $v(Y, X)$ can be expressed as:
$$\begin{array}{rl}
v(Y, X)
&=	 \sum_{i\in Y}\binom{\nu}{i}\times
	 	\abs{\prod_{a\in(Y\cup X)\setminus\set{i}}\frac{\nu -a}{i -a}/\binom{\nu}{i}}
	 +\sum_{i\in X}\binom{\nu}{i}\times
	 	\abs{\prod_{a\in(Y\cup X)\setminus\set{i}}\frac{\nu -a}{i -a}/\binom{\nu}{i}}	\\
&=	\sum_{i\in Y\cup X}\prod_{a\in(Y\cup X)\setminus\set{i}}
		\frac{\nu -a}{|i -a|}.
\end{array}$$

Let $i\in Y\cup X$. Equivalently, we can write:
\begin{align}\label{eq-vXYi}
v(Y, X) &\textstyle=\prod_{a\in(Y\cup X)\setminus\set{i}}\frac{\nu -a}{|i -a|}
			+\sum_{j\in(Y\cup X)\setminus\set{i}}
				\frac{\nu -i}{|j -i|} \times
				\prod_{a\in(Y\cup X)\setminus\set{i, j}}\frac{\nu -a}{|j -a|}.
\end{align}

We denote by $i_*$ and $i^*$ respectively the smallest and largest elements of $Y\cup X$, 
and by $i'_*$ and $i'^*$ respectively the smallest and largest elements of $(Y\cup X)\setminus\set{i_*,i^*}$.
Considering expression \cref{eq-vXYi} at $i =i_*$, we observe that $v(Y, X)$, as a function of $i_*$, is strictly increasing over $\set{0, 1 ,\ldots, i'_* -1}$. 
We symmetrically observe that $v(Y, X)$, as a function of $i^*$, is strictly decreasing over $\set{i'^* +1 ,\ldots, d -1, d}$. 
Now, we know from \cref{prop-Delta_PL-SRB} that the bases
%\begin{center}
	$\left(Y, X\setminus\set{i^*}\cup\set{d}\right)$, and 
	either $\left(Y, X\setminus\set{i_*}\cup\set{0}\right)$ or $\left(Y\setminus\set{i_*}\cup\set{0}, X\right)$
	(depending on $k\bmod{2}$)
%\end{center}
of $LP_{\nu, d, k}$ are feasible provided that $(Y, X)$ is. 
We conclude that $(Y, X)$ cannot be optimal unless $i^* =d$ and $i_* =0$.
\end{proof}

\Cref{thm-delta_opt} summarizes what we have learned about optimal CPAs, based on the results of \cref{prop-Delta-PL,%	PL ok
prop-Delta_PL-Base,% 	bases
prop-Delta_PL-SB,% 		expression solution (et donc val obj)
prop-Delta_PL-SRB,% 	bases réal
prop-Delta_PL-SB_opt}.%	bases opt

\begin{theorem}\label{thm-delta_opt}
Let $k >0$, $d\geq k$, $\nu >d$ be three integers, 
and $i =(i_0, i_1 ,\ldots, i_k)$ be a sequence of $k+1$ integers satisfying: 
\begin{align}\label{eq-delta-XY}
	0\leq i_0 < i_1 <\cdots< i_{k-1} <i_k \leq d.
\end{align}

Furthermore, let $R^*$ be a positive integer verifying:
\begin{align}\label{eq-delta-R*}
\textstyle
	R^*\times\prod_{s\in\set{0, 1 ,\ldots, k}
		\setminus\set{r}}\frac{\nu -i_s}{|i_r -i_s|}/\binom{\nu}{i_r}
	&\in\mathbb{N},	&r\in\set{0, 1 ,\ldots, k}.
\end{align}

Then there exists a regular $(\nu, d)$-CPA $(N,D)$ of strength $k$ with row weights $\set{i_0, i_1 ,\ldots, i_k, \nu}$. 
More precisely, the CPA contains:
\begin{itemize}
\item $R^*$ rows of all ones, which appear in $N$;
\item for each $r\in\set{0, 1 ,\ldots, k}$, 
	$R^*\times\prod_{s\in\set{0, 1 ,\ldots, k}
		\setminus\set{r}}\frac{\nu -i_s}{|i_r -i_s|}/\binom{\nu}{i_r}$ copies of each binary word of length $\nu$ and weight $i_r$, which appear in $D$ if $r\equiv k\pmod{2}$ and in $N$ otherwise. 
\end{itemize}
	
Furthermore, $\delta(\nu, d, k)$ is the maximum value, 
	over all sequences $i =(i_0, i_1 ,\ldots, i_k)$ of integers satisfying \cref{eq-delta-XY} with $i_0 =0$ and $i_k =d$,
attained by the following expression:
\begin{align}\label{eq-delta_opt}
\textstyle
	2/\left(1 +\sum_{r =0}^k \prod_{s =0}^{r -1}(\nu -i_s)/(i_r -i_s)
		\times\prod_{s =r +1}^k(\nu -i_s)/(i_s -i_r)\right).
\end{align}
\end{theorem}

% ____________________________ Conséquences sur les nombres gamma
\section{Consequences for the numbers $\gamma(q, p, k)$}
\label{sec-opt}
We draw consequences from \cref{thm-reg_joker,thm-delta_opt}.

\begin{corollary}\label{cor-gamma=delta}
For all integers $k > 0$, $p\geq k$, and $q\geq p$, we have $\gamma(q, p, k) = \delta(q, p, k)$.
\end{corollary}

\begin{proof}
The case $p = q$ is obvious, since $\gamma(q, q, k) = \delta(q, q, k) = 1$.
Thus assume $q > p$. 
We know from \cref{prop-Gamma_2_Delta} that $\delta(q, p, k)\geq \gamma(q, p, k)$.
Conversely, according to \cref{thm-delta_opt}, there exists a regular CPA $(N, D)$ that realizes $\delta(q, p, k)$, with the rows of $D$ having weight either $p$ or at most $p-2$, and the rows of $N$ having weight either $q$ or less than $p$. \Cref{thm-reg_joker} then allows us to derive from $(N, D)$ a $(q, p)$-ARPA $(Q, P)$ of strength $k$ such that $R^*(Q, P)/R(Q, P) = R^*(N, D)/R(N, D) = \delta(q, p, k)$. 
So $\gamma(q, p, k)\geq\delta(q, p, k)$, completing the proof.
\end{proof}

As a consequence of \cref{prop-Delta-PL,cor-gamma=delta}, for all integers $k > 0$, $p\geq k$, and $q > p$, the number $\gamma(q, p, k)$ can be efficiently computed by solving the linear program $LP_{q, p, k}$ with $q+p+1$ continuous variables and $k+1$ constraints. 
Next, we use \cref{thm-delta_opt} to obtain the expression for $\delta(q, p, k)$---and hence $\gamma(q, p, k)$---in the cases where $p = k$ or $k\in\{1, 2\}$. 
In addition, we show ARPAs achieving $\gamma(q, p, k)$ by applying \cref{alg-D2G} to the corresponding CPAs, thus illustrating \cref{thm-reg_joker}. 

\begin{corollary}\label{cor-delta-p=k+k<=2}
Let $k > 0$, $p\geq k$, $q > p$ be three integers such that $p = k$ or $k\in\set{1, 2}$. Then $\delta(q, p, k)$---and thus, $\gamma(q, p, k)$---is equal to
\begin{align}
\label{eq-gamma-qkk}\textstyle
	2/(1+\sum_{i = 0}^k \binom{q}{i}\binom{q-i-1}{k-i})	&\text{ if $p = k$},\\
\label{eq-gamma-qp1}
	p/q														&\text{ if $k = 1$},\\
\label{eq-gamma-qp2}
	\lceil p/2\rceil \lfloor p/2\rfloor/
	\left(\left(q-\lceil p/2\rceil\right)\left(q-\lfloor p/2\rfloor\right)\right)
															&\text{ if $k = 2$}.
\end{align}
\end{corollary}

\begin{proof}
According to \cref{thm-delta_opt}, $\delta(q, p, k)$ coincides with the maximum value taken by the expression \cref{eq-delta_opt} over all sequences $i = (i_0, i_1 ,\ldots, i_k)$ of $k+1$ integers satisfying \cref{eq-delta-XY}, i.e.:
$$\begin{array}{ll}
	i_0 =  0 < i_1 <\cdots< i_{k-1} < i_k = p.
\end{array}$$
 
When $p = k$, $i_0, i_1 ,\ldots, i_k$ are necessarily $0, 1 ,\ldots, k$. In this case, expression \cref{eq-delta_opt} evaluates to:
$$\begin{array}{rll}
2/\left(
	1+\sum_{r = 0}^k 
		\prod_{s = 0}^{r-1}\frac{q-s}{r-s} \times \prod_{s = r+1}^k\frac{q-s}{s-r}
\right)&=2/\left(
	1+\sum_{r = 0}^k \binom{q}{r}\times\binom{q-r-1}{k-r}
\right)
\end{array}.$$ 
(Alternatively, in this case, we can consider \cref{thm-gamma_qpk,thm-delta_d=k-UB,prop-Gamma_2_Delta}.)

If $k = 1$, then $i_0 = 0$, $i_1 = p$, and expression $\cref{eq-delta_opt}$ evaluates to: 
$$\begin{array}{rll}
	2/\left(1+\frac{q-p}{p-0}+\frac{q-0}{p-0}\right) 
	&=2/(1+2q/p-p/p)
	&=p/q
\end{array}.$$ 

When $k = 2$, considering $i_0 = 0 <i_1 <i_2 = p$, $\delta(q, p, 2)$ is the maximum over all $i_1\in[p-1]$ of the following expression:
$$\begin{array}{rll}
	2/\left(
		1 	+\frac{(q-i_1)(q-p)}{(i_1-0)(p-0)} 
			+\frac{(q-0)(q-p)}{(i_1-0)(p-i_1)}
			+\frac{(q-0)(q-i_1)}{(p-0)(p-i_1)}
	\right)	&=
	2/\left(
		1 	+\frac{(q-i_1)(q-p)}{i_1p} 
			+\frac{q(q-p)}{i_1(p-i_1)}
			+\frac{q(q-i_1)}{p(p-i_1)}
	\right)
\end{array}.$$ 
This expression simplifies to $i_1(p-i_1)/\left((q-i_1)(q-p+i_1)\right)$, which is maximized for $i_1\in\set{\lfloor p/2\rfloor, \lceil p/2\rceil}$. 
\end{proof}

%%%%%%%%%%%%% p = k
\begin{figure}[t]{\footnotesize
\caption{ARPAs achieving $\gamma(6, 2, 2)$ and $\gamma(5, 3, 3)$. 
These ARPAs are obtained by applying Algorithm \ref{alg-D2G} to the optimal regular CPAs constructed using Theorem \ref{thm-delta_opt}. 
Gray and black indicate entries equal to 1 and 0, respectively, in the original CPAs.}
\label{fig-gamma_q_k_k}
\begin{center}
\setlength\arraycolsep{3pt}
$\begin{array}{c|c}
\begin{array}{c}
	\gamma(6, 2, 2) =1/25\\[4pt]
	\begin{array}{ccc}\setlength\arraycolsep{1.5pt}
		\begin{array}{cccccc}
			Q^0			&Q^1	 	&Q^2		&Q^3		&Q^4		&Q^5		\\\hline
			\equj{0}	&\equj{1}	&\equj{2}	&\equj{3}	&\equj{4}	&\equj{5}	\\\hline
			\equj{0}	&\difj{0}	&\difj{0}	&\difj{0}	&\difj{0}	&\difj{0}	\\
			\equj{0}	&\difj{0}	&\difj{0}	&\difj{0}	&\difj{0}	&\difj{0}	\\
			\equj{0}	&\difj{0}	&\difj{0}	&\difj{0}	&\difj{0}	&\difj{0}	\\
			\equj{0}	&\difj{0}	&\difj{0}	&\difj{0}	&\difj{0}	&\difj{0}	\\
			\difj{1}	&\equj{1}	&\difj{0}	&\difj{0}	&\difj{0}	&\difj{0}	\\
			\difj{1}	&\equj{1}	&\difj{1}	&\difj{1}	&\difj{1}	&\difj{1}	\\
			\difj{1}	&\equj{1}	&\difj{1}	&\difj{1}	&\difj{1}	&\difj{1}	\\
			\difj{1}	&\equj{1}	&\difj{1}	&\difj{1}	&\difj{1}	&\difj{1}	\\
			\difj{1}	&\difj{0}	&\equj{2}	&\difj{0}	&\difj{0}	&\difj{0}	\\
			\difj{2}	&\difj{2}	&\equj{2}	&\difj{1}	&\difj{1}	&\difj{1}	\\
			\difj{2}	&\difj{2}	&\equj{2}	&\difj{2}	&\difj{2}	&\difj{2}	\\
			\difj{2}	&\difj{2}	&\equj{2}	&\difj{2}	&\difj{2}	&\difj{2}	\\
			\difj{1}	&\difj{0}	&\difj{0}	&\equj{3}	&\difj{0}	&\difj{0}	\\
			\difj{2}	&\difj{2}	&\difj{1}	&\equj{3}	&\difj{1}	&\difj{1}	\\
			\difj{3}	&\difj{3}	&\difj{3}	&\equj{3}	&\difj{2}	&\difj{2}	\\
			\difj{3}	&\difj{3}	&\difj{3}	&\equj{3}	&\difj{3}	&\difj{3}	\\
			\difj{1}	&\difj{0}	&\difj{0}	&\difj{0}	&\equj{4}	&\difj{0}	\\
			\difj{2}	&\difj{2}	&\difj{1}	&\difj{1}	&\equj{4}	&\difj{1}	\\
			\difj{3}	&\difj{3}	&\difj{3}	&\difj{2}	&\equj{4}	&\difj{2}	\\
			\difj{4}	&\difj{4}	&\difj{4}	&\difj{4}	&\equj{4}	&\difj{3}	\\
			\difj{1}	&\difj{0}	&\difj{0}	&\difj{0}	&\difj{0}	&\equj{5}	\\
			\difj{2}	&\difj{2}	&\difj{1}	&\difj{1}	&\difj{1}	&\equj{5}	\\
			\difj{3}	&\difj{3}	&\difj{3}	&\difj{2}	&\difj{2}	&\equj{5}	\\
			\difj{4}	&\difj{4}	&\difj{4}	&\difj{4}	&\difj{3}	&\equj{5}	\\
		\end{array}&&\setlength\arraycolsep{1.5pt}\begin{array}{cccccc}
			P^0			&P^1		&P^2		&P^3		&P^4		&P^5		\\\hline
			\equj{0}	&\equj{1}	&\difj{0}	&\difj{0}	&\difj{0}	&\difj{0}	\\
			\equj{0}	&\difj{0}	&\equj{2}	&\difj{0}	&\difj{0}	&\difj{0}	\\
			\equj{0}	&\difj{0}	&\difj{0}	&\equj{3}	&\difj{0}	&\difj{0}	\\
			\equj{0}	&\difj{0}	&\difj{0}	&\difj{0}	&\equj{4}	&\difj{0}	\\
			\equj{0}	&\difj{0}	&\difj{0}	&\difj{0}	&\difj{0}	&\equj{5}	\\
			\difj{1}	&\equj{1}	&\equj{2}	&\difj{1}	&\difj{1}	&\difj{1}	\\
			\difj{1}	&\equj{1}	&\difj{1}	&\equj{3}	&\difj{1}	&\difj{1}	\\
			\difj{1}	&\equj{1}	&\difj{1}	&\difj{1}	&\equj{4}	&\difj{1}	\\
			\difj{1}	&\equj{1}	&\difj{1}	&\difj{1}	&\difj{1}	&\equj{5}	\\
			\difj{2}	&\difj{2}	&\equj{2}	&\equj{3}	&\difj{2}	&\difj{2}	\\
			\difj{2}	&\difj{2}	&\equj{2}	&\difj{2}	&\equj{4}	&\difj{2}	\\
			\difj{2}	&\difj{2}	&\equj{2}	&\difj{2}	&\difj{2}	&\equj{5}	\\
			\difj{3}	&\difj{3}	&\difj{3}	&\equj{3}	&\equj{4}	&\difj{3}	\\
			\difj{3}	&\difj{3}	&\difj{3}	&\equj{3}	&\difj{3}	&\equj{5}	\\
			\difj{4}	&\difj{4}	&\difj{4}	&\difj{4}	&\equj{4}	&\equj{5}	\\\hline
			\difj{1}	&\difj{0}	&\difj{0}	&\difj{0}	&\difj{0}	&\difj{0}	\\
			\difj{1}	&\difj{0}	&\difj{0}	&\difj{0}	&\difj{0}	&\difj{0}	\\
			\difj{1}	&\difj{0}	&\difj{0}	&\difj{0}	&\difj{0}	&\difj{0}	\\
			\difj{1}	&\difj{0}	&\difj{0}	&\difj{0}	&\difj{0}	&\difj{0}	\\
			\difj{2}	&\difj{2}	&\difj{1}	&\difj{1}	&\difj{1}	&\difj{1}	\\
			\difj{2}	&\difj{2}	&\difj{1}	&\difj{1}	&\difj{1}	&\difj{1}	\\
			\difj{2}	&\difj{2}	&\difj{1}	&\difj{1}	&\difj{1}	&\difj{1}	\\
			\difj{3}	&\difj{3}	&\difj{3}	&\difj{2}	&\difj{2}	&\difj{2}	\\
			\difj{3}	&\difj{3}	&\difj{3}	&\difj{2}	&\difj{2}	&\difj{2}	\\
			\difj{4}	&\difj{4}	&\difj{4}	&\difj{4}	&\difj{3}	&\difj{3}	\\
		\end{array}
	\end{array}
\end{array}&\begin{array}{c} 
	\gamma(5, 3, 3) =1/25\\[4pt]
	\begin{array}{ccc}\setlength\arraycolsep{1.5pt}
		\begin{array}{ccccc}
			Q^0			&Q^1	 	&Q^2		&Q^3		&Q^4		\\\hline
			\equj{0}	&\equj{1}	&\equj{2}	&\equj{3}	&\equj{4}	\\\hline			
			\equj{0}	&\equj{1}	&\difj{0}	&\difj{0}	&\difj{0}	\\
			\equj{0}	&\equj{1}	&\difj{0}	&\difj{0}	&\difj{0}	\\
			\equj{0}	&\difj{0}	&\equj{2}	&\difj{0}	&\difj{0}	\\
			\equj{0}	&\difj{0}	&\equj{2}	&\difj{0}	&\difj{0}	\\
			\equj{0}	&\difj{0}	&\difj{0}	&\equj{3}	&\difj{0}	\\
			\equj{0}	&\difj{0}	&\difj{0}	&\equj{3}	&\difj{0}	\\
			\equj{0}	&\difj{0}	&\difj{0}	&\difj{0}	&\equj{4}	\\
			\equj{0}	&\difj{0}	&\difj{0}	&\difj{0}	&\equj{4}	\\
			\difj{1}	&\equj{1}	&\equj{2}	&\difj{0}	&\difj{0}	\\
			\difj{1}	&\equj{1}	&\equj{2}	&\difj{1}	&\difj{1}	\\
			\difj{1}	&\equj{1}	&\difj{0}	&\equj{3}	&\difj{0}	\\
			\difj{1}	&\equj{1}	&\difj{1}	&\equj{3}	&\difj{1}	\\
			\difj{1}	&\equj{1}	&\difj{0}	&\difj{0}	&\equj{4}	\\
			\difj{1}	&\equj{1}	&\difj{1}	&\difj{1}	&\equj{4}	\\
			\difj{1}	&\difj{0}	&\equj{2}	&\equj{3}	&\difj{0}	\\
			\difj{2}	&\difj{2}	&\equj{2}	&\equj{3}	&\difj{1}	\\
			\difj{1}	&\difj{0}	&\equj{2}	&\difj{0}	&\equj{4}	\\
			\difj{2}	&\difj{2}	&\equj{2}	&\difj{1}	&\equj{4}	\\
			\difj{1}	&\difj{0}	&\difj{0}	&\equj{3}	&\equj{4}	\\
			\difj{2}	&\difj{2}	&\difj{1}	&\equj{3}	&\equj{4}	\\\hline
			\difj{1}	&\difj{0}	&\difj{0}	&\difj{0}	&\difj{0}	\\
			\difj{1}	&\difj{0}	&\difj{0}	&\difj{0}	&\difj{0}	\\
			\difj{1}	&\difj{0}	&\difj{0}	&\difj{0}	&\difj{0}	\\
			\difj{2}	&\difj{2}	&\difj{1}	&\difj{1}	&\difj{1}	\\
		\end{array}&&\setlength\arraycolsep{1.5pt}\begin{array}{ccccc}
			P^0			&P^1		&P^2		&P^3		&P^4		\\\hline
			\equj{0}	&\equj{1}	&\equj{2}	&\difj{0}	&\difj{0}	\\
			\equj{0}	&\equj{1}	&\difj{0}	&\equj{3}	&\difj{0}	\\
			\equj{0}	&\equj{1}	&\difj{0}	&\difj{0}	&\equj{4}	\\
			\equj{0}	&\difj{0}	&\equj{2}	&\equj{3}	&\difj{0}	\\
			\equj{0}	&\difj{0}	&\equj{2}	&\difj{0}	&\equj{4}	\\
			\equj{0}	&\difj{0}	&\difj{0}	&\equj{3}	&\equj{4}	\\
			\difj{1}	&\equj{1}	&\equj{2}	&\equj{3}	&\difj{1}	\\
			\difj{1}	&\equj{1}	&\equj{2}	&\difj{1}	&\equj{4}	\\
			\difj{1}	&\equj{1}	&\difj{1}	&\equj{3}	&\equj{4}	\\
			\difj{2}	&\difj{2}	&\equj{2}	&\equj{3}	&\equj{4}	\\\hline
			\equj{0}	&\difj{0}	&\difj{0}	&\difj{0}	&\difj{0}	\\
			\equj{0}	&\difj{0}	&\difj{0}	&\difj{0}	&\difj{0}	\\
			\equj{0}	&\difj{0}	&\difj{0}	&\difj{0}	&\difj{0}	\\
			\difj{1}	&\equj{1}	&\difj{0}	&\difj{0}	&\difj{0}	\\
			\difj{1}	&\equj{1}	&\difj{0}	&\difj{0}	&\difj{0}	\\
			\difj{1}	&\equj{1}	&\difj{1}	&\difj{1}	&\difj{1}	\\
			\difj{1}	&\difj{0}	&\equj{2}	&\difj{0}	&\difj{0}	\\
			\difj{1}	&\difj{0}	&\equj{2}	&\difj{0}	&\difj{0}	\\
			\difj{2}	&\difj{2}	&\equj{2}	&\difj{1}	&\difj{1}	\\
			\difj{1}	&\difj{0}	&\difj{0}	&\equj{3}	&\difj{0}	\\
			\difj{1}	&\difj{0}	&\difj{0}	&\equj{3}	&\difj{0}	\\
			\difj{2}	&\difj{2}	&\difj{1}	&\equj{3}	&\difj{1}	\\
			\difj{1}	&\difj{0}	&\difj{0}	&\difj{0}	&\equj{4}	\\
			\difj{1}	&\difj{0}	&\difj{0}	&\difj{0}	&\equj{4}	\\
			\difj{2}	&\difj{2}	&\difj{1}	&\difj{1}	&\equj{4}	\\
		\end{array}
	\end{array}
\end{array}
\end{array}$
\end{center}}
\end{figure}

We make the ARPAs obtained for the cases $p = k$, $k = 1$, and $k = 2$ explicit. 
\Cref{fig-gamma_q_k_k} shows the ARPAs derived from \cref{thm-delta_opt,thm-reg_joker} when $(q, k)\in\set{(6, 2), (5, 3)}$ and $p = k$. 
These ARPAs are equivalent to those given by the recursive construction proposed in \cite{CT18}, in that they can be interpreted as the same CPAs. 
In these CPAs, the word of weight $q$ occurs exactly once in $N$, and the words of weight $r\in\set{0, 1 ,\ldots, k}$ occur $\binom{q-r-1}{k-1}$ times each, in $D$ if $r\equiv k\bmod{2}$, and in $N$ otherwise. No other words occur in either $D$ or $N$. In particular, when $q = k+1$, the binomial coefficients $\binom{q-r-1}{k-1}$ are all equal to one. Thus, when $q = k+1$, the resulting arrays are simple, with $2^{k-1}$ rows each, including a single row of the form $1\ 1\ \cdots\ 1$.

%%%%%%%%%%%%% k = 1
\begin{figure}[t]{\footnotesize
\caption{ARPAs achieving $\gamma(4, 2, 1)$, $\gamma(5, 3, 1)$ and $\gamma(5, 4, 1)$. These ARPAs are obtained by applying Algorithm \ref{alg-D2G} to the optimal regular CPAs constructed using Theorem \ref{thm-delta_opt}. 
Gray and black indicate entries equal to 1 and 0, respectively, in the original CPAs.}
\label{fig-gamma_q_p_1}
\begin{center}
\setlength\arraycolsep{3pt}
$\begin{array}{c|c|c}
\begin{array}{c}
	\gamma(4, 2, 1) =3/6 =1/2\\[4pt]
	\begin{array}{ccc}\setlength\arraycolsep{1.5pt}
		\begin{array}{cccc}
			Q^0			&Q^1	 	&Q^2		&Q^3		\\\hline
			\equj{0}	&\equj{1}	&\equj{2}	&\equj{3}	\\
			\equj{0}	&\equj{1}	&\equj{2}	&\equj{3}	\\
			\equj{0}	&\equj{1}	&\equj{2}	&\equj{3}	\\\hline
			\difj{1}	&\difj{0}	&\difj{0}	&\difj{0}	\\
			\difj{1}	&\difj{0}	&\difj{0}	&\difj{0}	\\
			\difj{2}	&\difj{2}	&\difj{1}	&\difj{1}	\\
		\end{array}&&\setlength\arraycolsep{1.5pt}\begin{array}{cccc}
			P^0			&P^1		&P^2		&P^3		\\\hline
			\equj{0}	&\equj{1}	&\difj{0}	&\difj{0}	\\
			\equj{0}	&\difj{0}	&\equj{2}	&\difj{0}	\\
			\equj{0}	&\difj{0}	&\difj{0}	&\equj{3}	\\
			\difj{1}	&\equj{1}	&\equj{2}	&\difj{1}	\\
			\difj{1}	&\equj{1}	&\difj{1}	&\equj{3}	\\
			\difj{2}	&\difj{2}	&\equj{2}	&\equj{3}	\\
		\end{array}
	\end{array}
\end{array}&\begin{array}{c} 
	\gamma(5, 3, 1) =6/10 =3/5\\[4pt]
	\begin{array}{ccc}\setlength\arraycolsep{1.5pt}
		\begin{array}{ccccc}
			Q^0			&Q^1		&Q^2	 	&Q^3		&Q^4		\\\hline
			\equj{0}	&\equj{1}	&\equj{2}	&\equj{3}	&\equj{4}	\\
			\equj{0}	&\equj{1}	&\equj{2}	&\equj{3}	&\equj{4}	\\
			\equj{0}	&\equj{1}	&\equj{2}	&\equj{3}	&\equj{4}	\\
			\equj{0}	&\equj{1}	&\equj{2}	&\equj{3}	&\equj{4}	\\
			\equj{0}	&\equj{1}	&\equj{2}	&\equj{3}	&\equj{4}	\\
			\equj{0}	&\equj{1}	&\equj{2}	&\equj{3}	&\equj{4}	\\\hline
			\difj{1}	&\difj{0}	&\difj{0}	&\difj{0}	&\difj{0}	\\
			\difj{1}	&\difj{0}	&\difj{0}	&\difj{0}	&\difj{0}	\\
			\difj{1}	&\difj{0}	&\difj{0}	&\difj{0}	&\difj{0}	\\
			\difj{2}	&\difj{2}	&\difj{1}	&\difj{1}	&\difj{1}	\\
		\end{array}&&\setlength\arraycolsep{1.5pt}\begin{array}{ccccc}
			P^0			&P^1		&P^2	 	&P^3		&P^4		\\\hline
			\equj{0}	&\equj{1}	&\equj{2}	&\difj{0}	&\difj{0}	\\
			\equj{0}	&\equj{1}	&\difj{0}	&\equj{3}	&\difj{0}	\\
			\equj{0}	&\equj{1}	&\difj{0}	&\difj{0}	&\equj{4}	\\
			\equj{0}	&\difj{0}	&\equj{2}	&\equj{3}	&\difj{0}	\\
			\equj{0}	&\difj{0}	&\equj{2}	&\difj{0}	&\equj{4}	\\
			\equj{0}	&\difj{0}	&\difj{0}	&\equj{3}	&\equj{4}	\\
			\difj{1}	&\equj{1}	&\equj{2}	&\equj{3}	&\difj{1}	\\
			\difj{1}	&\equj{1}	&\equj{2}	&\difj{1}	&\equj{4}	\\
			\difj{1}	&\equj{1}	&\difj{1}	&\equj{3}	&\equj{4}	\\
			\difj{2}	&\difj{2}	&\equj{2}	&\equj{3}	&\equj{4}	\\
		\end{array}
	\end{array}
\end{array}&\begin{array}{c} 
	\gamma(5, 4, 1) =4/5\\[4pt]
	\begin{array}{ccc}\setlength\arraycolsep{1.5pt}
		\begin{array}{ccccc}
			Q^0			&Q^1		&Q^2	 	&Q^3		&Q^4		\\\hline
			\equj{0}	&\equj{1}	&\equj{2}	&\equj{3}	&\equj{4}	\\
			\equj{0}	&\equj{1}	&\equj{2}	&\equj{3}	&\equj{4}	\\
			\equj{0}	&\equj{1}	&\equj{2}	&\equj{3}	&\equj{4}	\\
			\equj{0}	&\equj{1}	&\equj{2}	&\equj{3}	&\equj{4}	\\\hline
			\difj{1}	&\difj{0}	&\difj{0}	&\difj{0}	&\difj{0}	\\
		\end{array}&&\setlength\arraycolsep{1.5pt}\begin{array}{ccccc}
			P^0			&P^1		&P^2	 	&P^3		&P^4		\\\hline
			\equj{0}	&\equj{1}	&\equj{2}	&\equj{3}	&\difj{0}	\\
			\equj{0}	&\equj{1}	&\equj{2}	&\difj{0}	&\equj{4}	\\
			\equj{0}	&\equj{1}	&\difj{0}	&\equj{3}	&\equj{4}	\\
			\equj{0}	&\difj{0}	&\equj{2}	&\equj{3}	&\equj{4}	\\
			\difj{1}	&\equj{1}	&\equj{2}	&\equj{3}	&\equj{4}	\\
		\end{array}
	\end{array}
\end{array}
\end{array}$
\end{center}}
\end{figure}

When $k = 1$, we consider the basic solution $(y^*, x^*)$ of $LP_{q, p, 1}$ with non-zero coordinates 
$$\left\{\begin{array}{rll}
	y^*_0	&=1/\binom{q}{0}\times(q-p)/(p-0)		&=q/p-1,\\
	x^*_p	&=1/\binom{q}{p}\times(q-0)/(p-0)		&=1/\binom{q-1}{p-1}.\\
\end{array}\right.$$ 
To derive from this solution of $LP_{q, p, 1}$ a CPA with the required characteristics, for $R^* = \binom{q-1}{p-1}$, we define the solution $(y, x) = (R^* y^*, R^*, R^* x^*)$ of $P_{q, p, 1}$. 
This vector has non-zero coordinates: 
$$\begin{array}{ccc}
	y_0	=(q-p)/p\times\binom{q-1}{p-1} = \binom{q-1}{p},
	&x_p	=1/\binom{q-1}{p-1}\times\binom{q-1}{p-1} = 1,
	&y_q	=\binom{q-1}{p-1}.
\end{array}$$ 
Since these coordinates are integers, $(y, x)$ is the representative vector of a regular $(q, p)$-CPA $(N, D)$ of strength 1 such that $R^*(N, D)/R(N, D) = p/q = \delta(q, p, 1)$. 
We can then use the construction underlying \cref{thm-reg_joker} to transform this CPA into a $(q, p)$-ARPA $(Q, P)$ of strength 1 such that $R^*(Q, P)/R(Q, P) = R^*(N, D)/R(N, D) = \delta(q, p, 1)$. 
\Cref{fig-gamma_q_p_1} shows the ARPAs obtained when $(q, p)\in\set{(4, 2), (5, 3), (5, 4)}$. 

%%%%%%%%%%%%% k = 2
\begin{figure}[t]{\footnotesize
\caption{ARPAs achieving $\gamma(4, 3, 2)$, $\gamma(5, 3, 2)$ and $\gamma(5, 4, 2)$. These ARPAs are obtained by applying Algorithm \ref{alg-D2G} to optimal regular CPAs constructed using Theorem \ref{thm-delta_opt}. 
Gray and black indicate entries equal to 1 and 0, respectively, in the original CPAs.}
\label{fig-gamma_q_p_2}
\begin{center}
\setlength\arraycolsep{3pt}
$\begin{array}{c|c|c}
\begin{array}{c} 
	% ____ $\gamma(4, 3, 2) =2/6 =1/3$
	\gamma(4, 3, 2) =2/6 =1/3\\[4pt]
	\begin{array}{ccc}\setlength\arraycolsep{1.5pt}
		\begin{array}{cccc}
			Q^0			&Q^1	 	&Q^2		&Q^3		\\\hline
			\equj{0}	&\equj{1}	&\equj{2}	&\equj{3}	\\
			\equj{0}	&\equj{1}	&\equj{2}	&\equj{3}	\\\hline
			\equj{0}	&\difj{0}	&\difj{0}	&\difj{0}	\\
			\difj{1}	&\equj{1}	&\difj{0}	&\difj{0}	\\
			\difj{1}	&\difj{0}	&\equj{2}	&\difj{0}	\\
			\difj{1}	&\difj{0}	&\difj{0}	&\equj{3}	\\
		\end{array}&&\setlength\arraycolsep{1.5pt}\begin{array}{cccc}
			P^0			&P^1		&P^2		&P^3		\\\hline
			\equj{0}	&\equj{1}	&\equj{2}	&\difj{0}	\\
			\equj{0}	&\equj{1}	&\difj{0}	&\equj{3}	\\
			\equj{0}	&\difj{0}	&\equj{2}	&\equj{3}	\\
			\difj{1}	&\equj{1}	&\equj{2}	&\equj{3}	\\\hline
			\difj{1}	&\difj{0}	&\difj{0}	&\difj{0}	\\
			\difj{1}	&\difj{0}	&\difj{0}	&\difj{0}	\\
		\end{array}
	\end{array}
\end{array}&\begin{array}{c}
	% ____ $\gamma(5, 3, 2) = 1/N(5, 3, 2)$
	\gamma(5, 3, 2) =3/18 =1/6\\[4pt]
	\begin{array}{ccc}\setlength\arraycolsep{1.5pt}
		\begin{array}{ccccc}
			Q^0			&Q^1		&Q^2	 	&Q^3		&Q^4		\\\hline
			\equj{0}	&\equj{1}	&\equj{2}	&\equj{3}	&\equj{4}	\\
			\equj{0}	&\equj{1}	&\equj{2}	&\equj{3}	&\equj{4}	\\
			\equj{0}	&\equj{1}	&\equj{2}	&\equj{3}	&\equj{4}	\\\hline
			\equj{0}	&\difj{0}	&\difj{0}	&\difj{0}	&\difj{0}	\\
			\equj{0}	&\difj{0}	&\difj{0}	&\difj{0}	&\difj{0}	\\
			\equj{0}	&\difj{0}	&\difj{0}	&\difj{0}	&\difj{0}	\\
			\difj{1}	&\equj{1}	&\difj{0}	&\difj{0}	&\difj{0}	\\
			\difj{1}	&\equj{1}	&\difj{0}	&\difj{0}	&\difj{0}	\\
			\difj{1}	&\equj{1}	&\difj{1}	&\difj{1}	&\difj{1}	\\
			\difj{1}	&\difj{0}	&\equj{2}	&\difj{0}	&\difj{0}	\\
			\difj{1}	&\difj{0}	&\equj{2}	&\difj{0}	&\difj{0}	\\
			\difj{2}	&\difj{2}	&\equj{2}	&\difj{1}	&\difj{1}	\\
			\difj{1}	&\difj{0}	&\difj{0}	&\equj{3}	&\difj{0}	\\
			\difj{1}	&\difj{0}	&\difj{0}	&\equj{3}	&\difj{0}	\\
			\difj{2}	&\difj{2}	&\difj{1}	&\equj{3}	&\difj{1}	\\
			\difj{1}	&\difj{0}	&\difj{0}	&\difj{0}	&\equj{4}	\\
			\difj{1}	&\difj{0}	&\difj{0}	&\difj{0}	&\equj{4}	\\
			\difj{2}	&\difj{2}	&\difj{1}	&\difj{1}	&\equj{4}	\\
		\end{array}&&\setlength\arraycolsep{1.5pt}\begin{array}{ccccc}
			P^0			&P^1		&P^2	 	&P^3		&P^4		\\\hline
			\equj{0}	&\equj{1}	&\equj{2}	&\difj{0}	&\difj{0}	\\
			\equj{0}	&\equj{1}	&\difj{0}	&\equj{3}	&\difj{0}	\\
			\equj{0}	&\equj{1}	&\difj{0}	&\difj{0}	&\equj{4}	\\
			\equj{0}	&\difj{0}	&\equj{2}	&\equj{3}	&\difj{0}	\\
			\equj{0}	&\difj{0}	&\equj{2}	&\difj{0}	&\equj{4}	\\
			\equj{0}	&\difj{0}	&\difj{0}	&\equj{3}	&\equj{4}	\\
			\difj{1}	&\equj{1}	&\equj{2}	&\equj{3}	&\difj{1}	\\
			\difj{1}	&\equj{1}	&\equj{2}	&\difj{1}	&\equj{4}	\\
			\difj{1}	&\equj{1}	&\difj{1}	&\equj{3}	&\equj{4}	\\
			\difj{2}	&\difj{2}	&\equj{2}	&\equj{3}	&\equj{4}	\\\hline
			\difj{1}	&\difj{0}	&\difj{0}	&\difj{0}	&\difj{0}	\\
			\difj{1}	&\difj{0}	&\difj{0}	&\difj{0}	&\difj{0}	\\
			\difj{1}	&\difj{0}	&\difj{0}	&\difj{0}	&\difj{0}	\\
			\difj{1}	&\difj{0}	&\difj{0}	&\difj{0}	&\difj{0}	\\
			\difj{1}	&\difj{0}	&\difj{0}	&\difj{0}	&\difj{0}	\\
			\difj{1}	&\difj{0}	&\difj{0}	&\difj{0}	&\difj{0}	\\
			\difj{2}	&\difj{2}	&\difj{1}	&\difj{1}	&\difj{1}	\\
			\difj{2}	&\difj{2}	&\difj{1}	&\difj{1}	&\difj{1}	\\
		\end{array}
	\end{array}
\end{array}&\begin{array}{c} 
	% ____ $\gamma(5, 4, 2)$
	\gamma(5, 4, 2) =8/18 =4/9\\[4pt]
	\begin{array}{ccc}\setlength\arraycolsep{1.5pt}
		\begin{array}{ccccc}
			Q^0			&Q^1		&Q^2	 	&Q^3		&Q^4		\\\hline
			\equj{0}	&\equj{1}	&\equj{2}	&\equj{3}	&\equj{4}	\\
			\equj{0}	&\equj{1}	&\equj{2}	&\equj{3}	&\equj{4}	\\
			\equj{0}	&\equj{1}	&\equj{2}	&\equj{3}	&\equj{4}	\\
			\equj{0}	&\equj{1}	&\equj{2}	&\equj{3}	&\equj{4}	\\
			\equj{0}	&\equj{1}	&\equj{2}	&\equj{3}	&\equj{4}	\\
			\equj{0}	&\equj{1}	&\equj{2}	&\equj{3}	&\equj{4}	\\
			\equj{0}	&\equj{1}	&\equj{2}	&\equj{3}	&\equj{4}	\\
			\equj{0}	&\equj{1}	&\equj{2}	&\equj{3}	&\equj{4}	\\\hline
			\equj{0}	&\equj{1}	&\difj{0}	&\difj{0}	&\difj{0}	\\
			\equj{0}	&\difj{0}	&\equj{2}	&\difj{0}	&\difj{0}	\\
			\equj{0}	&\difj{0}	&\difj{0}	&\equj{3}	&\difj{0}	\\
			\equj{0}	&\difj{0}	&\difj{0}	&\difj{0}	&\equj{4}	\\
			\difj{1}	&\equj{1}	&\equj{2}	&\difj{0}	&\difj{0}	\\
			\difj{1}	&\equj{1}	&\difj{0}	&\equj{3}	&\difj{0}	\\
			\difj{1}	&\equj{1}	&\difj{0}	&\difj{0}	&\equj{4}	\\
			\difj{1}	&\difj{0}	&\equj{2}	&\equj{3}	&\difj{0}	\\
			\difj{1}	&\difj{0}	&\equj{2}	&\difj{0}	&\equj{4}	\\
			\difj{1}	&\difj{0}	&\difj{0}	&\equj{3}	&\equj{4}	\\
		\end{array}&&\setlength\arraycolsep{1.5pt}\begin{array}{ccccc}
			P^0			&P^1		&P^2	 	&P^3		&P^4		\\\hline
			\equj{0}	&\equj{1}	&\equj{2}	&\equj{3}	&\difj{0}	\\
			\equj{0}	&\equj{1}	&\equj{2}	&\equj{3}	&\difj{0}	\\
			\equj{0}	&\equj{1}	&\equj{2}	&\equj{3}	&\difj{0}	\\
			\equj{0}	&\equj{1}	&\equj{2}	&\difj{0}	&\equj{4}	\\
			\equj{0}	&\equj{1}	&\equj{2}	&\difj{0}	&\equj{4}	\\
			\equj{0}	&\equj{1}	&\equj{2}	&\difj{0}	&\equj{4}	\\
			\equj{0}	&\equj{1}	&\difj{0}	&\equj{3}	&\equj{4}	\\
			\equj{0}	&\equj{1}	&\difj{0}	&\equj{3}	&\equj{4}	\\
			\equj{0}	&\equj{1}	&\difj{0}	&\equj{3}	&\equj{4}	\\
			\equj{0}	&\difj{0}	&\equj{2}	&\equj{3}	&\equj{4}	\\
			\equj{0}	&\difj{0}	&\equj{2}	&\equj{3}	&\equj{4}	\\
			\equj{0}	&\difj{0}	&\equj{2}	&\equj{3}	&\equj{4}	\\
			\difj{1}	&\equj{1}	&\equj{2}	&\equj{3}	&\equj{4}	\\
			\difj{1}	&\equj{1}	&\equj{2}	&\equj{3}	&\equj{4}	\\
			\difj{1}	&\equj{1}	&\equj{2}	&\equj{3}	&\equj{4}	\\\hline
			\difj{1}	&\difj{0}	&\difj{0}	&\difj{0}	&\difj{0}	\\
			\difj{1}	&\difj{0}	&\difj{0}	&\difj{0}	&\difj{0}	\\
			\difj{1}	&\difj{0}	&\difj{0}	&\difj{0}	&\difj{0}	\\
		\end{array}
	\end{array}
\end{array}
\end{array}$
\end{center}}
\end{figure}

When $k = 2$, the proof of \cref{cor-delta-p=k+k<=2} involves the basic solution $(y^*, x^*)$ of $LP_{q, p, 2}$ with the following non-zero coordinates: 
$$\left\{\begin{array}{rll}
x^*_0	&=\frac{(q-p)(q-\lfloor p/2\rfloor)}{(p-0)(\lfloor p/2\rfloor-0)}/\binom{q}{0}
	&=\frac{(q-p)(q-\lfloor p/2\rfloor)}{p\lfloor p/2\rfloor},					\\[5pt]
y^*_{\lfloor p/2\rfloor}	
	&=\frac{(q-p)(q-0)}{(p-\lfloor p/2\rfloor)(\lfloor p/2\rfloor-0)}/\binom{q}{\lfloor p/2\rfloor}
	&=\frac{q-p}{\lceil p/2\rceil}/\binom{q-1}{\lfloor p/2\rfloor-1},			\\[5pt]
x^*_p	&=\frac{(q-\lfloor p/2\rfloor)(q-0)}{(p-\lfloor p/2\rfloor)(p-0)}/\binom{q}{p}
	&=\frac{q-\lfloor p/2\rfloor}{\lceil p/2\rceil}/\binom{q-1}{p-1}.			\\
\end{array}\right.$$ 
Let $R^* > 0$. According to the proof of \cref{prop-Delta-PL}, $(R^* y^*, R^*, R^* x^*)$ is the representative vector of a regular $(q, p)$-CPA $(N, D)$ of strength 2 such that $R^*(N, D)/R(N, D) = \delta(q, p, 2)$ if and only if $R^*$ is a positive integer such that $R^*\times x^*_0$, $R^*\times y^*_{\lfloor p/2\rfloor}$, and $R^*\times x^*_p$ are integers. 
For example, when $p\in\set{3, 4}$, the smallest such integer depends on $q\bmod{3}$ and takes onn of tthe values $q-2$, $(q-3)(q-1)$, and $(q-3)(q-1)/3$. 
Thus, for $p\in\set{3, 4}$, we obtain the following optimal solutions $(y^*, x^*)$ and $(y, x)$ of $LP_{q, p, 2}$ and $P_{q, p, 2}$, depending on $p$ and $(q\bmod{3})$ (for both solutions we give only their non-zero coordinates):
$$\begin{array}{cc|ccc|cccc}
\multicolumn{2}{c|}{\textit{Case considered:}}
&\multicolumn{3}{c|}{\textit{Solution $(y^*, x^*)$ of $LP_{q, p, 2}$:}}
&\multicolumn{4}{c}{\textit{Solution $(y, x) = (R^*\times y^*, R^*, R^*\times x^*)$ of $P_{q, p, 2}$:}}\\\hline
p						&q\bmod{3}
	&x^*_0					&y^*_{\lfloor p/2\rfloor}	&x^*_p	
	&x_0					&y_{\lfloor p/2\rfloor}		&x_p	&y_q				\\\hline
3						&\forall
	&\frac{(q-3)(q-1)}{3}	&\frac{q-3}{2}				&\frac{1}{q-2}
	&2\binom{q-1}{3}		&\binom{q-2}{2}			&1		&q-2				\\[5pt]
\multirow{2}{*}{4}		&2
	&\multirow{2}{*}{$\frac{(q-4)(q-2)}{8}$}
						&\multirow{2}{*}{$\frac{q-4}{2(q-1)}$}
													&\multirow{2}{*}{$\frac{3}{(q-1)(q-3)}$}
	&3\binom{q-1}{4}		&\binom{q-3}{2}			&3		&(q-3)(q-1)		\\[5pt]
						&0, 1
	&					&								&
	&\binom{q-1}{4}	&\binom{q-3}{2}/3				&1		&(q-3)(q-1)/3		\\
\end{array}$$ 
From $(y, x)$, we derive $(q, p)$-ARPAs $(Q, P)$ of strength 2 such that $R^*(Q, P)/R(Q, P) = \delta(q, p, 2)$ using \cref{alg-D2G}. 
\Cref{fig-gamma_q_p_2} shows the ARPAs obtained when $(q, p)\in\set{(4, 3), (5, 3), (5, 4)}$.

% ____________________________ Conclusion (génerale)
\section{Concluding remarks and further research}
\label{sec-conc}
\Cref{cor-gamma=delta} together with \cref{thm-CT18} implies that $\mathsf{k\,CSP\!-\!q}$ reduces to $\mathsf{k\,CSP\!-\!p}$ by a reduction that preserves the differential approximation guarantee up to a multiplicative factor of $\delta(q, p, k)$. 
Moreover, $\delta(q, p, k)$ is an estimate of the differential ratio achieved by the best solutions of an instance of $\mathsf{k\,CSP\!-\!q}$ among those whose coordinates take at most $p$ distinct values. 
It is worth noting that, for instances $I$ of $\mathsf{k\,CSP\!-\!q}$ with $\nu$ variables, $\delta(\nu, d, k)$ also gives an estimate of how well Hamming balls of radius $d$ approximate the diameter of $I$, which is the quantity $|\mathrm{opt}(I) - \mathrm{wor}(I)|$. More precisely, for any Hamming ball of radius $d$ of such an instance, the maximum difference between two solution values is a fraction at least $\delta(\nu, d, k)$ of the instance diameter (although not yet published, this result is available in a preprint\footnote{J.-F. Culus and S. Toulouse. Deriving differential approximation results for $k$-CSPs from combinatorial designs. {\em arXiv eprint 2409.03903}, https://arxiv.org/abs/2409.03903, 2024}).
Therefore, deriving from \cref{thm-delta_opt} the value of $\delta(\nu, d, k)$ for more triples $(\nu, d, k)$ would refine our understanding of two aspects of the differential approximability of $k$-CSPs. 

However, our ongoing research focuses on a deeper exploration of both ARPAs and CPAs. 

% ____________ tableaux de peu de lignes
\smallskip
\begin{table}[t]{\footnotesize
\caption{The ratio $R^*(Q, P)/R(Q, P)$ in $(q, p)$-ARPAs of strength $k$ that minimize $R(Q, P)$, among those that realize $\gamma(q, p, k)$ (the $\gamma$ columns), and the minimum number $R(Q, P)$ of rows in $(q, p)$-ARPAs of strength $k$ (the $R$ columns). 
We use gray to emphasize the cases where ARPAs that can be interpreted as regular CPAs realize the corresponding value. These values can be computed by solving linear programs with continuous and integer variables \cite{CT26-reg}.}
\label{tab-gamma}
\begin{center}
\begin{tabular}{l}\setlength\arraycolsep{6.5pt}
$\begin{array}{|c|c|cc|cc|cc|cc|cc|cc|}
% _____________________ gamma
\multicolumn{2}{c|}{} &\multicolumn{12}{c|}{q}\\\cline{3-14}
\multicolumn{2}{c|}{} &\multicolumn{2}{c|}{3}	&\multicolumn{2}{c|}{4}	&\multicolumn{2}{c|}{5}	&\multicolumn{2}{c|}{6}	&\multicolumn{2}{c|}{7}	&\multicolumn{2}{c|}{8}\\\cline{3-14}
k	&p	&\gamma &R	&\gamma &R	&\gamma &R	&\gamma 		&R	&\gamma 	&R	&\gamma 	&R	\\
\hline\multirow{5}{*}{2}	%%%%%%%%%%%%%%%%%%%%%%%%% k =2
&2&\gr{1/4}&\gr{4}&\gr{1/9}&\gr{9}&\gr{1/16}&\gr{16}&\gr{1/25}&\gr{25}&\gr{1/36}&\gr{36}
	&\gr{1/49}&\gr{49}\\
\cline{2-14}
&3	&-	&-	&\gr{2/6}&4	&1/6		&6	&1/10		&10		&1/15	&15		&&\\
\cline{2-14}
&4	&-	&-	&-		&-	&\gr{8/18}	&4	&1/4		&4		&\gr{8/50}		&7		&1/9		&9\\
\cline{2-14}
&5	&-	&-	&-		&-	&-			&-	&7/14		&4		&3/10			&4		&2/10		&6\\
\cline{2-14}
&6	&-	&-	&-		&-	&-			&-	&-			&-		&\gr{9/16}		&4		&9/25		&4\\
\hline\multirow{4}{*}{3}	%%%%%%%%%%%%%%%%%%%%%%%%% k =3
&3	&-	&-	&\gr{1/8}	&\gr{8}&\gr{1/25}&\gr{25}	&\gr{1/56}&\gr{56}&\gr{1/105}	&\gr{105}
	&\gr{1/176}	&\gr{176}\\
\cline{2-14}
&4	&-	&-	&-		&-	&\gr{3/15}	&8	&\gr{4/54}	&15		&				&		&\gr{3/150}	&\\ 
\cline{2-14}
&5	&-	&-	&-		&-	&-			&-	&\gr{6/24}	&8		&\gr{10/98}		&12		&\gr{5/96}	&\\ 
\cline{2-14}
&6	&-	&-	&-		&-	&-			&-	&-			&-		&8/28 			&8		&1/8		&8\\ 
\hline
\end{array}$
\end{tabular}
\end{center}}
\end{table}

%% NB Gamma(6, 5, 3) min R :: par Delta(6, 5, 3) min R =8 *et* prolong° sol° Gamma(5, 4, 3) min R

In particular, we are interested in minimizing the number of rows in the arrays, an objective for which regular designs may be suboptimal.
\Cref{tab-gamma} gives the ratio of $R^*(Q, P)$ to $R(Q, P)$ in optimal ARPAs of $\Gamma(q, p, k)$ using a minimum number of rows, and also the minimum number of rows in ARPAs of $\Gamma(q, p, k)$, for some sets $(q, p, k)$ of parameters. 
For instance, the optimal ARPAs for the case $p =k$ also minimize the number of rows in the arrays, since they maximize $\gamma(q, k, k)$ with a single row of the form $0\ 1\ \cdots\ q-1$.
In contrast, in \cref{sec-opt}, we constructed an ARPA that realizes $\gamma(5, 3, 2)$ using arrays of 18 rows (see \cref{fig-gamma_q_p_2}), while \cref{fig-Gamma-ex} shows an ARPA that also realizes $\gamma(5, 3, 2)$ with arrays of only 6 rows. Similarly, for $\gamma(q, p, 1)$, we considered arrays of $\binom{q}{p}$ rows while we can trivially construct $(q, p)$-ARPAs $(Q, P)$ of strength 1 with $R^*(Q, P) = 1$ and $R(Q, P) = \lceil q/p\rceil$.%
\footnote{
For each $r\in\set{1 ,\ldots, \lceil q/p\rceil}$, insert into $P$ the row $P_r$ defined for $j\in\Sigma_q$ by $P_r^j =j$ if $(r-1)p\leq j < rp$ and $(r-1)p$ otherwise. 
Insert the row $0\ 1\ \cdots\ q-1$ into $Q$. Then complete $Q$ to satisfy $(k_=)$.}%fin_footnote
Thus, the search for ARPAs or CPAs with as few rows as possible is a natural continuation of this work. In particular, the question arises: does the equivalence between optimal ARPAs and CPAs hold when the optimization criterion is to minimize the number of rows?  

% ____________ tableaux simples
We are also interested in cases where rows cannot be repeated in the arrays, for which CPAs with certain parameters may not even exist. For example, the relation \cref{eq-qkk-MN-rec} of rank $h =0$ implies for any two integers $k > 0$ and $\nu\geq k + 2$ that in every $(\nu, k)$-CPA of strength $k$, at least one row is repeated at least once.% 
\footnote{We successively observe:
$\max\set{b_0, a_0}
	\geq\abs{a_0-b_0}
	\geq\sum_{i=k+1}^\nu\binom{i}{0}\binom{i-1-0}{k-0} b_i
	\geq\binom{\nu-1}{k} b_\nu
	\geq 2$.
The second inequality follows from the equality \cref{eq-qkk-MN-rec} of rank 0. The last inequality holds because $b_\nu\geq 1$ (by $(\Delta_N)$) and $\binom{\nu-1}{k}\geq 2$ (by assumption $\nu-1 > k > 0$). Therefore, $N$ or $D$ must contain at least two rows of zeros.}%fin_footnote 
In contrast, we believe (but this remains to be proved) that $\gamma(q, k, k)$ is always achievable by ARPAs without repeated rows. 

% ____________ relaxed ARPAs
\begin{figure}[t]{\scriptsize
\caption{Relaxed $(q, p)$-ARPAs of strength 2 and 3. 
We highlight in gray the rows of the form $a\ a+1\ \cdots\ a+q-1$ in the array $Q$.}%fin_caption
\label{fig-Gamma_E-ex}
\begin{center}
\setlength\arraycolsep{3pt}
\begin{tabular}{c|c}
$\begin{array}{c}  			% _____________ 1è colonne
	(q, p, k) =(3, 2, 2)\\[4pt]	%% (q, p, k) =(3, 2, 2)
	\begin{array}{ccc}\setlength\arraycolsep{1.5pt}	
		\begin{array}{ccc}
			Q^0		&Q^1	&Q^2	\\\hline
			0		&0		&0		\\
			\gr{0}	&\gr{1}	&\gr{2}	\\
			0		&2		&1		\\
		\end{array}&&\setlength\arraycolsep{1.5pt} 
		\begin{array}{ccc}
			P^0		&P^1	&P^2	\\\hline
			0		&0		&1		\\
			0		&1		&0		\\
			0		&2		&2		\\
		\end{array}
	\end{array}\\[-6pt]\\ 
	R^*(Q, P)/R(Q, P) =1/3\\[-4pt]\\\hline\\[-6pt]
	(q, p, k) =(5, 3, 2)\\[4pt]	%% (q, p, k) =(5, 3, 2)
	\begin{array}{ccc}\setlength\arraycolsep{1.5pt}	
		\begin{array}{ccccc}
			Q^0		&Q^1	&Q^2	&Q^3	&Q^4	\\\hline
			0		&0		&0		&0		&0		\\
			0		&0		&3		&0		&2		\\
			\gr{0}	&\gr{1}	&\gr{2}	&\gr{3}	&\gr{4}	\\
			\gr{0}	&\gr{1}	&\gr{2}	&\gr{3}	&\gr{4}	\\
			\gr{0}	&\gr{1}	&\gr{2}	&\gr{3}	&\gr{4}	\\
			\gr{0}	&\gr{1}	&\gr{2}	&\gr{3}	&\gr{4}	\\
			0		&2		&1		&1		&1		\\
			0		&2		&4		&1		&3		\\
			0		&3		&0		&4		&3		\\
			0		&4		&4		&2		&0		\\
		\end{array}&&\setlength\arraycolsep{1.5pt} 
		\begin{array}{ccccc}
			P^0		&P^1	&P^2	&P^3	&P^4\\\hline
			0		&0		&2		&2		&4	\\
			0		&0		&2		&3		&3	\\
			0		&1		&0		&1		&4	\\
			0		&1		&1		&3		&3	\\
			0		&1		&2		&1		&2	\\
			0		&1		&4		&0		&1	\\
			0		&2		&2		&4		&0	\\
			0		&2		&3		&3		&0	\\
			0		&3		&4		&0		&4	\\
			0		&4		&0		&3		&4	\\
		\end{array}
	\end{array}\\[-6pt]\\ 
	R^*(Q, P)/R(Q, P) =4/10 =2/5
\end{array}$&$\begin{array}{c} 			% _____________ 2è colonne
	(q, p, k) =(5, 4, 3)\\[4pt]
	\begin{array}{ccc}\setlength\arraycolsep{1.5pt}	
		\begin{array}{ccccc}
			Q^0	&Q^1&Q^2&Q^3&Q^4\\\hline
			0	&0	&1	&3	&4	\\
			0	&0	&2	&2	&4	\\
			0	&0	&2	&3	&3	\\
			0	&1	&1	&2	&3	\\
			\gr{0}&\gr{1}&\gr{2}&\gr{3}&\gr{4}\\
			\gr{0}&\gr{1}&\gr{2}&\gr{3}&\gr{4}\\
			\gr{0}&\gr{1}&\gr{2}&\gr{3}&\gr{4}\\
			\gr{0}&\gr{1}&\gr{2}&\gr{3}&\gr{4}\\
			0	&1	&3	&4	&0	\\
			0	&2	&2	&3	&0	\\
			0	&2	&2	&4	&4	\\
			0	&2	&3	&3	&4	\\
		\end{array}&&\setlength\arraycolsep{1.5pt} 
		\begin{array}{ccccc}
			P^0	&P^1&P^2&P^3&P^4\\\hline
			0	&0	&1	&2	&3\\
			0	&0	&2	&3	&4\\
			0	&0	&2	&3	&4\\
			0	&1	&1	&3	&4\\
			0	&1	&2	&2	&4\\
			0	&1	&2	&3	&0\\
			0	&1	&2	&3	&3\\
			0	&1	&2	&4	&4\\
			0	&1	&3	&3	&4\\
			0	&2	&2	&3	&4\\
			0	&2	&2	&3	&4\\
			0	&2	&3	&4	&0\\
		\end{array}
	\end{array}\\[-6pt]\\
	R^*(Q, P)/R(Q, P) =4/12 =1/3
\end{array}$ 
\end{tabular}
\end{center}}
\end{figure}

\begin{figure}[t]{\scriptsize
\caption{Verification that the upper-left array pair in \cref{fig-Gamma_E-ex} satisfies the $(k_\sim)$ condition.}
\label{fig-Gamma_E-k=}
\begin{center}
\begin{tabular}{c|c|c}
$\begin{array}{cc}				%% ______________ {0, 1}
	\setlength\arraycolsep{1.5pt}	
	\begin{array}{cc}
		Q^0		&Q^1	\\\hline
		0		&0		\\
		\gr{0}	&\gr{1}	\\
		0		&2		\\
	\end{array}&\setlength\arraycolsep{1.5pt} 
	\begin{array}{cc}
		P^0		&P^1	\\\hline
		0		&0		\\
		0		&1		\\
		0		&2		\\
	\end{array}
\end{array}$&$\begin{array}{cc}	%% ______________ {0, 2}
	\setlength\arraycolsep{1.5pt}	
	\begin{array}{cc}
		Q^0		&Q^2	\\\hline
		0		&0		\\
		\gr{0}	&\gr{2}	\\
		0		&1		\\
	\end{array}&\setlength\arraycolsep{1.5pt} 
	\begin{array}{cc}
		P^0		&P^2	\\\hline
		0		&1		\\
		0		&0		\\
		0		&2		\\
	\end{array}
\end{array}$&$\begin{array}{cc}	%% ______________ {0, 3}
	\setlength\arraycolsep{1.5pt}	
	\begin{array}{cc}
		Q^1		&Q^2	\\\hline
		0		&0		\\
		\gr{1}	&\gr{2}	\\
		2		&1		\\
	\end{array}&\setlength\arraycolsep{1.5pt} 
	\begin{array}{cc}
		P^1		&P^2	\\\hline
		0		&1		\\
		1		&0		\\
		2		&2		\\
	\end{array}
\end{array}$
\end{tabular}
\end{center}}
\end{figure}

\newcommand{\apx}[1]{}		%% valeur exacte?
\newcommand{\pvo}[1]{#1}	%% rapport pv opt 
\begin{table}[t]{\footnotesize
\caption{Values of $\gamma_E(q, p, k)$ for some triples $(q, p, k)$. These values can be computed by solving linear programs with continuous and integer variables \cite{CT26-reg}.}
\label{tab-gamma_E}
\begin{center}
\begin{tabular}{cc}\setlength\arraycolsep{6.5pt}
$\begin{array}{|c|c|cccccc|}
\multicolumn{2}{c|}{} &\multicolumn{6}{c|}{q}\\\cline{3-8}
k&p	&3		&4		&5 		&6						&7						&8		\\
\hline\multirow{6}{*}{2}
&2	&1/3	&1/4	&1/5	&9/59					&1/7					&\pvo{1/8}	\\
&3	&-		&1/2	&2/5	&4/13					&2/7					&93/404	\\
&4 	&-		&-		&3/5	&7/15					&3/7					&3/8	\\
&5 	&-		&-		&-		&2/3					&11/21					&13/28	\\
&6 	&-		&-		&-		&-						&5/7					&4/7	\\
&7 	&-		&-		&-		&-						&-						&3/4	\\
\hline\multirow{5}{*}{3}
&3	&-		&1/4	&1/11	&38425/701342			&\apx{3676/107221}		&		\\
&4	&-		&-		&1/3	&1/6					&\pvo{5/52}				&		\\
&5	&-		&-		&-		&4/9					&\pvo{2/9}				&		\\
&6	&-		&-		&-		&-						&1/2					&		\\
&7	&-		&-		&-		&-						&-						&9/16	\\
\hline
\end{array}$
&$\begin{array}{|c|c|ccc|}	\multicolumn{5}{c}{}\\
\multicolumn{2}{c|}{} &\multicolumn{3}{c|}{q}\\\cline{3-5}
k&p	&5 		&6					&7						\\%	&8		\\
\hline\multirow{2}{*}{4} %% gamma_E(6, 4, 4) :: 0.03159029059 ~557/17632 
&4	&1/11	&\apx{557/17632} 	&\apx{0.013964734}		\\%	&		\\
&5	&-		&1/6				&\apx{0.058898}			\\%	&		\\
%&6	&-		&-						&\apx{0.2088948787}		&		\\
\hline\multirow{2}{*}{5} %% gamma_E(7, 6, 5) :: solution 6/60
&5	&-		&1/16				&\apx{0.01281777623}	\\%	&		\\ 
&6	&-		&-					&1/10					\\%	&		\\
\hline
\end{array}$
\end{tabular}
\end{center}}
\end{table}

\smallskip
Most importantly, we have started to investigate the slight relaxation of ARPAs where any two words $(w_1, w_2, \ldots, w_\ell)$ and $(w_1+a, w_2+a, \ldots, w_\ell+a)$ are regarded as equivalent (see \cite{CT18}). 
This means that $Q$ does not have to contain the exact row $0\ 1\ \cdots\ q-1$, but only some row of the form $a\ a+1\ \cdots\ a+q-1$. This also means that instead of $(k_=)$ the arrays $Q$ and $P$ must satisfy the following condition $(k_\sim)$: for all subsets $J$ of $k$ column indices and all $w\in\Sigma_q^k$, the number of rows $P_r$ of $P$ satisfying $P_r^J\in\set{w, w+(1, 1 ,\ldots, 1) ,\ldots, w+(q-1, q-1 ,\ldots, q-1)}$ must be the same as in $Q$.
\Cref{fig-Gamma_E-ex} shows examples. In particular, the pair of arrays given for the case $(q, p, k) = (3, 2, 2)$ satisfies $R(Q, P) = 3$ and $R^*(Q, P) = 1$, and we can see in \Cref{fig-Gamma_E-k=} that it verifies $(k_\sim)$.

\Cref{tab-gamma_E} gives some values of the counterpart, denoted by $\gamma_E(q, p, k)$, of $\gamma(q, p, k)$ for these relaxed ARPAs. 
Constructing such designs seems to be more difficult, because pairs of arrays satisfying $(k_\sim)$ over $\Sigma_q$ often fail to meet this condition over $\Sigma_{q+1}$. 
They also make the reduction from $\mathsf{k\,CSP\!-\!q}$ to $\mathsf{k\,CSP\!-\!p}$ more efficient provided that the constraints of the initial instance of $\mathsf{k\,CSP\!-\!q}$ remain stable under a uniform shift of all the variables. 
For example, an equation $(x_{i_1} + x_{i_2} +\cdots+ x_{i_k}\equiv a_i\bmod{q})$ where $k$ is a multiple of $q$, or a requirement that $k$ literals $x_{i_1}+a_{i_1}, x_{i_2}+a_{i_2} ,\ldots, x_{i_k}+a_{i_k}$ are all equal modulo $q$, are such constraints.
When $k=2$ and $q\in\set{3, 4, 5, 7, 8}$, we obtain a differential approximability bound of $(2-\pi/2)/q$---instead of $(2-\pi/2)/(q-1)^2$---for such instances of $\mathsf{2\,CSP\!-\!q}$. 
Identifying general constructs for these relaxed ARPAs not only seems quite challenging, but will also improve our knowledge of how well $\mathsf{k\,CSP\!-\!q}$ reduces to $\mathsf{k\,CSP\!-\!p}$. 

% ______________ Bilio
\bibliographystyle{plain}
\bibliography{gamma}

@inproceedings{CT18,
	author="Culus, Jean-Fran{\c{c}}ois and Toulouse, Sophie",
	title="2 CSPs All Are Approximable Within a Constant Differential Factor",
	booktitle="Combinatorial Optimization",
	year="2018",
	editor="Lee, Jon and Rinaldi, Giovanni and Mahjoub, A. Ridha",
	publisher="Springer International Publishing",
	address="Cham",
	pages="389--401",
	series    = {Lecture Notes in Computer Science},
	volume    = {10856},
	doi       = {10.1007/978-3-319-96151-4\_33}
}

@article{CT26-E,
	author = {Jean-François Culus and Sophie Toulouse},
	title = {From worst case to the average: Structural guarantees in k-CSP approximation via orthogonal arrays},
	journal = {Theor. Comput. Sci.},
	volume = {1076},
	pages = {115896},
	year = {2026},
	issn = {0304-3975},
	doi = {https://doi.org/10.1016/j.tcs.2026.115896},
}

@article{BR95,
	author={Mihir Bellare and Phillip Rogaway},
	title ={The complexity of approximating a nonlinear program},
	journal={Mathematical Programming},
	pages={429--441},
	volume={69},
	year={1995},
	doi ={10.1007/BF01585569}
}

@article{N98,
	author = {Yuri Nesterov},
	title = {Semidefinite relaxation and nonconvex quadratic optimization},
	journal = {Optimization Methods and Software},
	volume = {9},
	number = {1-3},
	pages = {141--160},
	year = {1998},
	publisher = {Taylor \& Francis},
	doi = {10.1080/10556789808805690},
}

@article{EP05,
	title = {Differential approximation of min sat, max sat and related problems},
	journal = {European J. of Operational Research},
	volume = {181},
	number = {2},
	pages = {620-633},
	year = {2007},
	issn = {0377-2217},
	doi = {https://doi.org/10.1016/j.ejor.2005.04.057},
	author = {Bruno Escoffier and Vangelis Th. Paschos},
}

@InCollection{MM17,
	author =	{Makarychev, Konstantin and Makarychev, Yury},
	title =	{{Approximation Algorithms for CSPs}},
	booktitle =	{The Constraint Satisfaction Problem: Complexity and Approximability},
	pages =	{287--325},
	series =	{Dagstuhl Follow-Ups},
	ISBN =	{978-3-95977-003-3},
	ISSN =	{1868-8977},
	year =	{2017},
	volume =	{7},
	editor =	{Krokhin, Andrei and Zivny, Stanislav},
	publisher =	{Schloss Dagstuhl -- Leibniz-Zentrum f{\"u}r Informatik},
	address =	{Dagstuhl, Germany},
	URN =		{urn:nbn:de:0030-drops-69685},
	doi =		{10.4230/DFU.Vol7.15301.287}
}

@article{ABMV77,
	author = "A. Aiello and E. Burattini and A. Massaroti and F. Ventriglia",
	title = "Towards a general principle of evaluation for approximate algorithms",
	journal = "RAIRO -- Theoretical Informatics and Applications",
	volume = "13",
	number = "3",
	year = "1979",
	pages = "227--239",
	doi = "10.1051/ita/1979130302271"
}

@article{AAP80,
	author = "Giorgio Ausiello and A. D'Atri and Marco Protasi",
	title = "Structure preserving reductions among convex optimization problems",
	journal = "J. of Computational System Sciences",
	volume = "21",
	number = "1",
	pages = "136--153",
	year = "1980",
	doi = "10.1016/0022-0000(80)90046-X"
}

@article{DP96,
	title = {On an approximation measure founded on the links between optimization and polynomial approximation theory},
	journal = {Theor. Comput. Sci.},
	volume = {158},
	number = {1},
	pages = {117-141},
	year = {1996},
	issn = {0304-3975},
	doi = {https://doi.org/10.1016/0304-3975(95)00060-7},
	author = {Marc Demange and Vangelis Th. Paschos},
}

@article{CT26-reg,
	title={{Optimizing Alphabet Reduction Pairs of Arrays}}, 
	author={Jean-François Culus and Sophie Toulouse},
	journal={Math. in Comput. Sci.},
	volume={20},
	number={10},
	year={2026},
	issn={1661-8289},
	doi={10.1007/s11786-026-00626-8},
}

@inproceedings{T26,
	title = {Constructing alphabet reduction pairs of arrays},
	author={Sophie Toulouse},
	year={2026},
	booktitle="LATIN 2026",
	publisher="Springer International Publishing",
	series    = {Lecture Notes in Computer Science},
	note = "To appear"
}

% ____________ Preuves
\begin{appendices}
\crefalias{section}{appsec}

\section{Extended proof of \cref{thm-delta_d=k-UB}}
\label{appendix:sec-UB}

We show that the left-hand side of \cref{eq-Delta_h_i}, 
$$\textstyle\sum_{\ell=h}^k(-1)^{k-\ell}\binom{\nu-\ell}{k-h}\binom{i-h}{\ell-h}\binom{i-1-\ell}{k-\ell}
	-\binom{\nu-i}{k-h},$$
is equal to
$$(-1)^{k-h+1}\frac{(i-h)!}{(i-1-k)!(k-h)!}\times f(A,B),$$
where $A$ and $B$ are defined by 
$$\begin{array}{rl}
	A 	&:= \set{\nu-k+h+1 ,\ldots, \nu},	\\
	B 	&:= \set{h ,\ldots, k}\cup\set{i}	
\end{array}$$
and $f(A,B)$ is defined by \cref{eq-fAB}.

According to \cref{eq-fAB}, we have:
$$\begin{array}{rl}
	f(A,B) 	&=\sum_{\ell=h}^k \displaystyle
				\frac{\prod_{a=\nu-k+h+1}^\nu(\ell-a)}
							{\prod_{b=h}^{\ell-1}(\ell-b)\times\prod_{\ell+1}^k(\ell-b)\times(\ell-i)}
				+\frac{\prod_{a=\nu-k+h+1}^\nu(i-a)}{\prod_{b=h}^k(i-b)}.	
\end{array}$$
 
We make the following simple observations:
\begin{enumerate}
	\item For $b\in B$, $\binom{\nu-b}{k-h}$ equals $0$ if $b\in A$;
		otherwise, this coefficient is $1/(k-h)!$ times the quantity 
				$\prod_{a\in A}(a-b) = (-1)^{k-h}\prod_{a\in A}(b-a)$.
	\item For $\ell\in\set{h ,\ldots, k}$, $\binom{i-h}{\ell-h}\binom{i-1-\ell}{k-\ell}$ 
			is $(i-h)!/(i-1-k)!$ times the quantity
			$\displaystyle\frac{1}{(i-\ell)(\ell-h)!(k-\ell)!}$, 
		while $(\ell-h)!=\prod_{b=h}^{\ell-1}(\ell-b)$ 
			and $(k-\ell)!=\prod_{\ell+1}^k(b-\ell)=(-1)^{k-\ell}\prod_{\ell+1}^k(\ell-b)$.
	\item The quantity $(i-h)!/(i-k-1)!$ corresponds exactly to $\prod_{b=h}^k(i-b)$.
\end{enumerate}

The conclusion follows directly.

% _______________________________________________________________________ PV (k=) Thm joker
\section{Extended proof of \cref{fact-algo-k=}}
\label{appendix:sec-fact_k=}

First, we prove \cref{lem-f_z}. Then, we establish the expressions in \cref{tab-fact6-valf} which, together with \cref{lem-f_z}, allow us to establish \cref{fact-algo-k=}.

\subsection{Detailed proof of \cref{lem-f_z}}

%% NB z_\nu <0 :: pour que r soit bien défini
Let $k > 0$ and $\nu > k$ be two integers, and let $z = (z_0, z_1 ,\ldots, z_\nu)$ be a sequence of integers satisfying \cref{delta_z-equ} with $z_\nu < 0$. 
We denote by $r$ the largest integer in $\set{k, k+1 ,\ldots, \nu-1}$ such that  $z_r\neq 0$ (note that the rank-$k$ equality \cref{delta_z-equ}, together with the assumption $z_\nu<0$, implies that such an integer exists). 
For natural numbers $h,\lambda,c$ satisfying $h\leq r$, $c\leq\nu-r$, and $\lambda\leq\nu-c-h$, we define:
\begin{align}\nonumber
f_z(h, \lambda, c)	&:=	\textstyle
	\sum_{i = h}^r
		\binom{\nu-c-h-\lambda}{i-h}\binom{\nu-c-i}{r-i}/\binom{\nu-1-i}{r-i} 
		\times z_i.
\end{align}
Our goal is to prove relation \cref{eq-f=0}, which states that these quantities vanish whenever $h+\lambda <k$.

First, assume $\lambda > 0$. Pascal's rule implies the following equalities:
$$\begin{array}{rll}
\binom{\nu-c-h-\lambda}{i-h} 
&=\binom{\nu-c-h-\lambda+1}{i-h}-\binom{\nu-c-h-\lambda}{i-h-1},
&i\in\set{h ,\ldots, r}
\end{array}.$$
Hence, $f_z(h, \lambda, c)$ can be rewritten as the difference $f_z(h, \lambda-1, c)-f_z(h+1, \lambda-1, c)$. 
Now suppose that $\lambda = 0$. Given the equality
$$\begin{array}{rll}
\binom{\nu-c-h}{i-h} \binom{\nu-c-i}{r-i} / \binom{\nu-1-i}{r-i}
	&=\binom{\nu-c-h}{r-h} \binom{\nu-1-h}{i-h} / \binom{\nu-1-h}{r-h},
	&i\in\set{h ,\ldots, r},
\end{array}$$
$f_z(h, 0, c)$ can be expressed as: 
\begin{align}\label{eq-f_z-0}
f_z(h, 0, c)	
	&\textstyle=\binom{\nu-c-h}{r-h}/\binom{\nu-1-h}{r-h} 
				\times \sum_{i = h}^r \binom{\nu-1-h}{i-h} z_i.
\end{align}
Assuming $k-h > 0$, we can write:
\begin{align}
\textstyle\sum_{i = h}^r \binom{\nu-1-h}{i-h} z_i
&\textstyle=\sum_{i = h}^r \sum_{j = 0}^{k-h-1} \binom{k-h-1}{j} \binom{\nu-k}{i-h-j} z_i
	\qquad\text{by Vandermonde's identity}\nonumber\\\nonumber
&\textstyle=\sum_{j = 0}^{k-h-1} \binom{k-h-1}{j} \sum_{i = h}^r \binom{\nu-k}{i-h-j} z_i\\
&\textstyle\label{eq-z_0}=\sum_{j = 0}^{k-h-1} \binom{k-h-1}{j} 
	\sum_{i = h+j}^{\nu-k+h+j} \binom{\nu-k}{i-h-j} z_i.
\end{align}
The last equality holds because for all $j\in\set{0 ,\ldots, k-h-1}$, we have the inequality $\nu-k+h+j\leq\nu-1$, while either $r\leq\nu-k+h+j$ and $z_i = 0$ for all $i\in\set{r+1, r+2 ,\ldots, \nu-k+h+j}$, or $r>\nu-k+h+j$ and $\binom{\nu-k}{i-h-j} = 0$ for all $i\in\set{\nu-k+h+j+1 ,\ldots, r-1, r}$.

The relations \cref{delta_z-equ}, \cref{eq-z_0} and \cref{eq-f_z-0} together imply that $f_z(h, 0, c)=0$, provided that $h < k$. 
Hence, for triples of parameters $(h, \lambda, c)$ such that $h+\lambda <k$ and $c\leq\nu-r$, the quantities $f_z(h, \lambda, c)$ satisfy the recurrence relation and the initial condition \cref{eq-f-rec}, namely:
$$\begin{array}{rl}
f_z(h, \lambda, c)
&=\left\{\begin{array}{rl}
	f_z(h, \lambda-1, c)-f_z(h+1, \lambda-1, c)	&\text{if $\lambda > 0$},\\
	0														&\text{otherwise}.
\end{array}\right.
\end{array}$$
From \cref{eq-f-rec}, we deduce that, for such parameters, $f_z(h, \lambda, c)$ is a weighted sum of terms of the form $f_z(h', 0, c)$ with $h'\leq h'+\lambda <k$. 
Since these terms vanish, relation \cref{eq-f=0} follows.

\subsection{Expressions in \cref{tab-fact6-valf}}
We justify the expressions for $R(H, L, v)$ given in \cref{tab-fact6-valf}.

\medskip
$\bullet$ Case \ref{g-c_in_H+g^{c-1}}. $R(H, L, v)$ is the sum of the expressions \cref{g(J)-c_in_H} with $\lambda=0$ and \cref{g^{c-1}(J)}. 
Accordingly,
$$\begin{array}{rl}
R(H, L, v)	
&=\binom{\nu-c-h}{r-h} z_r
	+\sum_{i = h}^{r-1}\left(\binom{\nu-c-1-i}{r-i}+\binom{\nu-c-1-i}{r-1-i}\right)
	\times\binom{\nu-c-h}{i-h} /	\binom{\nu-1-i}{r-i} \times z_i
\end{array}.$$
For $i\in\set{h ,\ldots, r-1}$, we have 
	$\binom{\nu-c-1-i}{r-i}+\binom{\nu-c-1-i}{r-1-i} = \binom{\nu-c-i}{r-i}$. 
For $i = r$, we observe that the coefficient of $z_r$ in $f_z(h, 0, c)$ is   $\binom{\nu-c-h}{r-h}$. We conclude that $R(H, L, v) = f_z(h, 0, c)$.

$\bullet$ Case \ref{g-c_in_H}. $R(H, L, v)$ is the expression \cref{g(J)-c_in_H}, corresponding exactly to $f_z(h, \lambda-1, c+1)$, except for the coefficients associated with $z_r$, namely  $\binom{\nu-c-h-\lambda}{r-h}$ in \cref{g(J)-c_in_H} and $\binom{\nu-c-h-\lambda}{r-h}\times\binom{\nu-c-1-r}{0}/\binom{\nu-1-r}{0}$ in $f_z(h, \lambda-1, c+1)$. 
However, these coefficients coincide provided that $c<\nu-r$. 

$\bullet$ Case \ref{g-c_notin_H}. $R(H, L, v)$ is the expression \cref{g(J)-c_notin_H}, which corresponds exactly to \cref{g(J)-c_in_H} taken at $h+1$.

$\bullet$ Case \ref{g^c}. $R(H, L, v)$ is the expression \cref{g^c(J)} with a positive value for the parameter $\lambda$. 
For $i\in\set{h ,\ldots, r-1}$, the coefficients of $z_i$ in \cref{g^c(J)}, $f_z(h, \lambda, c+1)$ and $f_z(h, \lambda-1, c+2)$ respectively are
%%$\binom{\nu-c-1-h-\lambda}{i-h}/\binom{\nu-1-i}{r-i}$ times
$$\begin{array}{l}
\frac{\binom{\nu-c-1-h-\lambda}{i-h}}{\binom{\nu-1-i}{r-i}}\times
	\binom{\nu-c-2-i}{r-1-i},	\	
\frac{\binom{\nu-c-1-h-\lambda}{i-h}}{\binom{\nu-1-i}{r-i}}\times
	\binom{\nu-c-1-i}{r-i},		\text{ and }
\frac{\binom{\nu-c-1-h-\lambda}{i-h}}{\binom{\nu-1-i}{r-i}}\times
	\binom{\nu-c-2-i}{r-i}.
\end{array}$$

The binomial coefficient $\binom{\nu-c-2-i}{r-1-i}$ is equal to $\binom{\nu-c-1-i}{r-i}-\binom{\nu-c-2-i}{r-i}$, which reduces to $\binom{\nu-c-1-i}{r-i}$ when $c = \nu-r-1$.
For $z_r$, its coefficient in \cref{g^c(J)} is 0, while its coefficient in $f_z(h, \lambda, c+1)$ is equal to $\binom{\nu-c-1-h-\lambda}{r-h}$. Note that the latter coefficient is zero when $c = \nu-r-1$ and $\lambda > 0$, and coincides with the coefficient of $z_r$ in $f_z(h, \lambda-1, c+2)$ when $c <\nu-r-1$.
We conclude that for $i\in\set{h ,\ldots, r}$, the coefficient of $z_i$ in $R(H, L, v)$ is the same as that in $f_z(h, \lambda, c+1)$ if $c = \nu-r-1$, and in $f_z(h, \lambda, c+1)-f_z(h, \lambda-1, c+2)$ otherwise.

$\bullet$ Case \ref{g-c_notin_H+g^{c-1}}. $R(H, L, v)$ is the sum of the expression \cref{g(J)-c_notin_H} with $\lambda=0$ and \cref{g^{c-1}(J)}. 
Accordingly,
$$\begin{array}{rl}
R(H, L, v)
&=\binom{\nu-c-1-h}{r-1-h} z_r 
	+\binom{\nu-c-1-h}{r-1-h}/\binom{\nu-1-h}{r-h} \times z_h\\
	&\qquad+\sum_{i = h+1}^{r-1}\left(
		 \binom{\nu-c-1-h}{i-1-h}\binom{\nu-c-1-i}{r-i} 
		+\binom{\nu-c-h}{i-h}\binom{\nu-c-1-i}{r-1-i}
	\right)/\binom{\nu-1-i}{r-i} \times z_i.
\end{array}$$

If $c = \nu-r$, then $\cref{g(J)-c_notin_H}$ reduces to $z_r$, while \cref{g^{c-1}(J)} coincides with $f_z(h, 0, c)-z_r$.
So we assume $c <\nu-r$. 
For $i\in\set{h+1 ,\ldots, r-1}$, we write:
$$\begin{array}{l}
\binom{\nu-c-1-h}{i-1-h}\binom{\nu-c-1-i}{r-i} 
	+\binom{\nu-c-h}{i-h}\binom{\nu-c-1-i}{r-1-i}\\
\qquad=\binom{\nu-c-h}{i-h}\times\left(
		\binom{\nu-c-1-i}{r-i}+\binom{\nu-c-1-i}{r-1-i}
	\right)-\binom{\nu-c-1-h}{i-h}\binom{\nu-c-1-i}{r-i}\\
\qquad=\binom{\nu-c-h}{i-h}\binom{\nu-c-i}{r-i} 
	-\binom{\nu-c-1-h}{i-h}\binom{\nu-c-1-i}{r-i}.
\end{array}$$

The coefficient of $z_i$ in $R(H, L, v)$ is therefore the same as in $f_z(h, 0, c)-f_z(h, 0, c+1)$. 
For $i\in\set{h, r}$, we similarly observe that the coefficients of $z_i$ in $f_z(h, 0, c)-f_z(h, 0, c+1)$ and in $R(H, L, v)$ are equal. 

$\bullet$ Case \ref{g+g^c}. $R(H, L, v)$ is the sum of \cref{g(J)-c_notin_H,g^c(J)} with $\lambda=k-h$. Accordingly,
$$\begin{array}{rl}
R(H, L, v)
&=\binom{\nu-c-1-k}{r-1-h} z_r 
	+\binom{\nu-c-2-h}{r-1-h}/\binom{\nu-1-h}{r-h} \times z_h\\
	&\qquad+\sum_{i = h+1}^{r-1}\left(
		 \binom{\nu-c-1-k}{i-1-h}\binom{\nu-c-1-i}{r-i} 
		+\binom{\nu-c-1-k}{i-h}\binom{\nu-c-2-i}{r-1-i}
	\right)/\binom{\nu-1-i}{r-i} \times z_i.
\end{array}$$

For $i\in\set{h+1 ,\ldots, r-1}$, we write:
$$\begin{array}{l}
\binom{\nu-c-1-k}{i-1-h}\binom{\nu-c-1-i}{r-i} 
		+\binom{\nu-c-1-k}{i-h}\binom{\nu-c-2-i}{r-1-i}\\
\qquad=\left(\binom{\nu-c-1-k}{i-1-h}+\binom{\nu-c-1-k}{i-h}\right)
	\times\binom{\nu-c-1-i}{r-i}-\binom{\nu-c-1-k}{i-h}\binom{\nu-c-2-i}{r-i}\\
\qquad=\binom{\nu-c-k}{i-h}\binom{\nu-c-1-i}{r-i} 
	-\binom{\nu-c-1-k}{i-h}\binom{\nu-c-2-i}{r-i}.
\end{array}$$
We deduce that the coefficient of $z_i$ in $R(H, L, v)$ is the same as in $f_z(h, k-h-1, c+1)$ if $c = \nu-r-1$, as in $f_z(h, k-h-1, c+1)-f_z(h, k-h-1, c+2)$ otherwise.

For $i\in\set{h, r}$, if $c = \nu-r-1$, we observe that the coefficients of $z_r$ in $f_z(h, k-h-1, \nu-r)$ and $R(H, L, v)$ are the binomial coefficients $\binom{r+1-k}{r-h}$ and $\binom{r-k}{r-1-h}$, which equal 1 when $h = k-1$, and 0 when $h < k-1$.  
Moreover, the coefficients of $z_h$ in $R(H, L, v)$ and $f_z(h, k-h-1, \nu-r)$ are both equal to $1/\binom{\nu-1-r}{r-h}$.
Otherwise, $c <\nu-r-1$, and we observe that, for both $z_r$ and $z_h$, the difference between their coefficients in $f_z(h, k-h-1, c+1)$ and $f_z(h, k-h-1, c+2)$ coincides with their respective coefficients in $R(H, L, v)$.

% _______________________________________________________________________ Proposition PL SB 
\section{Extended proof of \cref{prop-Delta_PL-SB} (proof of identity \cref{eq-SB-tech})}
\label{appendix:sec-PLSB}

For two natural numbers $\nu,k$ with $\nu > k > 0$, and a $(k+1)$-element subset $Z$ of $\set{0, 1 ,\ldots, \nu}$ such that $\nu\in Z$, we consider the vector $z = (z_0, z_1 ,\ldots, z_\nu)\in\mathbb{Z}^{\nu+1}$ defined componentwise by \cref{eq-SB_z}, i.e.:
	$$\begin{array}{rll}
	z_i	&=\left\{\begin{array}{rll}
			\prod_{a\in Z\setminus\set{i,\nu}}\frac{\nu-a}{i-a}/\binom{\nu}{i}	
				&\text{if $i\in Z\setminus\set{\nu}$}\\
			-1 	&\text{if $i = \nu$}\\
			0	&\text{otherwise}
		\end{array}\right.,	&i\in\set{0, 1 ,\ldots, \nu}
	\end{array}.$$

Our goal is to establish the identity \cref{eq-SB-tech},
$$\begin{array}{rll}
\sum_{i = h}^{\nu-k+h}\binom{\nu-k}{i-h}z_i	
	&=(-1)^{k-h+1}\displaystyle\frac{\prod_{i\in Z\setminus\set{\nu}}(\nu-i)}{\prod_{i = 0}^{k-1}(\nu-i)}
								\times f(A_h, B_h),					&h\in\set{0, 1 ,\ldots, k}
\end{array},$$
where, for $h\in\set{0, 1 ,\ldots, k}$:
\begin{itemize}
\item $A_h=\left(\set{0, 1 ,\ldots, h-1}\cup\set{\nu-k+h+1, \nu-k+h+2 ,\ldots, \nu}\right)\setminus Z$;
\item $B_h=Z\cap\set{h, h+1 ,\ldots, \nu-k+h}$;
\item $f(A_h, B_h)=\sum_{i\in B_h}\displaystyle\frac{\prod_{a\in A_h}(i-a)}{\prod_{b\in B_h\setminus\set{i}}(i-b)}$ (see \cref{lem-fAB}).
\end{itemize}

\medskip
First, suppose $h = k$. 
By \cref{eq-SB_z}, we have:
$$\begin{array}{rl}
\sum_{i = k}^\nu \binom{\nu-k}{i-k} z_i
	&=-1+\sum_{i\in Z\cap\set{k ,\ldots, \nu-1}} 
		\binom{\nu-k}{i-k}/\binom{\nu}{i}
		\times\prod_{a\in Z\setminus\set{i,\nu}}\frac{\nu-a}{i-a}
\end{array}.$$

For $i\in Z\cap\set{k ,\ldots, \nu-1}$, we observe:
$$\begin{array}{l}
\begin{array}{rll}
\binom{\nu-k}{i-k}/\binom{\nu}{i} 
	&=\frac{(\nu-k)!}{\nu!}\times\frac{i!}{(i-k)!}
	&=\prod_{a = 0}^{k-1} \frac{i-a}{\nu-a}
\end{array},\\[4pt]
\begin{array}{rll}
\prod_{a\in Z\setminus\set{i,\nu}}\frac{\nu-a}{i-a} 
	&=\frac{\prod_{a\in Z\setminus\set{\nu}}(\nu-a)}{\prod_{a\in Z\setminus\set{i}}(i-a)}
		\times\frac{\nu-i}{i-\nu}
	&=-\frac{\prod_{a\in Z\setminus\set{\nu}}(\nu-a)}{\prod_{a\in Z\setminus\set{i}}(i-a)}\\[9pt]
\end{array}.
\end{array}$$

Considering that $A_k = \set{0 ,\ldots, k-1}\setminus Z$ and $B_k = Z\cap\set{k ,\ldots, \nu}$, we deduce successively:
$$\begin{array}{rl}
	\sum_{i = k}^\nu \binom{\nu-k}{i-k} z_i
	&=-1-\frac{\prod_{a\in Z\setminus\set{\nu}}(\nu-a)}{\prod_{a = 0}^{k-1}(\nu-a)} \times
		\sum_{i\in Z\cap\set{k ,\ldots, \nu-1}} 
			\frac	{\prod_{a = 0}^{k-1} (i-a)}
					{\prod_{a\in Z\setminus\set{i}}(i-a)}	\\[8pt]
	&=-\frac{\prod_{a\in Z\setminus\set{\nu}}(\nu-a)}{\prod_{a = 0}^{k-1}(\nu-a)} \times
		\sum_{i\in B_k} 
			\frac	{\prod_{a = 0}^{k-1} (i-a)}
					{\prod_{a\in Z\setminus\set{i}}(i-a)}\\[8pt]
	&=-\frac{\prod_{a\in Z\setminus\set{\nu}}(\nu-a)}{\prod_{a = 0}^{k-1}(\nu-a)} \times 
		\sum_{i\in B_k} 
			\frac	{\prod_{a\in A_k} (i-a)}
					{\prod_{b\in B_k\setminus\set{i}}(i-b)}.
\end{array}$$
where the last equality follows from the relations:
$$\begin{array}{rll}
	\set{0 ,\ldots, k-1} 	&=(\set{0 ,\ldots, k-1}\cap Z) \sqcup A_k,  \\
	Z\setminus\set{i} 		&=(\set{0 ,\ldots, k-1}\cap Z) \sqcup (B_k\setminus\set{i}),	&i\in B_k.
\end{array}$$

For $h\in\set{0, 1 ,\ldots, k-1}$, we similarly obtain:
$$\begin{array}{rl}
\sum_{i = h}^{\nu-k+h} \binom{\nu-k}{i-h} z_i
&=\sum_{i\in Z\cap\set{h ,\ldots, \nu-k+h}} \binom{\nu-k}{i-h}/\binom{\nu}{i}
					\times\prod_{a\in Z\setminus\set{i,\nu}}\frac{\nu-a}{i-a}
	\qquad\text{by \cref{eq-SB_z}}\\[8pt]
&=-\sum_{i\in B_h} \frac{(\nu-k)!}{\nu!}
					\times\frac{i!(\nu-i)!}{(i-h)!(\nu-k-i+h)!}
					\times\prod_{a\in Z\setminus\set{i,\nu}}\frac{\nu-a}{i-a}
					\times\frac{\nu-i}{i-\nu}\\[8pt]
&=-\frac{\prod_{a\in Z\setminus\set{\nu}}(\nu-a)}{\prod_{a = 0}^{k-1}(\nu-a)}\times\sum_{i\in B_h}
	\frac	{\prod_{a = 0}^{h-1}(i-a)\times\prod_{a = \nu-k+h+1}^{\nu}(a-i)}
			{\prod_{a\in Z\setminus\set{\nu}}(i-a)}\\[8pt]
&=(-1)^{k-h+2}\times\frac{\prod_{a\in Z\setminus\set{\nu}}(\nu-a)}{\prod_{a = 0}^{k-1}(\nu-a)} 
	\times\sum_{i\in B_h}
	\frac	{\prod_{a\in\set{0 ,\ldots, \nu}\setminus\set{h ,\ldots, \nu-k+h}}(i-a)}
			{\prod_{a\in Z\setminus\set{i}}(i-a)}\\[8pt]
&=(-1)^{k-h}\times\frac{\prod_{a\in Z\setminus\set{\nu}}(\nu-a)}{\prod_{a = 0}^{k-1}(\nu-a)} 
	\times\sum_{i\in B_h}
	\frac	{\prod_{a\in A_h}(i-a)}
			{\prod_{a\in B_h\setminus\set{i}}(i-a)},\\
\end{array}$$
where the last equality follows after introducing the set $C_h:=\set{0 ,\ldots, \nu}\setminus\set{h ,\ldots, \nu-k+h}$ and observing that
$C_h = (C_h \cap Z) \sqcup A_h$ and, for $i\in B_h$, $Z\setminus\set{i} = (C_h \cap Z) \sqcup  (B_h\setminus\set{i})$.

\end{appendices}

\end{document}